\documentclass[12pt,a4paper,reqno]{amsart}
\usepackage{amsmath,amsthm,verbatim,amssymb,amsfonts,amscd, graphicx, esint}
\usepackage{xcolor}
\usepackage{geometry}
\usepackage{graphicx}
\usepackage{hyperref}
\usepackage{enumitem}
\usepackage{cancel}
\usepackage[english,capitalize]{cleveref}
\definecolor{darkblue}{rgb}{0,0,0.3}
\definecolor{urlblue}{rgb}{0,0,0.7}
\hypersetup{
	colorlinks=true,
	citecolor=blue,
	linkcolor=darkblue,
	urlcolor=urlblue,
}

\newtheorem{theorem}{Theorem}[section]
\newtheorem{lemma}[theorem]{Lemma}
\newtheorem{prop}[theorem]{Proposition}
\newtheorem{question}[theorem]{Question}
\newtheorem{conjecture}[theorem]{Conjecture}
\newtheorem{setting}[theorem]{Setting}
\newtheorem{assumptions}[theorem]{Assumption}
\newtheorem{corollary}[theorem]{Corollary}
\newtheorem{definition}[theorem]{Definition}
\newtheorem{remark}[theorem]{Remark}

\renewcommand{\leq}{\leqslant}
\renewcommand{\geq}{\geqslant}

\newcommand{\R}{\mathbb R}

\renewcommand{\d}{\ensuremath{\mathrm d}}

\DeclareMathOperator{\vol}{vol}

\renewcommand{\tilde}{\widetilde}
\renewcommand{\bar}{\overline}

\renewcommand{\epsilon}{\varepsilon}
\newcommand{\m}{\mathfrak m}
\newcommand{\Hn}{\mathcal H^n}
\newcommand{\TestF}{\mathrm{TestF}}
\newcommand{\TestV}{\mathrm{TestV}}
\newcommand{\Hess}{\operatorname{Hess}}
\newcommand{\loc}{\mathrm{loc}}
\newcommand{\supp}{\operatorname{supp}}
\newcommand{\dist}{\operatorname{dist}}
\newcommand{\id}{\operatorname{id}}
\newcommand{\tr}{\operatorname{tr}}

\newcommand{\TODO}[1]{}

\usepackage[maxbibnames=99,backend=biber, sorting=nyt, doi, url=false]{biblatex}
\title[Removability and RCD]{Removability of non-isolated singularities for Einstein metrics and RCD spaces}
\author{Gioacchino Antonelli, Gábor Székelyhidi}

\begin{document}

\begin{abstract}
    In this paper we establish removable singularities results for Einstein metrics and for metrics with Ricci curvature bounded below.

Let $n\geq 2$. On a closed $n$-manifold, we show that an $L^\infty$-Riemannian metric whose Ricci curvature is bounded below outside a singular set of codimension $> 3- \frac{1}{n-1}$ canonically extends to an $\mathrm{RCD}$ space.

As a consequence, using a new removable singularity theorem for Einstein metrics, we prove that in dimension $4$ any Einstein metric with $L^\infty$ singularities of codimension $>3-\frac{1}{3}$ extends smoothly across the singular set, possibly after changing the smooth structure. 
In higher dimensions, we construct a $C^{1,\alpha}$-Riemannian manifold structure on the regular set of a non-collapsed $\mathrm{RCD}$ space that is a Riemannian manifold with bounded $|\mathrm{Ric}|$ outside a set of codimension $>2$. 

Our results can be used to give a proof of Schoen's conjecture on scalar curvature singularities for metrics that are either continuous, or $L^\infty$ and sufficiently close to a smooth background metric.

\end{abstract}
\maketitle

\vspace{12pt}

\tableofcontents

\section{Introduction}
The question of removability of singularities is a classical topic in PDE, and in analysis more broadly. In this paper we consider removable singularities of Einstein metrics, and more generally metrics with Ricci curvature bounded below. Our setting will be a metric space $M$, which has the structure of a smooth Riemannian manifold $(M \setminus S, g)$ outside of a suitable subset $S\subset M$. We will be interested in settings in which various properties of $g$, such as having Ricci curvature bounded below, or being Einstein, can be extended across the set $S$. We recall that a manifold $M$ is said to be \textit{closed} if it is compact without boundary.

We will need to assume some structure across the set $S$. For our first results, we suppose that $M$ itself is a smooth manifold, and $g$ extends across $S$ as an $L^\infty$ metric. More precisely, suppose that $(M^n, g)$ is a closed $n$-dimensional $L^\infty$ Riemannian manifold. This means that $M^n$ admits a smooth Riemannian metric $h$, and $g$ is a measurable symmetric 2-tensor such that for a constant $\Lambda > 1$ we have 
\begin{equation}\label{eqn:LambdaDef}
\Lambda^{-1} h < g < \Lambda h, \quad
\text{almost everywhere}.
\end{equation}
We always assume that $M\setminus S$ is open and dense in $M$. 
\smallskip

First, suppose that $g$ has Ricci curvature bounded from below on $M\setminus S$. A natural question is whether the metric completion of the length metric space defined by $(M\setminus S, g)$ has Ricci curvature bounded below in the sense of RCD spaces. Recently this was considered by Dai--Wang--Wang--Wei~\cite[Theorem 1.5]{DWWW}, as well as Honda--Sun~\cite[Theorem 1.2]{HondaSun} in a more general setting. They showed that if $S$ has codimension at least four, then the metric completion of $(M\setminus S,g)$ is an RCD space.

Our first aim is to generalize \cite[Theorem 1.5]{DWWW}: it is enough that $S$ has upper Minkowski codimension strictly greater than $3-\frac{1}{n-1}$; see \cref{rem:MinkowskiVsAssouad} and \cref{def:Codimension}.
\begin{theorem}\label{thm:RCDcodim3}
 Let $n\geq 2$ be an integer and $K\in\mathbb R$. Suppose that $(M^n, g)$ is a closed $L^\infty$ Riemannian manifold, and $S\subset M$ is a closed subset of upper Minkowski codimension greater than $3-\frac{1}{n-1}$. Suppose that $\mathrm{Ric}(g) \geq Kg$ on $M\setminus S$. Then the metric completion of $(M\setminus S, g)$ is an $RCD(K,n)$-space. 
\end{theorem}

Note that \cref{thm:RCDcodim3} implies removability of smooth submanifolds of codimension three. Moreover, the codimension three condition is sharp for the removability of smooth submanifolds, since a 2d cone over a circle of length greater than $2\pi$ is not an RCD space; 
see, e.g., \cite[Remark 4.14]{HondaSun}. We remark that \cite{HondaSun} considers the removability of more general types of subsets in a more general setting than the one of closed $L^\infty$-Riemannian manifolds; see also the recent \cite[Section 6]{BZ26}. At the same time our method differs from \cite{DWWW} and \cite{HondaSun} in that we use a comparison argument, rather than integration by parts, to prove the crucial gradient estimate for eigenfunctions of the Laplacian. For further comments on \cref{thm:RCDcodim3}, and in particular on the codimension threshold $3-\frac{1}{n-1}$, see \cref{rem:Better-on-Lambda}, \cref{rem:PossiblyMoreGeneralResults}. 

In particular, in \cref{rem:Better-on-Lambda} we observe that \cref{thm:RCDcodim3} can be slightly improved: it is enough that $S\subset M$ is a closed subset of codimension greater than $3-\frac{1}{n-1}-\xi(n,\Lambda)$, where $0<\xi(n,\Lambda)\leq \frac{n-2}{n-1}$ is a positive threshold only depending on the dimension $n$ and the constant $\Lambda$ in \eqref{eqn:LambdaDef}. Moreover, we notice that $3-\frac{1}{n-1}-\xi(n,\Lambda)\to 2$ as $\Lambda\to 1^+$. For related RCD extension theorems in the spirit of \cref{thm:RCDcodim3}, see Bertrand--Ketterer--Mondello--Richard \cite{BertrandKettererMondelloRichard}.
\smallskip

Our motivation for proving \cref{thm:RCDcodim3} is to study removability of singularities of $L^\infty$ Einstein metrics. In dimension four we obtain the following result.
\begin{theorem}\label{thm:4dEinstein}
    Let $(M^4, g)$ be a four-dimensional closed $L^\infty$ manifold, and $S\subset M$ be a closed subset of Assouad codimension greater than $3-\frac{1}{3}$; see \cref{def:Codimension}. Suppose that $\mathrm{Ric}(g) = Kg$ on $M\setminus S$ for some $K\in \mathbb{R}$. Then $g$ extends as a smooth Einstein metric on $M$, perhaps for a different smooth structure on $M$. 
\end{theorem}

Removability of isolated singularities has been considered in various works previously, in different contexts, such as Smith--Yang~\cite[Theorem 3.1]{SmithYang}, Dai--Wang--Wang--Wei~\cite[Theorem 1.6]{DWWW}, Bando--Kasue--Nakajima~\cite[Theorem 5.1]{BandoKasueNakajima}, and the work of Uhlenbeck on removability of singularities for solutions of Yang--Mills equations \cite[Theorem 4.1]{UhlenbeckYangMills}. Removability of non-isolated singularities has also been considered in the literature. E.g., in Shi--Tam~\cite[Theorem 6.1]{ShiTamPacific} the authors prove that the singularity of a Lipschitz continuous Riemannian metric which is Einstein outside a set of codimension $>1$ is removable. See also Lee--Tam~\cite[Corollary 1.2]{LeeTam}, where the authors prove that the singularity of a Riemannian metric in $W^{1,p}_{\mathrm{loc}}$ with $p>n$ which is Einstein outside a set of codimension $>2$ is removable. In \cite[Theorem 1.2]{DWWW} the authors remove possibly non-isolated singularities for manifolds of the form $T^n\#M$. See also the recent Wang--Wang--Xie~\cite[Theorem 1.5]{WWX26} for the case of $T^n$. As far as we know, Theorem~\ref{thm:4dEinstein} is the first general result (without additional topological assumptions on the underlying manifold) about removability of non-isolated singular sets for metrics that are not continuous a priori.
\smallskip

As we will explain below, our results can be applied to the following conjecture due to Schoen; see, e.g., Li--Mantoulidis~\cite[Conjecture 1.5]{LiMantoulidis}. 

\begin{conjecture}\label{conj:Schoen}
    Suppose that $(M^n, g)$ is an $L^\infty$-Riemannian manifold, and $S\subset M$ is a smooth closed embedded submanifold of codimension at least 3. Suppose that $g$ has non-negative scalar curvature on $M\setminus S$, and $M$ has non-positive Yamabe invariant. Then $g$ extends smoothly to $M$ as a metric with zero Ricci curvature. 
\end{conjecture}

The conjecture is known in dimension 3 by
Li--Mantoulidis~\cite{LiMantoulidis}. In higher dimensions some partial results were obtained; see, for instance, Dai--Wang--Wang--Wei~\cite{DWWW}, Kazaras~\cite{Kazaras}, Lee--Tam~\cite{LeeTam}; and the recent Wang--Wang--Xie~\cite{WWX26}, Khuri--Wang--Wang \cite{KWW26}, and Bi--Zhu \cite{BZ26}. In general the conjecture fails in high dimensions, as shown by Cecchini--Frenck--Zeidler~\cite{CecchiniFrenckZeidler}. The conjecture naturally breaks into two sub-problems:
\begin{itemize}
    \item[(i)] Show that under the assumptions of the conjecture we have $\mathrm{Ric}(g) = 0$ on $M\setminus S$.
    \item[(ii)] Show that if $\mathrm{Ric}(g)= 0$ on $M\setminus S$, then $g$ extends as a smooth Ricci flat metric on $M$ (perhaps with a different smooth structure). 
\end{itemize}
In fact, our Theorem~\ref{thm:4dEinstein} addresses (ii).

For manifolds $M$ of dimension $n\geq 6$ one cannot expect a removable singularities result such as in Theorem~\ref{thm:4dEinstein}, since there are non-standard Einstein metrics on the 5-sphere by \cite{Bohm}. Indeed, if $\mathrm{Ric}(g_N) = (n-2)g_N$ on an $(n-1)$-dimensional manifold $N$, then the spherical suspension $(0,\pi)\times N$ with metric $dt^2+\sin^2(t)g_N$ is an $L^\infty$-metric which is smooth Einstein outside the poles. Moreover, note that the proof of \cref{thm:4dEinstein} extends to dimension $5$ (the assumption, in this case, being ``Assouad $\mathrm{codimension}>3-\frac{1}{4}$") if and only if there are no non-round Einstein metrics on $\mathbb S^4$; see \cref{rem:ToDim5}. 

Nevertheless, many ingredients in the proof of Theorem~\ref{thm:4dEinstein} apply in higher dimensions, and even in the case when we only have a two-sided Ricci curvature bound. In fact Theorem~\ref{thm:4dEinstein} follows from the following general result, together with Theorem~\ref{thm:RCDcodim3}. 
\begin{theorem}\label{thm:extension}
    Let $n\geq 2$ be an integer, and let $K\geq 0$. Let $(X,\mathrm{d},\mathcal{H}^n)$ be an $\mathrm{RCD}(-K,n)$ space. Let us assume the following:
    \begin{enumerate}
        \item There is a closed set $\Sigma \subset X$ with Assouad codimension greater than $2$ (see \cref{def:Codimension}) such that $X\setminus\Sigma$ is a smooth Riemannian manifold;
        \item $|\mathrm{Ric}| \leq K$ on $X\setminus \Sigma$.
    \end{enumerate}
    Let $X^{\mathrm{reg}}$ be the set of regular points of $X$. Then $X^{\mathrm{reg}}$ is open, and has the structure of a $C^{1,\alpha}$-Riemannian manifold for any $\alpha < 1$.
\end{theorem}
The assumption on the codimension being greater than $2$ in \cref{thm:extension} is sharp; see \cref{rem:Codimension-Sharp}. Moreover we stress that the assumptions in item (1) and item (2) in \cref{thm:extension} are important in order to have a $C^{1,\alpha}$-Riemannian manifold structure on $X^{\mathrm{reg}}$. In fact, in \cite[Theorem 1.2]{ColdingNaberReifenberg} the authors construct non-collapsed Ricci limit spaces $X$ such that $X=X^{\mathrm{reg}}$ but, for any $0<\alpha<1$, there is no $C^{1,\alpha}$-Riemannian structure on $X=X^{\mathrm{reg}}$. A different criterion for manifold regularity was obtained recently by Honda--Zhang \cite[Theorem 1.1]{HondaZhang}.  
\smallskip

Let us now comment on the application of our results to \cref{conj:Schoen}. Notice that the following result is also implicit in the work of Lee--Tam \cite{LeeTam} under slightly stronger assumptions; see, e.g., \cite[Corollary 1.1]{LeeTam}.

\begin{theorem}\label{thm:highdimEinstein}
  Let $n\geq 2$ be an integer, and $a>0$. There exists $0<\varepsilon:=\varepsilon(n,a)<1$ such that the following holds.
  
  Let $(M^n, g)$ be a closed $L^\infty$-Riemannian manifold, and $S\subset M$ a closed subset of Assouad codimension greater than $2+a$; see \cref{def:Codimension}. Suppose that $\mathrm{Ric}(g) = Kg$ on $M\setminus S$ for some $K\in\mathbb R$, and in addition suppose that one of the following conditions holds:
  \begin{itemize}
      \item[(a)] Either $(1 - \epsilon)h \leq g \leq (1+\epsilon) h$ for a smooth metric $h$ on $M$,
      \item[(b)] Or $g$ is continuous. 
  \end{itemize}
    Then $g$ extends to $M$ as a smooth Einstein metric, perhaps for a different smooth structure.
\end{theorem}

As a consequence of \cref{thm:highdimEinstein}, using \cite[Corollary 1.3]{LeeTam} (see also \cite[Theorem 1.1]{BurkhardtGuim}) and the proof of \cite[Theorem 1.1]{LeeTam}, Conjecture~\ref{conj:Schoen} holds for metrics satisfying (a) or (b) in \cref{thm:highdimEinstein}. 
We stress that \cref{thm:highdimEinstein} is a consequence of our results discussed above. Indeed, the assumptions imply using Theorem~\ref{thm:RCDcodim3} (see, in particular, the slightly stronger statement in \cref{rem:Better-on-Lambda}) that the metric completion of $M\setminus S$ is an RCD space; then, the conditions (a) or (b) imply that all points are regular, so that Theorem~\ref{thm:extension} can be applied.

The necessity of allowing a different smooth structure can be seen from the existence of exotic structures on $S^7$ that are biLipschitz equivalent to the standard structure, with a map that is smooth away from a point (see Cecchini--Frenck--Zeidler~\cite[Theorem 2.7]{CecchiniFrenckZeidler}).

Notice that if we assume even more regularity for the metric, such as being in the Sobolev space $W^{1,p}$ for $p > n$, stronger results can be obtained using the Ricci flow; see Jiang--Sheng--Zhang~\cite{JiangShengZhang},  \cite[Theorem 1.1]{ShiTamPacific}, and \cite[Corollary 1.2]{LeeTam}. It is worth noticing that in  \cite[Theorem 4.2]{JiangShengZhang} the authors prove that if the metric is $C^0\cap W^{1,p}$ with $p>n$, $\mathcal{H}^{n-\frac{p}{p-1}}(\Sigma)<\infty$, and the metric is asymptotically flat and Ricci flat outside $\Sigma$, then the space is $\mathrm{RCD}$.
\smallskip

We conclude with a natural open question related to \cref{thm:RCDcodim3}.
\begin{question}\label{quest:RCD-if-affirmative}
    Let $n\geq 2$ be an integer. Suppose that $(M^n, g)$ is a closed $L^\infty$ Riemannian manifold, and $S\subset M$ is a closed subset of upper Minkowski codimension greater than $2$. Suppose that $\mathrm{Ric}(g) \geq Kg$ on $M\setminus S$ for some $K\in\mathbb R$. Is it true that the metric completion of $(M\setminus S, g)$ is an $RCD(K,n)$-space?
\end{question}
Notice that $\mathrm{codimension}>2$ cannot be further weakened in the above question: see, e.g., \cite[Remark 4.14]{HondaSun}. We refer the reader to \cref{rem:Better-on-Lambda} and \cref{rem:PossiblyMoreGeneralResults} for related discussions. Notice that if \cref{quest:RCD-if-affirmative} has a positive answer, one could immediately weaken the assumption of \cref{thm:4dEinstein} to ``$S\subset M$ has (Assouad) codimension greater than $2$". 
\smallskip

\textbf{Acknowledgments}. The authors were supported in part by NSF grants DMS-2550590 (G. A.) and DMS-2506325 (G. Sz.). The authors wish to thank Matt Gursky, Aaron Naber, and Daniele Semola for useful conversations around the topic of this paper. 
\smallskip 

We formulated the results of this paper and devised the overall proof strategy in 2025 and early 2026, when we also prepared an initial draft of this paper with all the proofs. After having identified a mistake in our proof of \cref{thm:main}, we used ChatGPT to assist us with its proof. This enabled us to complete our original strategy and finish the paper. We also used ChatGPT for routine review, and identifying potentially relevant literature. 

\section{Extending the RCD property across \texorpdfstring{$S$}{S}}

\subsection{Preliminaries}

Let us recall the definition of almost smooth compact metric measure space (see \cite{Honda, Gabor}). In our treatment we will stick to the compact case. Using the recent \cite[Corollary 5.4]{HondaSun} we believe that our results can be generalized to the non-compact weighted case with additional technicalities which are out of the scope of this note. We note that, while we were completing this paper, some variations of the arguments in this section have been used in the independent papers \cite{WWX26, KWW26, BZ26} to obtain similar results.

\begin{definition}[Almost smooth metric measure space]\label{def:AlmostSmooth}
A compact metric measure space $(Z, d, \mathfrak{m})$ is an \textit{$n$-dimensional almost smooth metric measure space}, if there is an open subset $\Omega \subset Z$ satisfying the following conditions:
\begin{enumerate}
    \item There is a smooth $n$-dimensional Riemannian manifold $(M, g)$ and a homeomorphism $\varphi : \Omega \to M$, such that $\varphi$ defines a local isometry between $(\Omega, d)$ and $(M, d_g)$.
    \item The restriction of the measure $\mathfrak{m}$ to $\Omega$ coincides with the $n$-dimensional Hausdorff measure.
    \item The complement $Z \setminus \Omega$ has measure zero, i.e., $\mathfrak{m}(Z \setminus \Omega) = 0$, and it has zero $2$-capacity in the following sense: there is a sequence of smooth functions $\varphi_i : \Omega \to [0,1]$ with compact support in $\Omega$ such that
    \begin{enumerate}
        \item For any compact $A \Subset \Omega$, we have $\varphi_i|_A = 1$ for sufficiently large $i$.
        \item We have
        \begin{equation}
            \lim_{i\to\infty} \int_{\Omega} |\nabla \varphi_i|^2 \, \mathrm{d}\mathcal{H}^n = 0.
        \end{equation}
    \end{enumerate}
\end{enumerate}
\end{definition}

\begin{setting}\label{Setting}
    Let $(M,d_g)$ be an $n$-dimensional incomplete Riemannian manifold equipped with the measure $\mathcal{H}^n$. Let $(Z,d)$ be the metric completion of $(M,d_g)$, and assume $(Z,d)$ to be compact. Define a measure $\mathfrak{m}$ on $Z$ so that $\mathfrak{m}|_M=\mathcal{H}^n|_M$, and $\mathfrak{m}|_{Z\setminus M}\equiv 0$. We call $S:=Z\setminus M$ the set of \textit{singular} points, and $M$ the regular points. Sometimes we denote $\mathcal{H}^n$ by $\mathrm{vol}$ on the regular part. 
\end{setting} 

Notice that in \cref{Setting}, we have that $(Z,d,\mathfrak{m})$ is infinitesimally Hilbertian; namely, the Cheeger energy is quadratic. Let $f\in W^{1,2}(Z,d,\mathfrak m)$ and $\lambda\geq 0$. When we write $\Delta f\geq -\lambda f$ we mean that 
\begin{equation}\label{eqn:WeakPDE}
\int_Z \langle \nabla f,\nabla g\rangle \mathrm{d}\mathfrak{m} \leq \lambda\int_Z fg \mathrm{d}\mathfrak{m}, \quad \text{for every $g\in \mathrm{Lip}_c(Z,d)$ with $g\geq 0$.}
\end{equation}
Moreover, with $\Delta f=-\lambda f$ we mean that \eqref{eqn:WeakPDE} holds with the equality sign, for every $g\in \mathrm{Lip}_c(Z,d)$.

\begin{assumptions}\label{assumptions}
    Let $n\geq 2$, and let $(Z,d,\mathfrak{m})$ be as in \cref{Setting}. Furthermore, we assume the following.
    \begin{enumerate}
        \item (Global Sobolev inequality) There exists $q>1$ and a constant $\mathcal{S}$ such that for every $u\in W^{1,2}(Z,d,\mathfrak{m})$ we have 
        \begin{equation}\label{eqn:GlobalSobolev}
            \|u\|_{2q}\leq \mathcal{S}\left(\|\nabla u\|_2+\|u\|_2\right).
        \end{equation}
        \item (Ricci Lower bound) On the regular part $(M,g)\subset Z$ we have 
        \begin{equation}\label{eqn:RicciControl}
        \mathrm{Ric}_{(M,g)} \geq Kg,
        \end{equation}
        for some $K\in\mathbb R$. 
    \end{enumerate}
\end{assumptions}
We say that a Radon measure $\mathfrak{m}$ on a metric space $(Z,d)$ is \textit{locally doubling} if, for every $R>0$, there exist $C(R)>0$ such that 
\[
\mathfrak{m}(B_{2r}(x))\leq C(R)\mathfrak{m}(B_r(x)), \qquad \text{for every $x\in Z$ and $0<r\leq R$.}
\]
In our terminology, a metric-measure space $(Z,d,\mathfrak{m})$ is said to be a \textit{PI space} if $\mathfrak{m}$ is locally doubling and supports a \textit{local} $1$-Poincaré inequality, see \cite[Section 8.1]{HKST15}.
\begin{remark}\label{rem:LinftyArePi}
    Notice that spaces as in \cref{thm:RCDcodim3} satisfy both \cref{Setting} and \cref{assumptions}. Indeed, since they are biLipschitz equivalent to a smooth manifold, with comparable measure, they are PI spaces. Hence the global Sobolev inequality follows from \cite[Theorem 5.1(1)]{HK00}.
\end{remark}

\begin{lemma}\label{lem:Step1}
    Let $(Z,d,\mathfrak{m})$ be as in \cref{Setting}, and let us assume it verifies the global Sobolev inequality \eqref{eqn:GlobalSobolev} in \cref{assumptions}. Then we have the following:
    \begin{enumerate}
        \item If $f\in W^{1,2}(Z,d,\mathfrak{m})$ verifies $\Delta f\geq -\lambda f$ for some $\lambda\geq 0$, then $f^+\in L^\infty(Z,\mathfrak{m})$. Actually, we have $f^+\leq C(q,\mathcal{S},\lambda)\|f\|_2$. Here $f^+$ is the positive part of $f$.
        \item If $f\in W^{1,2}(Z,d,\mathfrak{m})$ verifies $\Delta f = -\lambda f$ for some $\lambda\geq 0$, then $f\in L^\infty(Z,\mathfrak{m})$. Actually, we have $|f|\leq C(q,\mathcal{S},\lambda)\|f\|_2$.
    \end{enumerate}
\end{lemma}
\begin{proof}
    The proof is classical but we provide it here for the reader's ease. This comes from a classical application of the Moser iteration argument, and it follows from the global Sobolev inequality \eqref{eqn:GlobalSobolev}. Let us first sketch item (1). Let $p\geq 1$ be an arbitrary number, and denote $\alpha:=p+1\geq 2$. Testing the inequality $\Delta f\geq -\lambda f$ with $(f^+)^p$ (see \eqref{eqn:WeakPDE}) we get
    \begin{equation}\label{eqn:UpperEstimate}
    \|\nabla ((f^+)^{\alpha/2})\|_2 \leq \frac{\alpha\sqrt{\lambda}}{2\sqrt{\alpha-1}} \|f^+\|^{\alpha/2}_\alpha.
    \end{equation}
    Using the Sobolev inequality we get
    \[
    \|\nabla ((f^+)^{\alpha/2})\|_2 \geq \mathcal{S}^{-1}\|(f^+)^{\alpha/2}\|_{2q} - \|(f^+)^{\alpha/2}\|_{2} = \mathcal{S}^{-1}\|f^+\|_{q\alpha}^{\alpha/2} - \|f^+\|_{\alpha}^{\alpha/2}.
    \]
    Joining the last two inequalities we get 
    \[
    \|f^+\|_{q\alpha} \leq \mathcal{S}^{2/\alpha}\left(\frac{\alpha\sqrt{\lambda}}{2\sqrt{\alpha-1}}+1\right)^{2/\alpha}\|f^+\|_\alpha.
    \]
    Applying iteratively the last inequality with $\alpha_k:=2q^k$ for all $k\geq 0$ we get $\|f^+\|_\infty\leq C(q,\mathcal{S},\lambda)\|f\|_2$, i.e., the sought conclusion of item (1). To prove item (2), it is enough to apply item (1) to $-f$ as well.
\end{proof}
\begin{lemma}\label{lem:Step2}
    Let $(Z,d,\mathfrak{m})$ be as in \cref{Setting}, and let us assume it verifies the global Sobolev inequality \eqref{eqn:GlobalSobolev}, and the Ricci lower bound \eqref{eqn:RicciControl} in \cref{assumptions}. Let $f\in W^{1,2}(Z,d,\mathfrak{m})$ verify $\Delta f= -\lambda f$ for some $\lambda\geq 0$. Then there is a constant $C>0$ such that 
    \begin{equation}\label{eqn:GradEstimate}
        |\nabla f|(x)\leq C d(x,S)^{-1}, \qquad \forall x:d(x,S)\leq 1.
    \end{equation}
\end{lemma}
\begin{proof}
    By item (2) of \cref{lem:Step1}, $\|f\|_\infty\leq C(q,\mathcal{S},\lambda)\|f\|_2$. Then, after having obtained the $L^\infty$-estimate on $f$, \eqref{eqn:GradEstimate} is a consequence of the latter estimate, and the classical Cheng--Yau~\cite{ChengYau} gradient estimate on the regular part of $M$. The detailed computations are worked out, e.g., in the proof of item (2) of \cite[Lemma 2.7]{DWWW}, and we omit them. In particular it is enough to use the first inequality on \cite[Page 10]{DWWW}, which is a consequence of the Riemannian computations in \cite[Theorem 2.6]{DWWW}.
\end{proof}

\begin{lemma}\label{lem:Bochner++}
    Let $n\geq 2$ be an integer and $K\in\mathbb R$. Let $U$ be an open set of a smooth $n$-dimensional Riemannian manifold $(M^n,g)$. Let us assume $\mathrm{Ric}\geq K$ on $U$. Let $f\in C^{\infty}(U)$ be a bounded function on $U$ such that $\Delta f=-\lambda f$ on $U$ for some $\lambda\geq 0$.
    
    For every $\frac{n-2}{2(n-1)}<\beta<1$ there is $c:=c(n,\beta,\lambda, K,\|f\|_{L^{\infty}(U)})>0$ such that, calling 
    \[
    u:=(1+|\nabla f|^2)^{\beta},
    \]
    we have 
    \begin{equation}\label{eqn:AlmostSubHarmonic}
    \Delta u \geq -cu, \qquad \text{on $U$}.
    \end{equation}
\end{lemma}
\begin{proof}
    Using the improved Kato inequality for eigenfunctions as in the proof of \cite[Theorem 5.1]{CCHSTT}, for every $\kappa<\frac{1}{n-1}$, and for every $\beta\in (0,1)$, we have
\begin{equation}\label{eqn:PDENew}
\Delta (1 + |\nabla f|^2)^\beta \geq \beta\left(2-\frac{4(1-\beta)}{1+\kappa}\right) (1 + |\nabla f|^2)^{\beta-1} |\nabla^2 f|^2 - c (1+|\nabla f|^2)^\beta. 
\end{equation}
Here the constant $c$ depends on $\lambda, K,\beta,\kappa,||f||_{L^\infty(U)}$. There exists $\eta>0$ such that $\beta=\frac{n-2}{2(n-1)}+\frac{\eta}{2}$, and choosing $\frac{1}{n-1}-\eta<\kappa<\frac{1}{n-1}$ we have
\[
\beta\left(2-\frac{4(1-\beta)}{1+\kappa}\right) >0.
\]
Using this in \eqref{eqn:PDENew},
 we get \eqref{eqn:AlmostSubHarmonic}, as desired.
 \end{proof}

 To prove \cref{thm:RCDcodim3} we are going to use the following result, see \cite[Corollary 8]{Gabor}, whose proof is a slight modification of \cite[Corollary 3.10]{Honda}.

\begin{theorem}\label{thm:ToRCD}
Let $(Z,d,\mathfrak{m})$ be as in \cref{Setting}. Then $(Z,d,\mathfrak{m})$ is an \textit{RCD}$(K, n)$ space, where $K \in \mathbb{R}$, if it is an almost smooth compact metric measure space, and the following conditions hold.
\begin{enumerate}
    \item The Sobolev to Lipschitz property holds, that is any $f \in W^{1,2}(Z,d,\mathfrak{m})$, with $|\nabla f|(x) \leq 1$ for $\mathfrak{m}$-almost every $x$, has a 1-Lipschitz representative.
    \item The $L^2$-strong compactness condition holds, that is the inclusion $W^{1,2}(Z,d,\mathfrak{m}) \hookrightarrow L^2(Z,\mathfrak{m})$ is a compact operator.
    \item Any eigenfunction in $W^{1,2}(Z,d,\mathfrak{m})$ of the Laplacian on $Z$ is Lipschitz.
    \item $\mathrm{Ric}_{(M,g)} \geq Kg$ on the regular part $M\subset Z$.
\end{enumerate}
\end{theorem}

\subsection{The proof of \texorpdfstring{\cref{thm:RCDcodim3}}{Theorem 1.1}}

In the following result we construct a useful barrier function used in the proof of \cref{thm:RCDcodim3}. We start with a definition.
\begin{definition}\label{def:Codimension}
    Let $(X,d)$ be a metric space of Hausdorff dimension $n$, and $\Sigma\subset X$ be a subset. Let $0\leq k\leq n$. We say that $\Sigma$ {\rm has 
    (Assouad)
    codimension greater than $k$} if either $k=n$ and $\Sigma=\emptyset$; or $k<n$ and there is $0\leq c<n-k$, $\mathcal{K} > 0$, and $\rho_0>0$ such that the following holds. For every $x\in X$, every $0<r\leq \rho_0$, and every $0<s\leq 1$ we have that $\Sigma\cap B_r(x)$ can be covered by $\leq \mathcal{K} s^{-c}$ balls of radius $sr$. We stress that if $k=n$, we have $\Sigma=\emptyset$.
\end{definition}
 
 When $S$ is a smooth submanifold, the following lemma (and consequently \cref{LinftyResultForThm1} and \cref{thm:LipschitzRepresentativeThm1}) has a slightly simpler proof. Indeed, it would be enough to define $v(x):=\int_S G(x,y)\mathrm{d}\mathcal{H}^d(y)$ in \eqref{eqn:DefinitionNu}. Inspired by the recent work \cite[Proposition 3.2]{WWX26},  we can give a proof of the following Lemma in a more general setting.
 \begin{lemma}\label{lem:ConstructionvLinfty}
     Let $n\geq 2$, $(M^n,g)$ be a closed $L^\infty$-Riemannian manifold, and suppose that $S\subset M$ is closed with codimension greater than 2. Let $c$ be as in \cref{def:Codimension} and fix $0<a<n-c-2$. Assume that $g$ is smooth on $M\setminus S$.
     
     Let $\Omega\Subset M$ be a smooth connected open subset, with $\overline{\Omega}\neq M$, such that $S\Subset \Omega$. Then there is a function $v:\Omega\setminus S\to (0,\infty)$ such that the following hold:
     \begin{enumerate}
         \item $v$ is harmonic on $\Omega\setminus S$ and $v\in L^1(\Omega)$;
         \item For every $\tilde\Omega$ with $S\Subset \widetilde\Omega\Subset\Omega$ there exists $C>0$ and $\varepsilon>0$ such that:
         \begin{equation}\label{eqn:Asymptoticsv}
        v(x)\geq C\mathrm{d}(x,S)^{-a}, \qquad \text{for every $x\in \widetilde\Omega\setminus S$ with $\mathrm{d}(x,S)<\varepsilon$.}
         \end{equation}
     \end{enumerate}
 \end{lemma}
\begin{proof}
    When $n=2$, we have $S=\emptyset$ and one can easily check that the statement is true; hence, we will assume $n\geq 3$ from now on. In this proof the constant $C$ might change from line to line. Using \cref{def:Codimension}, for every integer $j\geq j_0$, where $j_0$ is large enough, we can choose points $\{y_{j,i}\}_{i=1}^{N_j}\subset S$ such that 
    \[
    S\subset \bigcup_{i=1}^{N_j} B_{2^{-j}}(y_{j,i}), \qquad N_j\leq C2^{jc}.
    \]
    Let us define the positive measure
    \[
    \nu:=\sum_{j\geq j_0}\sum_{i=1}^{N_j} 2^{-j(n-2-a)}\delta_{y_{j,i}}.
    \]
    Notice that $\nu(S)\leq C\sum_{j\geq j_0} 2^{-j(n-2-a-c)}<\infty$; hence $\nu$ is a positive and finite measure.
    
    Let $D$ be the diagonal in $\Omega\times \Omega$. Let $G:\Omega\times \Omega\setminus D\to (0,\infty)$ be the Green's function with Dirichlet boundary conditions on $\Omega$. Define
    \begin{equation}\label{eqn:DefinitionNu}
    v(x):=\int_{S} G(x,y)\mathrm{d}\nu(y), \quad \forall x\in \Omega\setminus S.
    \end{equation}
    Notice that $v$ is smooth, positive and harmonic on $\Omega\setminus S$. Moreover, for every $\Omega_1$ with $\widetilde\Omega\Subset \Omega_1\Subset\Omega$ there exists $C<1$ such that 
    \begin{equation}\label{eqn:EstG}
    C^{-1}d(x,y)^{2-n}\geq G(x,y)\geq Cd(x,y)^{2-n}, \qquad \text{for every $x,y\in \Omega_1\times\Omega_1\setminus D$}.
    \end{equation}
    The latter property is classical if $g$ is a smooth Riemannian metric, and follows from \cite[Theorem 7.1]{LittmanStampacchiaWeinberger} when $(M,g)$ is an $L^\infty$-Riemannian manifold. 
    
    We can prove the lower inequality in \eqref{eqn:Asymptoticsv} as follows. The number $\varepsilon$ will be chosen small enough during the proof. Let us fix $x\in B_{\varepsilon}(S)\setminus S\Subset\widetilde\Omega$. Let $\bar y\in S$ be such that $d(x,\bar y)=d(x,S)=:\bar r$. If $\varepsilon>0$ is small enough, there is $j\geq j_0$ so that $\bar r\leq 2^{-j}<2\bar r$. Choose $i$ such that $\bar y\in B_{2^{-j}}(y_{j,i})$. For $\varepsilon$ small enough we have $x,y_{j,i}\in\Omega_1$ and $d(x,y_{j,i})\leq 3\cdot 2^{-j}$.
    Thus
    \begin{equation}
    \begin{aligned}
        v(x)&=\int_S G(x,y)\mathrm{d}\nu(y) \geq 2^{-j(n-2-a)}G(x,y_{j,i}) \\
        &\geq C 2^{-j(n-2-a)}d(x,y_{j,i})^{2-n} \geq C2^{ja}\geq Cd(x,S)^{-a},
    \end{aligned}
    \end{equation}
    as desired. Finally notice that by the upper estimate in \eqref{eqn:EstG} and the fact that $\nu(S)<\infty$, one gets that $v\in L^1(\Omega)$, as desired.
\end{proof}
\begin{remark}\label{rem:MinkowskiVsAssouad}
    As it is readily seen from its proof, for \cref{lem:ConstructionvLinfty} (and thus for \cref{LinftyResultForThm1}, \cref{thm:LipschitzRepresentativeThm1}, and \cref{thm:RCDcodim3}) it is enough to assume that $\Sigma$ has upper Minkowski codimension greater than $2$, instead of the stronger ``(Assouad) codimension greater than $2$" in \cref{def:Codimension}. We say that a compact set $\Sigma\subset X$ {\rm has 
    upper Minkowski codimension greater than $k$} if either $k=n$ and $\Sigma=\emptyset$ (where $n$ is the Hausdorff dimension of $X$); or $k<n$ and there is $0\leq c<n-k$, $\mathcal{K} > 0$, and $\rho_0>0$ such that the following holds. For every $0<s\leq 1$ we have that $\Sigma$ can be covered by $\leq \mathcal{K} s^{-c}$ balls of radius $s\rho_0$. 
\end{remark}

\begin{prop}\label{LinftyResultForThm1}
    Let $n\geq 2$, $(M^n,g)$ be a closed $L^\infty$-Riemannian manifold, and suppose that $S\subset M$ is closed with codimension greater than 2. Let $c$ be as in \cref{def:Codimension}. Assume that $g$ is smooth on $M\setminus S$.
    
Let $r(x):= d(x,S)$. Assume $u:M\to [0,+\infty)$ is an $L^1$ function, which is smooth on the regular set $M\setminus S$, and such that there exists $\alpha>0$ and $\beta >2+c-n$ for which 
\[
\Delta u \geq -\alpha u, \qquad u \leq \alpha r^{\beta} \qquad \text{on the regular set $M\setminus S$}.
\]
Then $u\in L^\infty(M)$.
\end{prop}
\begin{proof}
When $n=2$, $S=\emptyset$, thus the statement is easily seen to be true. Let us assume $n\geq 3$ from now on. By arguing in a product $ I\times M$, where $I$ is an open interval, and considering $\tilde u(t,x):=e^{\sqrt{\alpha}t}u$ (and using a localized version of \cref{lem:ConstructionvLinfty} in the argument below), we can assume without loss of generality that $\Delta u\geq 0$. Since $u$ is smooth on $M\setminus S$, it is enough to show it is bounded in a neighborhood of $S$. 

    Choose $a$ such that
    \[
    \max\{-\beta,0\} < a < n-c-2.
    \]

    Let us take $v$ the function constructed in \cref{lem:ConstructionvLinfty} with $S\Subset \Omega\Subset B_{\xi}(S)$, for a small enough $\xi>0$. Fix $\Omega'$ with $S\Subset\Omega'\Subset\Omega$. We have that $v$ is harmonic on $\Omega\setminus S$ and, by using Item (2) of \cref{lem:ConstructionvLinfty}, by possibly shrinking $\Omega'$ there is $\tilde c$ such that
    \[
    v(x)\geq \tilde cr(x)^{-a} \quad \forall x\in \Omega'\setminus S.
    \]
    
    Let us now fix $\vartheta>0$, and consider the function $u_\vartheta:=u-\vartheta v$. Notice that
    \[
    u(x)-\vartheta v(x) \leq \alpha r(x)^{\beta}\left(1-\tilde c \alpha^{-1}\vartheta r(x)^{-a-\beta}\right), \quad \forall x\in \Omega'\setminus S.
    \]
    Since $-a-\beta<0$, we get that for every $\vartheta$, the function $u_\vartheta$ is $\leq 0$ in a sufficiently small neighborhood $B_{T_\vartheta}(S)$ of $S$. Since $\Delta u \geq 0$ on $M\setminus S$ and $v$ is harmonic on $\Omega\setminus S$, we get that 
    \[
    g_\vartheta:=\max\{0,u_\vartheta\}
    \]
    is a subharmonic function on $\Omega\setminus S$ which is constantly equal to zero in $B_{T_\vartheta}(S)\setminus S$. So it extends to a weakly sub-harmonic function on $\Omega$ which is zero on $B_{T_\vartheta}(S)$. Moreover, it is readily seen that $g_\vartheta\in W^{1,2}(\Omega)$ for every $\vartheta>0$. Indeed $u_\vartheta\in W^{1,2}(\Omega\setminus B_{\delta}(S))$ for every small $\delta>0$, and so $g_\vartheta\in W^{1,2}(\Omega\setminus B_{\delta}(S))$ for every small $\delta>0$ as well. But $g_\vartheta\equiv 0$ on $B_{T_\vartheta}(S)$, so $g_\vartheta\in W^{1,2}(\Omega)$. 
    
    Notice that $0\leq g_\vartheta\leq u$. Thus, by also using the $L^\infty$-$L^1$ bound for sub-harmonic functions in PI-spaces (recall \cref{rem:LinftyArePi}), see, e.g., \cite[Theorem A.4]{BenattiViolo}, for any $\tilde\Omega\Subset \Omega$ there is a constant $\tilde C>0$ such that
    \[
    \|g_\vartheta\|_{L^\infty(\tilde\Omega)} \leq \tilde C\|g_\vartheta\|_{L^1(\Omega)}\leq \tilde C\|u\|_{L^1(M)}<\infty.
    \]
    Notice that $\tilde C$ is independent of $\vartheta$.
    Then taking $\tilde\Omega\Subset \Omega$ to be a neighborhood of $S$, and $\vartheta\to 0$, the proof is concluded because the previous inequality shows that $u$ is bounded in a neighborhood of $S$, as desired.
\end{proof}

\begin{theorem}\label{thm:LipschitzRepresentativeThm1}
Let $n\geq 2$, $(M^n,g)$ be a closed $L^\infty$-Riemannian manifold, and suppose that $S\subset M$ is closed with codimension greater than $3-\frac{1}{n-1}$; see \cref{def:Codimension}. Assume that $g$ is smooth on $M\setminus S$ and $\mathrm{Ric}(g)\geq Kg$ on $M\setminus S$ for some $K\in\mathbb R$. Let $f\in W^{1,2}(M)$ satisfy $\Delta f = -\lambda f$ for some $\lambda \geq 0$. Then $f$ has a Lipschitz representative.
\end{theorem}
\begin{proof}
    Let $c$ be as in \cref{def:Codimension}. By the beginning of the proof of \cref{lem:Step2} we have that $f\in L^{\infty}(M)$. Let us fix a $\gamma$ such that
    \[
    \frac{n-2}{2(n-1)} < \gamma < \min\left\{\frac{n-c-2}{2},1\right\}.
    \]
    This can be done because $c< n-\left(3-\frac{1}{n-1}\right)$. Thus, calling $u:=(1+|\nabla f|^2)^\gamma$, and using that $f\in L^{\infty}(M)$, and \cref{lem:Bochner++}, we have, for some $\alpha>0$, 
    \[
    \Delta u \geq -\alpha u, \qquad \text{on $M\setminus S$}.
    \]\
    Moreover, by \cref{lem:Step2} we have, on $M\setminus S$, up to possibly enlarging $\alpha$,
    \[
    u\leq \alpha r^{-2\gamma},
    \]
    where $-2\gamma>2+c-n$, and $r(x):=\mathrm{d}(x,S)$. Since $f\in W^{1,2}(M)$ and $\gamma<1$ we have $u\in L^1(M)$. Thus we can apply \cref{LinftyResultForThm1} and deduce that $|\nabla f|\in L^{\infty}(M)$, as desired.
\end{proof}

\begin{proof}[Proof of \cref{thm:RCDcodim3}]
    Under the assumptions in \cref{thm:RCDcodim3}: The item (1) in \cref{thm:ToRCD} readily holds on the metric completion of $M\setminus S$; the item (2) in \cref{thm:ToRCD} follows because $(M,g)$ is uniformly equivalent to a smooth background metric, for which Rellich theorem holds; the item (4) in \cref{thm:ToRCD} follows from the Ricci lower bound on the regular part $M\setminus S$. In addition, by arguing verbatim as in \cite[Lemma 2.2]{DWWW} (which hold whenever $S$ has codimension greater than 2), the metric completion of $(M\setminus S,g)$ is a compact almost smooth metric measure space because $S$ has zero 2-capacity in it.
    
    Hence, according to \cref{thm:ToRCD}, the proof of \cref{thm:RCDcodim3} is concluded as soon as item (3) in \cref{thm:ToRCD} is verified. Finally, item (3) in \cref{thm:ToRCD} is provided by \cref{thm:LipschitzRepresentativeThm1}, as desired.
\end{proof}

\begin{remark}\label{rem:Better-on-Lambda}
    Arguing as in \cite[Lemma 2.7]{DWWW} one can slightly improve \cref{lem:Step2}. Indeed, one has that there is $0<\beta:=\beta(n,\Lambda)<1$ such that $|\nabla f|(x)\leq Cd(x,S)^{-1+\beta}$ for every $x$ such that $d(x,S)\leq 1$. Moreover, as $\Lambda\to 1^+$, $\beta\to 1^-$ (as a consequence of \cite[Theorem 1.5]{ByunWang}).

    Thus, arguing verbatim as in the proof of \cref{thm:LipschitzRepresentativeThm1}, and using the notation therein, we have the following. Fix $\frac{n-2}{2(n-1)}<\gamma<1$. There exists $\alpha>0$ such that $\Delta u\geq -\alpha u$ on $M\setminus S$, and $u\leq \alpha r^{2\gamma(-1+\beta)}$, where $\beta:=\beta(n,\Lambda)$ is as above. Now notice that if 
    $$
    c<n-\left(3-\frac{1}{n-1}-\frac{n-2}{n-1}\beta\right),
    $$ 
    then we can choose $\gamma$ so that it also satisfies $2\gamma(-1+\beta)>2+c-n$. Hence, a small variation of the proof of \cref{thm:LipschitzRepresentativeThm1}, as described above, gives the following improvement of \cref{thm:RCDcodim3}. 
    
    Fix $\Lambda>1$ and $n\geq 3$ (the case $n=2$ is already optimal). There exists $0< \xi(\Lambda,n)\leq \frac{n-2}{n-1}$ such that the following holds. Let $(M^n,g)$ be a closed $L^\infty$-Riemannian manifold whose metric $g$ satisfies \eqref{eqn:LambdaDef}. Assume that $g$ is smooth on $M\setminus S$, $S\subset M$ is closed with codimension greater than 
    $$
    c_{n,\Lambda}:=3-\frac{1}{n-1}-\xi(\Lambda,n)\geq 2,
    $$
    and $\mathrm{Ric}\geq K$ on $M\setminus S$. Hence, the metric completion of $(M\setminus S,g)$ is an $\mathrm{RCD}(K,n)$ space. Notice that, as $\Lambda\to 1^+$, $c_{n,\Lambda}\to 2^+$.

    When the metric $g\in C^0$, the content of this Remark was already observed in \cite[Remark 2.8]{DWWW}.
\end{remark}
\begin{remark}\label{rem:PossiblyMoreGeneralResults}
   Note that  in the proof of \cref{lem:ConstructionvLinfty} we are using that $(M^n,g)$ is an $L^\infty$-Riemannian manifold only to get the estimates on the Green function in \eqref{eqn:EstG}. In particular, the upper estimate on $G$ is used to show that $v\in L^1(\Omega)$, while the lower estimate is used to prove the decay estimate \eqref{eqn:Asymptoticsv}. In the proof of \cref{LinftyResultForThm1} we are using that $(M^n,g)$ is an $L^\infty$-Riemannian manifold only to invoke \cref{lem:ConstructionvLinfty} and when we use the $L^\infty$-$L^1$ bound for sub-harmonic functions, which works for arbitrary infinitesimally Hilbertian PI spaces.
   
   Hence, the techniques of \cref{lem:ConstructionvLinfty} and \cref{LinftyResultForThm1} seem flexible enough to show similar results for metric measure spaces that are infinitesimally Hilbertian, PI spaces, for which one has well-behaved two-sided bounds for the Dirichlet Green kernel. 

   On the other hand, in \cref{thm:LipschitzRepresentativeThm1}, the assumption on the codimension being greater than $3-\frac{1}{n-1}$ cannot be weakened with the present proof for arbitrary $L^\infty$-metrics, apart from what is discussed in \cref{rem:Better-on-Lambda}. In fact, the obstruction comes from \cref{lem:Bochner++}, and in particular from \eqref{eqn:PDENew}, which, for $\beta\leq \frac{n-2}{2(n-1)}$, becomes useless; and also from  the upper gradient estimate in \cref{lem:Step2}. It would be interesting to understand whether:
   \begin{enumerate}
       \item with a different proof, one can show \cref{thm:LipschitzRepresentativeThm1} with the weaker assumption that $S$ has codimension greater than $2$, which would be optimal (see \cref{quest:RCD-if-affirmative});
       \item the techniques of this section can prove (variations of) \cref{thm:LipschitzRepresentativeThm1} in settings that are more general than that of an $L^\infty$-Riemannian manifold. E.g. for a properly chosen subclass of infinitesimally Hilbertian PI spaces that are almost smooth and for which one has good two-sided controls on the Dirichlet Green kernel. For some results in this direction, see \cite[Section 6]{BZ26}.
   \end{enumerate}
   Related to the second item above, it would be interesting to explore if the techniques in this section can extend some of the results in \cite{HondaSun}.
\end{remark}

\section{Subpower loss of top-dimensional splitting maps}\label{sec:gram-defect}
The aim of this section is to prove the following Theorem \ref{thm:main}, showing that top-dimensional splitting maps persist at smaller scales, with an arbitrarily small power loss. 
In this section all constants denoted by $C$ may change from line to line (occasionally,
dependence on other constants will be tracked). For a symmetric matrix $A$, $|A|$ denotes its Hilbert--Schmidt norm. The following result is trivial when $n=1$, so we will always assume that $n\geq 2$. 

\begin{theorem}\label{thm:main}
For every integer $n\geq 1$ and every $\kappa>0$, there are $
 \varepsilon=\varepsilon(n,\kappa)>0$ and $
 C=C(n,\kappa)>0$
with the following property. Let $(X,\d,\Hn,x)$ be a pointed
$\operatorname{RCD}(-\eta,n)$ space, where $0\leq\eta\leq1$, and let
\[
 u=(u^1,\ldots,u^n):B_2(x)\to \R^n
\]
have harmonic components. Denote by $I$ the $n\times n$ identity matrix and put
\begin{equation}\label{eq:Gudefn}
 G_u:=\left(\langle\nabla u^i,\nabla u^j\rangle\right)_{i,j=1}^n,
 \qquad
 D:=\fint_{B_1(x)}|G_u-I|\,\mathrm{d}\Hn.
\end{equation}
If $D+\eta\leq \varepsilon$, then
\begin{equation}\label{eq:main-matrix}
 \fint_{B_r(x)}|G_u-I|\,\mathrm{d}\Hn
 \leq Cr^{-\kappa}(D+\eta)
 \qquad\forall 0<r<\frac14.
\end{equation}
\end{theorem}

The result is new even in the case of smooth manifolds with nonnegative Ricci curvature. It is worth remarking that by Cheeger--Colding theory~\cite{ChCo0}, if $u$ is a sufficiently good splitting function at scale 1, i.e., $D+\eta$ above is small, then any small ball $B_r(x)$ is Gromov-Hausdorff close to the Euclidean ball, and thus it admits good splitting functions. The content of Theorem~\ref{thm:main} is to quantify the scales at which the same function $u$ can still serve as a good splitting function, even if it is not the optimal one. 
\smallskip

Notice that from \cref{thm:main} (and \cref{rem:HessianNotNecessary}) we immediately get the following corollary. For the basic definitions and properties related to splitting maps, we refer to the Appendix. 
\begin{corollary}\label{cor:Propagation-Splitting!}
    Let $n\geq 1$ be an integer. For every $\kappa>0$ there exist
$0<\varepsilon_0=\varepsilon_0(n,\kappa)<1$ and
$C=C(n,\kappa)>0$ with the following property. 

Let $0\leq \eta\leq \varepsilon\leq\varepsilon_0$. Suppose that
$(X,\mathrm d,\mathcal H^n)$ is an $\operatorname{RCD}(-\eta,n)$ space, and that
\[
u=(u_1,\ldots,u_n):B_2(x)\to\mathbb R^n
\]
is an $\varepsilon$-splitting map. Then $u\big|_{B_r(x)}:B_r(x)\to\mathbb R^n$
is a $Cr^{-\kappa}\varepsilon$-splitting map for every
$0<r<1/4$.
\end{corollary}
The previous result should be compared to the geometric transformation theorem in \cite[Theorem 7.2]{ChJiNa} (see also \cite[Proposition 3.13]{BNSInv}). In our case, we are not modifying the splitting map with a triangular transformation $T_r$. Instead, we get a quantification on the constant $C(r)$ for which the same map is $C(r)$-splitting at scales $0<r<1/4$. Notice that the subpower loss $r^{-\kappa}$ in the statement above is necessary, as it can be readily seen considering flat cones whose cone angle tends to $2\pi$. Moreover, we stress that the linear dependence on $\varepsilon$, which is crucial in our proof, is also optimal; this can be seen by considering, e.g., the map $u(x):=\sqrt{1-\varepsilon}x$ on $\mathbb R^n$, for small $\varepsilon$.

\subsection{Preliminary results}

We start with some preliminary results which will be useful to prove Theorem \ref{thm:main}. Throughout this section we assume that the reader is familiar with Gigli's calculus \cite{GigliNDG}. We also recommend the self-contained exposition in \cite[Section 2.1 and 2.2]{GigliViolo}. Occasionally, we will refer to precise statements of \cite{GigliNDG, GigliViolo} for the reader's convenience. Test functions (resp., test vector fields) will be denoted by $\mathrm{TestF},\mathrm{TestV}$.

\subsubsection{Basic definitions}

On a metric measure space $(X,\mathrm{d},\mathfrak{m})$ we use
\[
 \TestF(X):=\bigl\{f\in D(\Delta)\cap L^\infty(\mathfrak m):
 |\nabla f|\in L^\infty(\mathfrak m),\
 \Delta f\in W^{1,2}(X)\bigr\},
\]
where $D(\Delta)$ denotes the $L^2$-domain of the Laplacian.  The
space of test vector fields is
\[
 \TestV(X):=\left\{\sum_{i=1}^{\ell}g_i\nabla f_i:
 \ell\in\mathbb N,\ f_i,g_i\in\TestF(X)\right\}.
\]
See \cite{GigliNDG}. Similar definitions can be given locally: see, e.g., \cite[Section 2]{GigliViolo}.
\medskip

Let $(X_i,\mathrm{d}_i,\mathfrak{m_i})\to (X,\mathrm{d},\mathfrak{m})$ in the pmGH sense. For $q>1$, we say that \emph{$f_i\to f$ strongly in $L^q$} iff $f_i\mathfrak{m}_i\hookrightarrow f\mathfrak{m}$ weakly-* in a realization of the pmGH convergence and $\|f_i\|_{L^q(\mathfrak{m}_i)}\to \|f\|_{L^q(\mathfrak{m})}$. With $q=1$, we say that \emph{$f_i\to f$ strongly in $L^1$} iff $\sigma\circ f_i\to \sigma\circ f$ in $L^2$-strong, where $\sigma(x):=\mathrm{sgn}(x)\sqrt{|x|}$. Strong convergence in $W^{1,2}$ means: strong convergence in $L^2$ and $\int|\nabla f_i|^2\mathfrak{m}_i\to \int |\nabla f|^2\mathfrak{m}$. See \cite[Section 1.3]{AmbrosioHonda}. Local definitions are given similarly, in the obvious way: see \cite[Section 4]{AmbrosioHondaSpectral}.

\subsubsection{From $\varepsilon$-splitting maps to volume control} Write $v_{K,n}(r)$ for the volume of the radius-$r$ ball in the
simply-connected $n$-dimensional model with Ricci curvature $K$. 
Notice that if $\eta_k\to 0$, then
\begin{equation}\label{eq:Phi-expansion}
\sup_{r\in (0,1]} \left|\frac{v_{-\eta_k,n}(r)}{\omega_n r^n}-1\right|\to 0.
\end{equation}
For an $\operatorname{RCD}(-\eta,n)$ space $(X,\mathrm{d},\mathcal{H}^n)$, Bishop--Gromov says that $
 r\mapsto\frac{\Hn(B_r(y))}{v_{-\eta,n}(r)}$
is nonincreasing on $(0,\infty)$.  Moreover the Bishop density is at most
one at every point \cite[Corollary~2.14]{DePhilippisGigli}.
Consequently, for $0<s\leq t$,
\begin{equation}\label{eq:negative-BG-exact}
 \frac{\mathcal{H}^n(B_s(y))}{\omega_n s^n}\geq \frac{\mathcal{H}^n(B_t(y))}{\omega_n t^n}
 \frac{\frac{v_{-\eta,n}(s)}{\omega_n s^n}}{\frac{v_{-\eta,n}(t)}{\omega_n t^n}},
 \qquad
 \frac{\mathcal{H}^n(B_s(y))}{\omega_n s^n}\leq \frac{v_{-\eta,n}(s)}{\omega_n s^n}.
\end{equation}

\begin{lemma}\label{lem:almost-splitting}
For each $n\geq2$, set $\rho:=1/(100n)$.  There is a function
$\Psi(\kappa)\to 0$ as $\kappa\to 0$ with the following
property.  Let $(Y,\mathrm{d}_Y,\m_Y,y)$ be $\operatorname{RCD}(-\kappa,n)$,
$0<\kappa<1$, and let $u=(u_1,\ldots,u_n)$. Let $I$ denote the $n\times n$ identity matrix.  Assume
\begin{align}
 &u\text{ is harmonic on }B_6(y),\qquad
 \max_i\|\nabla u_i\|_{L^\infty(B_{6}(y))}\leq C(n),\label{eq:as-lip}\\
 &\fint_{B_{6}(y)}|G_u-I|\,\mathrm{d}\m_Y
 +\fint_{B_{6}(y)}\sum_i|\Hess u_i|^2\,\mathrm{d}\m_Y\leq \kappa.
 \label{eq:as-errors}
\end{align}
Then
\begin{equation}\label{eq:as-gh}
 d_{GH}\left((\overline B_{2\rho}(y),y),
 (\overline B_{2\rho}(0^n),0^n)\right)\leq\Psi(\kappa).
\end{equation}
\end{lemma}

\begin{proof}
After normalizing
the reference measures by their masses on the unit balls this comes from
\cite[Theorem 3.8(ii)]{BNSInv}; see also 
\cite[Proposition 1.5]{BruePasqualettoSemola}.  
\end{proof}

\begin{lemma}\label{lem:volume-recovery}
Suppose $(X_k,\mathrm{d}_k,\Hn,x_k)$ is a sequence of pointed
$\operatorname{RCD}(-\eta_k,n)$ spaces, $n\geq2$, and $u_k$ satisfies
\eqref{eq:as-lip} on $B_6(x_k)$ and \eqref{eq:as-errors} with right-hand
side $\kappa_k$.
Assume $\kappa_k+\eta_k\to0$.  Then, for the fixed
$\rho$ in Lemma~\ref{lem:almost-splitting},
\begin{equation}\label{eq:uniform-volume-recovery}
 \sup_{0<s\leq\rho}
 \left|\frac{\Hn(B_s(x_k))}{\omega_ns^n}-1\right|\to 0.
\end{equation}
\end{lemma}

\begin{proof}
Apply Lemma~\ref{lem:almost-splitting} with
$\max\{\kappa_k,\eta_k\}$ in place of its parameter.  The centered closed
$\rho$-balls converge in pointed GH distance to the $\rho$-Euclidean ball. Volume convergence (\cite[Theorem 1.3]{DePhilippisGigli}) gives $
 \Hn(B_{\rho}(x_k))/\omega_n\rho^n\to 1$.
Using \eqref{eq:negative-BG-exact} and \eqref{eq:Phi-expansion}, the latter convergence to $1$ gives
\eqref{eq:uniform-volume-recovery}, as desired.
\end{proof}

\begin{lemma}
    The quadratic form $G_u$ in \eqref{eq:Gudefn} satisfies the estimate
    \begin{equation}\label{eq:positive-Gram-bound}
 0\leq G_u\leq(1+C(n)(D+\eta))I
 \qquad \Hn\text{-a.e. on }B_{7/8}(x).
\end{equation}
\end{lemma}
\begin{proof}
For $a\in\R^n$, define
\[
 q_a:=a^TG_ua=|\nabla(a\cdot u)|^2.
\]
For $0\leq \eta\leq 1$,  the localized improved Bochner inequality \cite[Proposition~2.14]{GigliViolo} on an
$\operatorname{RCD}(-\eta,n)$ space gives, as measures on $B_1(x)$,
\begin{equation}\label{eqn:Boundqa}
 \frac12\mathbf{\Delta}q_a+\eta q_a\Hn
 \geq |\Hess(a\cdot u)|^2\Hn\geq0.
\end{equation}
Moreover we know that $q_a\in W^{1,2}_{\loc}$. We claim that we have

\begin{equation}\label{eqn:ToShowToShow}
 \mathop{\rm ess\,sup}_{B_{7/8}(x)}q_a
 \leq |a|^2(1+C(n)(D+\eta))
\end{equation}
Using rational vectors first, and then
density in $a$, \eqref{eqn:ToShowToShow} will imply the required estimate \eqref{eq:positive-Gram-bound}.

The estimate \eqref{eqn:ToShowToShow} can be obtained directly from heat kernel estimates, arguing verbatim as in the proof of Item (2) of \cref{lem:PropertiesHarmonic} in the Appendix. Let us sketch the argument here, referring to the proof of \cref{lem:PropertiesHarmonic} for complete details. Let $\varphi\in\TestF(X)$ such that
\[
 0\leq\varphi\leq1,\qquad
 \varphi\equiv1\ \text{on }B_{15/16}(x),\qquad
 \supp\varphi\Subset B_{31/32}(x),\qquad
 |\nabla\varphi|+|\Delta\varphi|\leq C(n),
\]
whose existence follows from \cite[Proposition~2.13]{GigliViolo} (see also \cite[Lemma 3.1]{MondinoNaber}), and put $h:=\varphi^2$ and $\mathcal A:=\supp(|\nabla h|+|\Delta h|)$. If $y\in B_{7/8}(x)$, then $\dist(y,\mathcal A)\geq1/16$. Let $p_t$ denote the heat kernel. 
Set $w_a:=q_a-|a|^2$. For every Lebesgue point $y\in B_{7/8}(x)$ of $h w_a$, define
\[
 F_y(t):=\int_X h(z)w_a(z)p_t(y,z)\,\mathrm{d}\Hn(z).
\]
We have, following verbatim the proof of Item (2) of \cref{lem:PropertiesHarmonic}, that for a.e. $t\in (0,1)$,
\[
 F_y'(t)+2\eta F_y(t)
 \geq-C(n)|a|^2\left(\eta+D\omega(t)\right),
\]
where $\omega(t)=t^{-(n+1)/2}e^{-c(n)/t}$, and $|F_y(1)|\leq C(n)|a|^2 D$. Therefore, integrating the preceding differential inequality from $1$ to $t$,
\[
 F_y(t)\leq e^{2\eta(1-t)}F_y(1)
 +C(n)|a|^2\int_t^1e^{2\eta(s-t)}
                 \left(\eta+D\omega(s)\right)\,\mathrm{d}s.
\]
Finally, $F_y(t)=P_t(hw_a)(y)\to h(y)w_a(y)=w_a(y)$ as $t\to 0$ at $\Hn$-a.e. such $y$. Since $\eta\leq1$ and $\omega\in L^1(0,1)$, it follows that
\[
 \mathop{\rm ess\,sup}_{B_{7/8}(x)}w_a
 \leq C(n)|a|^2(D+\eta),
\]
from which \eqref{eqn:ToShowToShow} follows, as desired.
\end{proof}

For $0<r\leq1$, Bishop--Gromov comparison and $\eta\leq1$
imply
\begin{equation}\label{eq:crude-volume-bound}
 E(r):=\fint_{B_r(x)}|G_u-I|\,\mathrm{d}\Hn
 \leq
 \frac{v_{-\eta,n}(1)}{v_{-\eta,n}(r)}D
 \leq C(n)r^{-n}D
 \leq C(n)r^{-n}(D+\eta).
\end{equation}
If $\kappa\geq n$, this already proves
\eqref{eq:main-matrix}.  We henceforth assume 
\begin{equation}\label{eqn:0-kappa-n}
    0<\kappa<n.
\end{equation}

\subsection{The proof of Theorem~\ref{thm:main}: Strategy and first reductions}
In this section we will begin the proof of Theorem~\ref{thm:main}. There will be a number of intermediate steps. We argue by contradiction, so suppose that the theorem is false for a fixed choice of $n\geq 2$ and $\kappa>0$.  Then we can choose a sequence of data
\begin{equation}\label{eqn:TheSequence}
 (X_k,\mathrm{d}_k,\Hn,x_k,u_k),\qquad
 b_k:=D_k+\eta_k\leq k^{-1},
\end{equation}
for which (recall the definition of $E_k$ as in \eqref{eq:crude-volume-bound})
\[
 \Lambda_k:=
 \sup_{0<r\leq 1/4}\frac{r^\kappa E_k(r)}{b_k}
 \to\infty,
\]
as $k\to\infty$. 
The case $b_k=0$ has $D_k=\eta_k=0$ and is immediate, so without loss of generality, the chosen
sequence has $b_k>0$. We give a rough sketch of how this will lead to a contradiction. On first reading it may be helpful to imagine that the $(X_k, \mathrm{d}_k)$ were smooth Riemannian manifolds, since in the RCD setting there are considerable additional technical complications. 
\begin{enumerate}
    \item In Section~\ref{sec:rescaling} we will rescale the metrics $\mathrm{d}_k$ by the maximizing radii in the definition of $\Lambda_k$, and show that these rescaled metric spaces converge to the Euclidean space, while suitable rescalings and rotation of the $u_k$ converge to the identity map $v$ on $\mathbb{R}^n$.
    \item After rescaling, the defect $|G_k - I|$ has average $a_k$ on the unit ball. We consider the normalized defect $H_k := a_k^{-1}(G_k - I)$, and show that $H_k \to H$ in $L^q_{\mathrm{loc}}$ for a symmetric matrix valued function $H$ on $\mathbb{R}^n$, which satisfies
    \begin{equation}\label{eq:Hprop}
\fint_{B_1(0)} |H|\, \mathrm{d}\mu_\infty = 1, \qquad \fint_{B_R(0)} |H| \, \mathrm{d}\mu_\infty \leq 2R^{-\kappa}, \quad \forall R\geq 1, 
    \end{equation}
    where $\mu_\infty=\omega_n^{-1}\mathcal{L}^n$ is the normalized Lebesgue measure on $\mathbb R^n$.
    A consequence of the Bochner formula is that $\Delta H = \omega_n \nu$ for a positive semidefinite matrix valued measure $\nu$. Because of our normalization using the ``worst'' radius $r_k$, we have 
    \begin{equation}\label{eq:3} \tr \nu(B_R(0)) \leq C(n, \kappa) R^{n-2-\kappa}, 
    \end{equation} 
    for all $R \geq 1$. 
    \item The crucial observation is that the $H_k$ satisfy a conservation law (coming from the fact that the stress-energy tensor of a harmonic map is divergence-free) which can be passed to the limit as $k\to \infty$. As a result $H$ satisfies the equation 
    \[ \mathrm{div}\left( H - \frac{1}{2}(\tr H)I\right) = 0, \]
    on $\mathbb{R}^n$ in the distribution sense. Consequently the matrix valued measure $\nu$ also satisfies the same equation. 
    \item Using this equation and the positive semidefiniteness of $\nu$ leads to a monotonicity formula for $\nu$:
    \[ R^{2-n} \tr \nu(B_R(0)) \text{ is increasing with }R. \]
    Combined with the bound \eqref{eq:3}, this implies that $\nu = 0$.
    \item It follows that the components of $H$ are harmonic functions, but this contradicts \eqref{eq:Hprop}. 
\end{enumerate}

\subsubsection{Rescaling by the maximizing radius}\label{sec:rescaling}

Take $G_k:=G_{u_k}$ and let
\[
 E_k(r)=\fint_{B_r(x_k)}|G_k-I|\,\mathrm{d}\Hn,
 \qquad
 \Lambda_k=\sup_{0<r\leq 1/4}
       \frac{r^\kappa E_k(r)}{b_k}.
\]
The supremum is finite because \eqref{eq:positive-Gram-bound} gives
$E_k(r)\leq C(n)$ for $0<r\leq1/4$. Choose $r_k\in(0,1/4]$ such that
\begin{equation}\label{eq:approx-maximizer}
 \frac{r_k^\kappa E_k(r_k)}{b_k}\geq\frac12\Lambda_k,
 \qquad \text{and define} \qquad  a_k:=E_k(r_k).
\end{equation}
We can assume $\Lambda_k>0$, so $a_k>0$, for all $k$. The bound \eqref{eq:crude-volume-bound} gives
$
 \frac12\Lambda_k\leq C(n)r_k^{\kappa-n},$
so $r_k\to0$, since we assumed $\kappa < n$, and we know that $\Lambda_k\to\infty$.

The definition of $\Lambda_k$ and
\eqref{eq:approx-maximizer} give
\begin{equation}\label{eqn:DecayEkRk}
 E_k(Rr_k)\leq 2a_kR^{-\kappa}
 \quad \forall R>0\text{ such that }Rr_k\leq\frac14,
\end{equation}
and
\begin{equation}\label{eq:b-over-a}
 \frac{b_k}{a_k}\leq\frac{2r_k^\kappa}{\Lambda_k}\to 0.
\end{equation}
For all large $k$, \eqref{eq:positive-Gram-bound} also gives
\begin{equation}\label{eqn:akBounded}
    a_k\leq C(n).
\end{equation}

We rescale the metric and the map by
\begin{equation}\label{eq:first-rescaling}
 \mathrm{d}'_k:=r_k^{-1}\mathrm{d}_k,
 \qquad
 \widehat\m_k:=r_k^{-n}\Hn=\mathcal H^n_{\mathrm{d}'_k},
 \qquad
 \mu_k:=\frac{\widehat\m_k}{\widehat\m_k(B_1^{\mathrm{d}'_k}(x_k))},
 \qquad
 v_k:=\frac{u_k-u_k(x_k)}{r_k}.
\end{equation}
Thus $\mu_k(B_1^{\mathrm{d}'_k}(x_k))=1$, and the rescaled space is
$\operatorname{RCD}(-\lambda_k,n)$ with
\begin{equation}\label{eq:lambda}
 \lambda_k:=\eta_kr_k^2\to 0.
\end{equation}
The $G_k$ matrix defined from the $v_k$'s with respect to the new metric coincides with the original one defined from the $u_k$'s, as it can be readily checked.  The map $v_k$ is defined and harmonic on
$B_{2/r_k}^{\mathrm{d}'_k}(x_k)$. 

\begin{lemma}
    The rescaled spaces converge in the pmGH sense:
    \[ (X_k, \mathrm{d}_k', \mu_k, x_k) \to (\mathbb{R}^n, \mathrm{d}_E, \mu_\infty, 0), \]
    where $\mathrm{d}_E$ is the Euclidean metric, and $\mu_\infty := \omega_n^{-1}\mathcal{L}^n$ is the normalized Lebesgue measure. 
\end{lemma}
\begin{proof}
The proof follows from a standard application of Lemma \ref{lem:volume-recovery} and the volume convergence theorem. We spell out the details for the reader's convenience. 

We verify the hypotheses needed for
Lemma \ref{lem:volume-recovery}. We use the notation and setting explained before the statement of the Lemma. Choose a test cutoff $\chi_k$ compactly supported in $B_1^{\mathrm{d}_k}(x_k)$ and
satisfying
\[
 0\leq \chi_k\leq 1,\qquad
 \chi_k\equiv 1\quad\text{on }B_{7/8}^{\mathrm{d}_k}(x_k),\qquad
 |\nabla\chi_k|\leq C(n),\qquad
 |\Delta\chi_k|\leq C(n).
\]
Such cutoffs exist by \cite[Proposition~2.13]{GigliViolo} (see also
\cite[Lemma~3.1]{MondinoNaber}).  The localized improved Bochner
inequality \cite[Proposition~2.14]{GigliViolo}, tested against $\chi_k$, gives, for every $k$ and $i=1,\ldots,n$
\begin{align*}
 \int\chi_k|\Hess u_k^i|^2\,\mathrm{d}\Hn
 &\leq
 \frac12\int(G_{k,ii}-1)\Delta\chi_k\,\mathrm{d}\Hn
 +\eta_k\int\chi_kG_{k,ii}\,\mathrm{d}\Hn\\
 &\leq C(n)b_k\Hn(B_1^{\mathrm{d}_k}(x_k)).
\end{align*}
Bishop--Gromov
comparison therefore gives
\begin{equation}\label{eq:hessian-smallness-fixed-ball}
 \fint_{B_{3/4}^{\mathrm{d}_k}(x_k)}
 |\Hess u_k^i|^2\,\mathrm{d}\Hn
 \leq C(n)b_k\to 0.
\end{equation}
On the same ball $B_{3/4}^{\mathrm{d}_k}(x_k)$, $|G_{u_k}-I|$ has average at most $C(n)D_k$ and the
gradients are uniformly bounded by \eqref{eq:positive-Gram-bound}. 
Up to a harmless constant rescaling of the metrics $\mathrm{d}_k$, we can hence apply Lemma \ref{lem:volume-recovery}. Thus, there exists $\rho_0>0$ such that 
\begin{equation}\label{eq:volume-recovery-original}
 \sup_{0<s\leq\rho_0}
 \left|\frac{\Hn(B_s^{\mathrm{d}_k}(x_k))}{\omega_ns^n}-1\right|\to 0.
\end{equation}

For every $R>0$, since
$Rr_k\to0$, we have
\begin{equation}\label{eq:all-fixed-volume}
 \widehat\m_k(B_R^{\mathrm{d}'_k}(x_k))
 =r_k^{-n}\Hn(B_{Rr_k}^{\mathrm{d}_k}(x_k))\to\omega_nR^n.
\end{equation}
Now, pointed RCD compactness, stability of the RCD condition, and \cite[Theorem 1.2]{DePhilippisGigli} give
\[
 (X_k,\mathrm{d}'_k,\widehat\m_k,x_k)\to
 (X_\infty,\mathrm{d}_\infty,\widehat\m_\infty,x_\infty)
\]
in the pmGH sense, where the limit is a noncollapsed
$\operatorname{RCD}(0,n)$ space, namely $\widehat\m_\infty = \mathcal{H}^n_{\mathrm{d}_\infty}$.  Equation \eqref{eq:all-fixed-volume} and
volume convergence \cite[Theorem 1.3]{DePhilippisGigli} yield
\[
 \widehat \m_\infty(B_R^{\mathrm{d}_\infty}(x_\infty))=\omega_nR^n
 \qquad\text{for every }R>0.
\]
Finally, the rigidity in \cite[Corollary 1.7]{DePhilippisGigli}, gives
\begin{equation}\label{eq:Euclidean-pmGH}
 (X_k,\mathrm{d}'_k,\mu_k,x_k)\to
 \left(\R^n,\mathrm{d}_{\mathrm E},\mu_\infty,0\right),
 \qquad \text{where} \qquad 
 \mu_\infty:=\omega_n^{-1}\mathcal L^n.
\end{equation}
Indeed, \eqref{eq:all-fixed-volume} at $R=1$ gives
$\widehat\m_k(B_1(x_k))\to\omega_n$.  Dividing the unnormalized
noncollapsed convergence by these masses therefore produces precisely the limit measure
$\mu_\infty=\omega_n^{-1}\mathcal L^n$, as desired.
\end{proof}
\begin{remark}\label{rem:ExFootnote}
    Fix $L>0$.  For all large $k$, Bishop--Gromov comparison, and $\mu_k(B_1^{\mathrm{d}'_k}(x_k))=1$ imply that there are constants $0<c_L\le C_L<\infty$ such that
\[
c_Ls^n\le\mu_k(B_s^{\mathrm{d}'_k}(x))\le C_Ls^n,
\]
for every $x\in B_{4L}^{\mathrm{d}'_k}(x_k)$ and $0<s<L$.  Hence doubling/Ahlfors constants, the Poincaré inequality constant from \cite{Rajala} (and then all Poincaré--Sobolev inequalities constants derived from the local Poincaré inequality, as in \cite{HK00}), cutoff, mean-value, and Caccioppoli constants are uniform for sufficiently large $k$ on balls with bounded radii. 
Moreover pmGH convergence to Euclidean space gives a uniformly positive
measure for $B_{2L}^{\mathrm{d}'_k}(x_k)\setminus B_{3L/2}^{\mathrm{d}'_k}(x_k)$.  Therefore
the Dirichlet--Sobolev constants appearing in Lemma
\ref{lem:measure-data-rigorous} are uniform in $k$, for $k$ large, as well.
\end{remark}

We next consider the behavior of the rescaled harmonic functions $v_k$. 
\begin{lemma}
    Under the rescalings above the functions $v_k$ converge locally uniformly and strongly in $W^{1,2}_{\mathrm{loc}}$ to a map $v:\mathbb{R}^n \to \mathbb{R}^n$. We have $v(z) = Az$ for an orthogonal matrix $A$. 
\end{lemma}
\begin{proof}
Since $|\nabla u_k|$ is uniformly bounded on $B_{7/8}^{\mathrm{d}_k}(x_k)$ (in the un-rescaled metric), on every fixed rescaled ball we have uniform local
Lipschitz bounds for $v_k$.  Since $v_k(x_k)=0$, Arzelà--Ascoli first, and \cite[Theorem 4.4]{AmbrosioHondaSpectral} after,
provides a locally uniform and strongly $W^{1,2}_{\loc}$ convergent
subsequence
\begin{equation}\label{eq:v-limit}
 v_k\to v:\R^n\longrightarrow\R^n,
\end{equation}
and every component of $v$ is harmonic.  $W^{1,2}_{\mathrm{loc}}$-convergence of $v_k^{i}\pm v_k^j$, and the application of \cite[Theorem 4.4(3)]{AmbrosioHondaSpectral}, gives
\begin{equation}\label{eq:Gram-strong-first}
 \langle\nabla v_k^i,\nabla v_k^j\rangle
 \to
 \langle\nabla v^i,\nabla v^j\rangle
 \quad\text{strongly in }L^1_{\loc}.
\end{equation}

At every fixed
radius $R>0$, \eqref{eqn:DecayEkRk}, \eqref{eqn:akBounded},
\eqref{eq:Euclidean-pmGH}, and \eqref{eq:Gram-strong-first} give
\begin{equation}\label{eq:v-outer-decay}
 \fint_{B_R(0)}|G_v-I|\,\mathrm{d}\mu_\infty
 \leq C(n) R^{-\kappa}.
\end{equation}
Let us recall: since $|\nabla u_k|$ is uniformly bounded on $B_{7/8}^{\mathrm{d}_k}(x_k)$ (in the un-rescaled metric), on every fixed rescaled ball we have uniform local
Lipschitz bounds for $v_k$. Thus, the map $v$ is globally Lipschitz and harmonic on $\mathbb R^n$.
Therefore each first derivative of each component is a bounded entire
harmonic function, and Liouville's theorem gives $v(z)=Az+c$ for a
constant matrix $A$.  Since $v(0)=0$, $c=0$.  Letting $R\to\infty$ in
\eqref{eq:v-outer-decay} gives $AA^T=I$.  In particular $A\in O(n)$.
At radius one, \eqref{eq:Gram-strong-first} now implies
\begin{equation}\label{eq:a-to-zero}
 a_k=\fint_{B_1^{\mathrm{d}'_k}(x_k)}|G_{v_k}-I|\,\mathrm{d}\mu_k
 \to |AA^T-I|=0. \qedhere
\end{equation}
\end{proof}

We will compose all $v_k$ with the fixed orthogonal map $A^T$, so without loss of generality we can assume from now on that
\begin{equation}\label{eq:v-to-id}
 v_k\to \id_{\R^n}
 \quad\text{locally uniformly and strongly in }W^{1,2}_{\loc},
 \qquad a_k\to 0.
\end{equation}

\subsection{Strong \texorpdfstring{$L^q$}{Lq}
compactness of the defect and conclusion of the proof}\label{sec:poisson-compactness}

From now on, we will be working in the rescaled metrics $\mathrm{d}_k'$ defined in \eqref{eq:first-rescaling}. Hence, when we write $B_R(x_k)$ we mean $B_R^{\mathrm{d}_k'}(x_k)$.

\subsubsection{Limiting behavior of the defect}
Recall that we are working with the rescaled spaces as in \eqref{eq:first-rescaling}. Let us set
\begin{equation}\label{eq:def-Hk}
 H_k:=\frac{G_k-I}{a_k}.
\end{equation}
On every relatively compact ball on which the $v_k^i$ are harmonic, they belong to the local test class, so 
\cite[Proposition~3.1.3]{GigliNDG} shows that
$G_{k,ij}\in W^{1,2}_{\loc}$ and that its measure-valued Laplacian is
defined. We define the symmetric matrix-valued Radon measure $\nu_k$, which is defined on the exhausting balls $B_{2/r_k}^{\mathrm{d}_k'}(x_k)$,
entrywise by (recall \eqref{eq:lambda}, \eqref{eq:approx-maximizer} and \eqref{eq:first-rescaling} for the definitions of $\lambda_k,a_k,\mu_k$)
\begin{equation}\label{eq:def-nuk}
 (\nu_k)_{ij}:=\frac{1}{a_k}\mathbf{\Delta}(G_k)_{ij}
       +\frac{2\lambda_k}{a_k} (G_k)_{ij}\mu_k.
\end{equation}
The {localized improved Bochner inequality
\cite[Proposition~2.14]{GigliViolo}}
implies, for every
$c\in\mathbb R^n$,
\begin{equation}\label{eq:polarised-Bochner}
 c^T\nu_k c
 =\frac{1}{a_k}\mathbf{\Delta}|\nabla(c\cdot v_k)|^2+\frac{2\lambda_k}{a_k}|\nabla(c\cdot v_k)|^2\mu_k
 \geq\frac{2}{a_k} |\Hess(c\cdot v_k)|^2\mu_k.
\end{equation}
Therefore the $\nu_k$ are positive semidefinite as
matrix-valued measures.  In this section our goal is the following. 
\begin{lemma}\label{lem:EstimatesH}
There is a symmetric matrix valued function $H$, and a positive semidefinite matrix valued Radon measure $\nu$ on $\mathbb{R}^n$ such that we have 
\begin{equation}\label{eq:H-strong-L1}
 H_k\to H
 \qquad\hbox{strongly in }L^q_{\loc}\quad\text{for every } 1< q<\frac{n}{n-1},
\end{equation}
and 
\begin{equation}\label{eq:nu-weak}
 \nu_k\rightharpoonup\nu
 \qquad\hbox{locally weakly-* as measures}.
\end{equation}
These convergences are understood in a fixed proper metric space
into which the pointed measured Gromov--Hausdorff convergence is
isometrically realized. 

Moreover, we also have 
\[
\fint_{B_1(0)}|H|\,\mathrm{d}\mu_\infty=1,
\]
\[
\fint_{B_R(0)}|H|\,\mathrm{d}\mu_\infty\leq 2R^{-\kappa}, \qquad \tr\nu(B_R(0))\leq C(n,\kappa)R^{n-2-\kappa}
 \qquad\forall R\geq1,
\]
and the equality 
\begin{equation}\label{eqn:LimitingLaplacian}
   \Delta_{\mathbb R^n}H=\omega_n\nu
\end{equation}
holds distributionally on $\mathbb R^n$.
\end{lemma}

Before the proof we need some preliminary observations. Let $\beta_k=2\eta_k r_k^2/a_k$. 
The $H_k$'s satisfy
\begin{equation}\label{eq:corrected-Poisson}
 \ \mathbf{\Delta}H_k=\nu_k-\beta_kG_k\mu_k,
\end{equation}
and by \eqref{eq:b-over-a}, and the fact that $r_k\to 0$,
\begin{equation}\label{eq:beta-to-zero}
 0\leq\beta_k=\frac{2\eta_kr_k^2}{a_k}
 \leq2r_k^2\frac{b_k}{a_k}\to 0.
\end{equation}
As a consequence of \eqref{eq:polarised-Bochner}, we obtain the following inequality of scalar measures,
\begin{equation}\label{eq:entry-TV-by-trace}
 |(\nu_k)_{ij}|
 \leq\frac12\left((\nu_k)_{ii}+(\nu_k)_{jj}\right)
 \leq\tr\nu_k.
\end{equation}

We next obtain a mass upper bound for $\tr \nu_k$. Fix $R\geq1$.  For all sufficiently large $k$, the map
$v_k$ is defined on $B_{4R}(x_k)$.  Take
$\chi_{k,R}\in\TestF$ compactly supported in $B_{2R}(x_k)$
(\cite[Proposition~2.13]{GigliViolo}, see also
\cite[Lemma~3.1]{MondinoNaber}) such that
\begin{equation}\label{eq:Laplacian-cutoff}
 0\leq\chi_{k,R}\leq1,
 \qquad \chi_{k,R}=1\quad\hbox{on }B_R(x_k),
 \qquad |\Delta\chi_{k,R}|\leq C(n)R^{-2}.
\end{equation}
Here $k$ is large enough that $\lambda_kR^2\leq1$, so the cutoff
constant is uniform.
Using positivity of $\tr\nu_k$, integrating by parts, and
$|\tr(G_k-I)|\leq\sqrt n\,|G_k-I|$, we obtain
\begin{equation}\label{eq:mass-cutoff-step}
\begin{aligned}
 \tr\nu_k(B_R(x_k))
 &\leq\int\chi_{k,R}\,\mathrm{d}(\tr\nu_k)=\frac1{a_k}\int(\tr G_k-n)\Delta\chi_{k,R}\,\mathrm{d}\mu_k+\beta_k\int\chi_{k,R}\tr G_k\,\mathrm{d}\mu_k\\
 &\leq\frac{C(n)}{a_kR^2}
       \int_{B_{2R}(x_k)}|G_k-I|\,\mathrm{d}\mu_k
       +C(n)\beta_k\mu_k(B_{2R}(x_k)).
       \end{aligned}
\end{equation}
For every fixed $R$,
$\widehat\m_k(B_1(x_k))\to\omega_n$, together with the convergence to $\mathbb R^n$ gives
\begin{equation}\label{eq:normalised-volume-upper}
 \mu_k(B_{2R}(x_k))\leq C(n)R^n,
\end{equation}
for all sufficiently large $k$.  Combining
\eqref{eq:mass-cutoff-step}, \eqref{eq:normalised-volume-upper}, and
\eqref{eqn:DecayEkRk} at radius $2R$ gives (see also \eqref{eq:beta-to-zero})
\begin{equation}\label{eq:nuk-growth}
 \tr\nu_k(B_R(x_k))
 \leq C(n,\kappa)R^{n-2-\kappa}+C(n)\beta_kR^n,
 \qquad
 \limsup_{k\to\infty}\tr\nu_k(B_R(x_k))
 \leq C(n,\kappa)R^{n-2-\kappa}.
\end{equation}
In particular by \eqref{eq:entry-TV-by-trace} the measures $\nu_k$ have uniformly bounded total
variation on every fixed ball. Hence, by local weak-* compactness of Radon measures, there is a symmetric matrix-valued Radon measure $\nu$ on $\mathbb R^n$ such that $\nu_k\rightharpoonup \nu$ entrywise. By sending \eqref{eq:polarised-Bochner} to the limit, we get that $\nu$ is positive semidefinite too.   Also, taking the trace in
\eqref{eq:polarised-Bochner},
\begin{equation}\label{eq:Hess-Bochner-bound-RCD}
 \frac2{a_k}\sum_i|\Hess v_k^i|^2\mu_k\leq\tr\nu_k.
\end{equation}
We shall also repeatedly use the immediate consequence of
\eqref{eqn:DecayEkRk},
\begin{equation}\label{eq:Hk-L1-growth}
 \fint_{B_R(x_k)}|H_k|\,\mathrm{d}\mu_k\leq 2R^{-\kappa},
\end{equation}
which is true, fixed $R$, for all large $k$. We can now prove the Lemma.
\begin{proof}[Proof of Lemma \ref{lem:EstimatesH}]
Recall (see \eqref{eq:corrected-Poisson})
\[
 \nu_k-\beta_kG_k\mu_k=\mathbf{\Delta}H_k.
\]
On every fixed ball, \eqref{eq:positive-Gram-bound} bounds $G_k$,
while \eqref{eq:beta-to-zero}, \eqref{eq:entry-TV-by-trace}, and
\eqref{eq:nuk-growth} bound the total variation of every entry of
$\nu_k$.  Taking also into account \eqref{eq:Hk-L1-growth} and \cref{rem:ExFootnote}, Lemma \ref{lem:measure-data-rigorous} proved in the
Appendix (after a harmless renormalization of the measure) gives the
following: for every $1\leq q<\frac{n}{n-1}$ and $R>0$,
$\|H_k\|_{W^{1,q}(B_R(x_k))}$ is uniformly bounded above. The latter,
together with \cite[Proposition 3.39]{HondaLp} and a diagonal procedure,
immediately gives \eqref{eq:H-strong-L1}, as desired.  In addition,
cutoff localization and the liminf part of the Mosco convergence of
$q$-Cheeger energies \cite[Theorem 1.8.1]{AmbrosioHonda} give $H\in W^{1,q}_{\loc}(\mathbb R^n)$.

Strong $L^{q>1}_{\loc}$-convergence implies that
$\int_{B_R(x_k)}|H_k|\mathrm{d}\mu_k\to \int_{B_R(0)}|H|\mathrm{d}\mu_\infty$ for every
$R>0$. Indeed, fix $R$ and choose continuous radial cutoffs
$\chi^-_\varepsilon\le\mathbf1_{B_R}\le\chi^+_\varepsilon$ whose
difference is supported in the $\varepsilon$-annulus about
$\partial B_R$. By \cite[Corollary 3.28]{HondaLp} we have $|H_k|$
converges strongly in $L^q_{\mathrm{loc}}$ to $|H|$, with
$1<q<\frac{n}{n-1}$. This, together with
\cite[Proposition 3.27]{HondaLp}, gives
$\int |H_k|\chi_\varepsilon^{\pm}\mathrm{d}\mu_k \to
\int |H|\chi_{\varepsilon}^{\pm}\mathrm{d}\mu_\infty$.  The uniform $L^q$
bound gives uniform integrability, while pmGH convergence and
$\mu_\infty(\partial B_R)=0$ make the annular measures tend to zero
as $\varepsilon\downarrow0$.  Therefore, taking the limit as
$\varepsilon\to 0$, we have
\[
\int_{B_R(x_k)}|H_k|\,\mathrm{d}\mu_k\to
\int_{B_R(0)}|H|\,\mathrm{d}\mu_\infty,
\]
as claimed.

The latter, together with the convergence of the reference measures, allows us to pass
\eqref{eq:Hk-L1-growth} to the limit. In addition, since
$\mu_k(B_1(x_k))=1$ and, by definition of $a_k$, we have
 $\int_{B_1(x_k)}|H_k|\,\mathrm{d}\mu_k=1$,
we obtain
\begin{equation}\label{eq:H-limit-growth}
 \fint_{B_1(0)}|H|\,\mathrm{d}\mu_\infty=1,
 \qquad
 \fint_{B_R(0)}|H|\,\mathrm{d}\mu_\infty\leq 2R^{-\kappa}
 \quad \forall R>0.
\end{equation}
Using first radii $R$ with $\tr\nu(\partial B_R)=0$, and then outer
approximation by such radii,
\eqref{eq:nuk-growth} also gives (using also \eqref{eq:beta-to-zero})
\begin{equation}\label{eq:nu-limit-growth}
 \tr\nu(B_R(0))\leq C(n,\kappa)R^{n-2-\kappa},
 \qquad\forall R\geq1.
\end{equation}
It remains to prove \eqref{eqn:LimitingLaplacian}. We do this in two steps.
\smallskip

\textbf{Step 1.} Recall $
|G_k-I|=a_k|H_k|$.
The bound in \eqref{eq:positive-Gram-bound} gives the following. Fix $L>0$. For all sufficiently large $k$, on every ball
$B_L(x_k)$ we have $
 0\leq G_k\leq(1+C(n)b_k)I$. 
It follows that $\|G_k-I\|_{L^\infty(B_L(x_k))}\leq C(n)$ for $k$ large enough. We aim at showing that, for every $L>0$,
\begin{equation}\label{eq:quadratic-remainderPre}
 \int_{B_L(x_k)}\frac{|G_k-I|^2}{a_k}\,\mathrm{d}\mu_k
 =\int_{B_L(x_k)}|H_k||G_k-I|\,\mathrm{d}\mu_k
 \to 0.
\end{equation}

Fix $L,\epsilon>0$.  
First choose $s>0$
so small that
\begin{equation}\label{eqn:varepsilon}
 s\sup_k\int_{B_L(x_k)}|H_k|\,\mathrm{d}\mu_k<\epsilon/2,
\end{equation}
holds for every $k$ (up to passing to subsequences). This is possible due to \eqref{eq:Hk-L1-growth} and the fact that $\mu_k(B_L(x_k))$ is uniformly bounded above by Bishop volume bound. Now \eqref{eq:Hk-L1-growth}, Chebyshev's inequality, the fact that $\mu_k(B_L(x_k))$ is uniformly bounded above, and $a_k\to 0$ (see \eqref{eq:v-to-id})
give the following: for each $s>0$,
\begin{equation}\label{eq:F-in-measure}
 \mu_k\left(B_L(x_k)\cap\{|G_k-I|>s\}\right)
 \leq\frac{a_k}{s}\int_{B_L(x_k)}|H_k|\,\mathrm{d}\mu_k
 \to 0.
\end{equation}

Strong $L^q_{\mathrm{loc}}$ convergence with $q>1$ in \eqref{eq:H-strong-L1} gives, by definition, that $\int_{B_R(x_k)}|H_k|^q\mathrm{d}\mu_k \to \int_{B_R(0)}|H|^q\mathrm{d}\mu_\infty$ for every $R>0$.  Thus, using H\"older inequality, we have uniform integrability on $B_L(x_k)$; namely, for fixed $L>0$, there is $\delta>0$ such that if $A\subset B_L(x_k)$ satisfies
$\mu_k(A)<\delta$, then
$\int_A|H_k|\,\mathrm{d}\mu_k<\epsilon/(2C(n))$ for every $k$ (up to subsequences).  By
\eqref{eq:F-in-measure}, the set $A_k=\{|G_k-I|>s\}\cap B_L(x_k)$ has measure less
than $\delta$ for all sufficiently large $k$.  Splitting the integral below
over $A_k$ and its complement, using the latter uniform integrability and \eqref{eqn:varepsilon}, yields
\begin{equation}\label{eq:quadratic-remainder}
 \int_{B_L(x_k)}\frac{|G_k-I|^2}{a_k}\,\mathrm{d}\mu_k
 =\int_{B_L(x_k)}|H_k||G_k-I|\,\mathrm{d}\mu_k
 \leq\epsilon,
\end{equation}
for $k$ large enough. Given that $\epsilon>0$ is arbitrary we get \eqref{eq:quadratic-remainderPre}, as desired.
\medskip

\textbf{Step 2.} Let us finish the argument. Fix $\phi\in C_c^\infty(\mathbb R^n)$, say
$\supp\phi\subset B_S(0)$.  Choose $L>S+4$.  By the local uniform convergence \eqref{eq:v-to-id}, for all sufficiently large $k$ the
functions $\phi(v_k)$, $D\phi(v_k)$, and $D^2\phi(v_k)$ vanish throughout
the annulus $
 B_{2L}(x_k)\setminus B_L(x_k)$.

Choose $\zeta_k\in\TestF(X_k)$ such that
\[
 0\leq\zeta_k\leq1,\qquad
 \zeta_k=1\ \text{on }B_L(x_k),\qquad
 \supp\zeta_k\Subset B_{2L}(x_k),\qquad
 |\nabla\zeta_k|\leq C/L,
\]
as supplied by \cite[Proposition~2.13]{GigliViolo} (see also \cite[Lemma 3.1]{MondinoNaber}).  Choose also
$\chi_k\in\TestF(X_k)$ equal to one on $B_{2L}(x_k)$ and compactly
supported in $B_{3L}(x_k)$.  For $k$ sufficiently large,
$B_{3L}(x_k)$ is contained in the domain of harmonicity of $v_k$; thus,
extending the products by zero outside this domain, put $
 \widetilde v_k^j:=\chi_kv_k^j$.
Locality and the algebra properties (see, e.g., \cite[Equation (2.19)]{GigliViolo}) of $\TestF$ give
$\widetilde v_k^j\in\TestF(X_k)$.  Set
\[
 \Phi(z):=\phi(z)-\phi(0),\qquad
 \psi_k:=\zeta_k\Phi(\widetilde v_k)+\phi(0)\zeta_k.
\]
By
\cite[Lemma~3.1.4]{GigliNDG} we have
$\Phi(\widetilde v_k)\in\TestF(X_k)$, and consequently
$\psi_k\in\TestF(X_k)$.  Algebraically, $
 \psi_k=\zeta_k\phi(\widetilde v_k)$
and, since $\chi_k=1$ on a neighborhood of $\supp\zeta_k$, $
 \psi_k=\zeta_k\phi(v_k)$.
 
For all sufficiently large $k$, local uniform convergence of $v_k$ to
$\mathrm{id}_{\mathbb R^n}$ and the choice $L>S+4$ imply that
$\phi(v_k)$, $D\phi(v_k)$, and $D^2\phi(v_k)$ vanish on
$B_{2L}(x_k)\setminus B_L(x_k)$.  The product and chain rules (see \cite[Equation (2.3.24) and Lemma 3.1.4]{GigliNDG} and \cite[Propositions 4.1 and 4.3]{GigliViolo}) therefore
give, on $B_{2L}(x_k)$,
\[
 \begin{aligned}
 \Delta\psi_k
 &=\zeta_k\left[
      D^2\phi(v_k):G_k
      +\sum_{\alpha=1}^n
        \partial_\alpha\phi(v_k)\Delta v_k^\alpha
    \right]
   +\phi(v_k)\Delta\zeta_k 
   +2\left\langle\nabla\zeta_k,
       \sum_{\alpha=1}^n
       \partial_\alpha\phi(v_k)\nabla v_k^\alpha
     \right\rangle .
 \end{aligned}
\]
On $B_L(x_k)$ we have $\zeta_k=1$,
$\nabla\zeta_k=0$, and $\Delta\zeta_k=0$; on
$B_{2L}(x_k)\setminus B_L(x_k)$ the terms $\phi(v_k),D\phi(v_k),D^2\phi(v_k)$
vanish.  Since every $v_k^\alpha$ is harmonic, this proves the
identity
\begin{equation}\label{eq:chain-rule-test}
 \Delta\psi_k
 =D^2\phi(v_k):G_k
 =\Delta_{\mathbb R^n}\phi(v_k)
   +D^2\phi(v_k):(G_k-I)
 \qquad\text{on }B_{2L}(x_k).
\end{equation}
Since $\supp\psi_k\Subset B_{2L}(x_k)$, testing the $(i,j)$ entry of
$\mathbf{\Delta}H_k=\nu_k-\beta_kG_k\mu_k$ against $\psi_k$ produces no
boundary term and yields
\begin{equation}\label{eq:test-Hk-nuk}
\begin{aligned}
 \int\psi_k\,\mathrm{d}(\nu_k)_{ij}
 &=\int_{B_{2L}(x_k)}
      (H_k)_{ij}\Delta_{\mathbb R^n}\phi(v_k)\,\mathrm{d}\mu_k+
 \int_{B_{2L}(x_k)}
      (H_k)_{ij}
      \bigl[D^2\phi(v_k):(G_k-I)\bigr]\,\mathrm{d}\mu_k\\&\phantom{+}+
 \beta_k\int_{B_{2L}(x_k)}
      \psi_k(G_k)_{ij}\,\mathrm{d}\mu_k.
\end{aligned}
\end{equation}
The second integral in the right-hand-side above tends to zero by
\eqref{eq:quadratic-remainderPre}, while the last tends to zero by
\eqref{eq:beta-to-zero}, \eqref{eq:positive-Gram-bound}, and the fact that $\mu_k(B_{2L}(x_k))$ is uniformly bounded for $k$ large enough.  The first tends, by
\eqref{eq:H-strong-L1}, local uniform convergence of $v_k$ to $\mathrm{id}_{\mathbb R^n}$, and \cite[Proposition 3.27]{HondaLp} to
\[
 \int_{\mathbb R^n}H_{ij}(x)\Delta_{\mathbb R^n}\phi(x)
 \,\mathrm{d}\mu_\infty(x).
\]
Finally, the functions $\psi_k$ are equi-Lipschitz, have uniformly
bounded support, and converge uniformly in a common Gromov--Hausdorff
realisation to $x\mapsto\phi(x)$.  The local total-variation bound for
$(\nu_k)_{ij}$ (see \eqref{eq:entry-TV-by-trace} and \eqref{eq:nuk-growth}) and \eqref{eq:nu-weak} therefore imply
\[
 \int\psi_k\,\mathrm{d}(\nu_k)_{ij}\to 
 \int_{\mathbb R^n}\phi(x)\,\mathrm{d}\nu_{ij}(x).
\]
Passing to the limit in \eqref{eq:test-Hk-nuk}, for every $i,j$ and every
$\phi\in C_c^\infty(\mathbb R^n)$ (and recalling $\mu_\infty=\omega_n^{-1}\mathcal L^n$), proves
 $\Delta_{\mathbb R^n}H=\omega_n\nu$,
as desired.
\end{proof}

\subsubsection{Using the conservation law}
In this section we exploit the divergence free property of the stress-energy tensor of the harmonic maps $v_k$ to prove the corresponding conservation law for $H$. The goal is to prove the following.
\begin{lemma}\label{lem:StressIdentity}
With the notation as in Lemma \ref{lem:EstimatesH}, we have
    \begin{equation}
 \operatorname{div}\left(H-\frac12(\operatorname{tr}H)I\right)=0
 \quad\text{in }\mathcal D'(\mathbb R^n).
 \label{eq:linear-gauge}
\end{equation}
\end{lemma}

Before giving the proof we give some intuition behind the result, and show some preliminary results. Let $v=(v^1,\ldots,v^n)$ be a harmonic map from an $n$-dimensional
Riemannian domain $(M^n,g)$ into $\R^n$.  Put
\[
 U:=\sum_{\alpha=1}^n \mathrm{d}v^\alpha\otimes \mathrm{d}v^\alpha,
 \qquad |\mathrm{d}v|^2=\tr_g U.
\]
The usual harmonic-map stress-energy tensor is
 $T_v:=\frac12|\mathrm{d}v|^2g-U$. 
For every compactly supported smooth vector field $X$, it satisfies
\begin{equation}\label{eq:stress-weak-smooth}
 \int T_v:\nabla X\,\mathrm{d}\vol_g=0.
\end{equation}
This conservation law goes back at least to
Baird--Eells~\cite[Theorem 2.9]{BairdEells}.  A close precedent for its
use in a blow-up argument, and for Lemma~\ref{lem:Monotonicity} below,
is \cite[\S1, (1.3), (1.5)--(1.7), and Lemma 1.6]{Lin1999}. In Lemma~\ref{lem:RCDconserv} below we are going to prove an analogue of \eqref{eq:stress-weak-smooth} in our metric measure spaces $(X_k,\mathrm{d}_k',\mu_k)$.
\smallskip

For the computations below we will heavily rely on the theory and results of Gigli's paper \cite{GigliNDG}. We denote by $g_k$ the canonical Riemannian tensor of
$(X_k,\mathrm{d}_k',\mu_k)$, namely the unique symmetric
$L^\infty(X_k,\mu_k)$-bilinear form on $L^2(TX_k)$ satisfying
\[
 g_k(\nabla f,\nabla h)
 =\langle\nabla f,\nabla h\rangle
 \qquad\mu_k\text{-a.e.},
\]
for every $f,h\in W^{1,2}(X_k,\mathrm{d}_k',\mu_k)$; see \cite[Proposition~3.2]{AHPT}. For a covariant $2$-tensor $T$,
we write
\[
 \tr T:=g_k:T,
\]
where $:$ denotes the pointwise Hilbert--Schmidt scalar product.
In particular, notice that
$
 g:(\alpha\otimes\beta)=\langle\alpha,\beta\rangle
$
for every $\alpha,\beta\in L^0(T^*X)$. Notice that for $A\in L^2((T^*)^{\otimes 2} X_k)$ the Hilbert--Schmidt norm is $|A|:=\sqrt{A:A}$. Since $X_k$ has essential
dimension $n$, \cite[Lemma~3.3]{AHPT} gives
$|g_k|^2=n$ almost everywhere. Thus $\tr g_k=n$ almost everywhere as well. Below, we will use the musical
isomorphisms to switch from tangent to cotangent tensors, and vice-versa.

Let us set
\[
 U_k:=\sum_{i=1}^n \mathrm{d}v_k^i\otimes \mathrm{d}v_k^i,
 \qquad
 \mathcal K_k:=g_k-U_k,
 \qquad
 \mathcal S_k:=\mathcal K_k-\frac12
       (\tr\mathcal K_k)g_k.
\]
In analogy to \eqref{eq:stress-weak-smooth} we claim the following.

\begin{lemma} \label{lem:RCDconserv} For every $X\in\TestV(X_k)$ compactly supported in the region where all
the components of $v_k$ are harmonic, we have
\begin{equation}
 \int\mathcal S_k:(\nabla X)^{\flat}\,\mathrm{d}\mu_k=0
 \label{eq:div-Sk}
\end{equation}
\end{lemma}
\begin{proof}
    We divide the proof into three steps. 
\medskip

\textit{Step 1.}
Let $f$ be harmonic on an open set $\Omega\subset X_k$, and let $X\in\TestV(X_k)$ be such that $\supp X\Subset\Omega$. We claim that
\begin{equation}\label{eq:component-stress-Pre}
 \int(\mathrm{d}f\otimes \mathrm{d}f):(\nabla X)^{\flat}\,\mathrm{d}\mu_k
 =\frac12\int|\nabla f|^2\operatorname{div}X\,\mathrm{d}\mu_k.
\end{equation}

\textit{Proof of Step 1.} Let $X\in\TestV(X_k)$ satisfy
$\supp X\Subset\Omega$. Using, e.g., \cite[Lemma 3.1]{MondinoNaber}, choose $\vartheta\in\TestF(X_k)$ such that
\[
 \vartheta=1\quad\text{on a neighborhood of }\supp X,
 \qquad \supp\vartheta\Subset\Omega,
\]
and set $
 F:=\vartheta f$, 
extending the product by zero outside $\Omega$. Put
\[
 h:=\langle\nabla F,X\rangle,
 \qquad Y:=h\nabla F.
\]
Since $F\in\TestF(X_k)$, one has (using \cite[Theorem 3.3.8]{GigliNDG})
\[
 \nabla F\in L^\infty(TX_k),\qquad
 \Hess F\in L^2((T^*)^{\otimes 2} X_k).
\]
Since $X\in\TestV(X_k)$, one has (using \cite[Theorem 3.4.2(v)]{GigliNDG} and recalling \cite[Definition 3.4.1]{GigliNDG})
\[
 X\in L^\infty(TX_k),\qquad
 \nabla X\in L^2(T^{\otimes 2} X_k).
\]
Compatibility in \cite[Proposition 3.4.6]{GigliNDG}, \cite[Equation (3.4.6)]{GigliNDG}, and
$\nabla(\nabla F)=(\Hess F)^\sharp$ (\cite[Theorem 3.4.2(iv)]{GigliNDG}), give, for every vector field $Z\in L^0(TX_k)$
\begin{equation} \label{eq:metric-compatibility}
\begin{aligned}
 \mathrm{d}h(Z)
 &=\mathrm{d}\langle\nabla F,X\rangle(Z)=\langle\nabla_Z\nabla F,X\rangle
   +\langle\nabla F,\nabla_ZX\rangle\\
   &=\Hess F(Z,X)+\nabla X:(Z\otimes\nabla F).
\end{aligned}
\end{equation}
Moreover, taking into account also \cite[Proposition 3.4.6(i)]{GigliNDG},
$
 h\in W^{1,2}(X_k)\cap L^\infty(X_k)$. 
Taking $Z=\nabla F$ in \eqref{eq:metric-compatibility} gives
\begin{equation}\label{eq:dh-gradF}
 \mathrm{d}h(\nabla F)
 =\Hess F(\nabla F,X)+(\mathrm{d}F\otimes \mathrm{d}F):(\nabla X)^{\flat}.
\end{equation}

The divergence Leibniz rule
\cite[Proposition~4.2(2)]{GigliViolo}, which we can apply since $h\in W^{1,2}(X_k)\cap L^\infty(X_k)$, gives
\begin{equation} \label{eq:div-product}
\begin{aligned}
 \operatorname{div}(h\nabla F)
 &=\mathrm{d}h(\nabla F)+h\operatorname{div}(\nabla F)=\mathrm{d}h(\nabla F)+h\Delta F.
\end{aligned}
\end{equation}
Combining \eqref{eq:dh-gradF} and \eqref{eq:div-product}, we obtain
\begin{equation}\label{eq:pointwise-divergence}
 \operatorname{div}\bigl(h\nabla F\bigr)
 =(\mathrm{d}F\otimes \mathrm{d}F):(\nabla X)^{\flat}
  +\Hess F(\nabla F,X)+(\Delta F)\langle \nabla F,X\rangle.
\end{equation}

The field $Y=h\nabla F$ has compact support contained in
$\supp X$. Thus,
from the definition of divergence \cite[Definition~2.3.11]{GigliNDG} and a cut-off argument, one can directly obtain
\begin{equation}\label{eq:int-divY-zero}
 \int\operatorname{div}Y\,\mathrm{d}\mu_k=0.
\end{equation}

Integrating \eqref{eq:pointwise-divergence}, using
\eqref{eq:int-divY-zero} and $\Delta f=0$ on $\Omega$, and noticing that $F=f$ on $\mathrm{supp}(X)$, gives
\begin{equation}\label{eq:before-harmonicity}
 \int(\mathrm{d}f\otimes \mathrm{d}f):(\nabla X)^{\flat}\,\mathrm{d}\mu_k
 =-\int\Hess f(\nabla f,X)\,\mathrm{d}\mu_k.
\end{equation}

The product rule for gradients
\cite[Proposition~3.3.22, (3.3.33)]{GigliNDG} gives
\[
 \mathrm{d}|\nabla f|^2=2\Hess f(\nabla f,\mathord\cdot), \quad \text{and hence} \quad  \Hess f(\nabla f,X)=\frac12\mathrm{d}|\nabla f|^2(X).
\]
Since $|\nabla f|^2\in W^{1,2}_{\mathrm{loc}}(X_k)$, the definition of
divergence gives
\[
 -\frac12\int \mathrm{d}|\nabla f|^2(X)\,\mathrm{d}\mu_k
 =\frac12\int|\nabla f|^2\operatorname{div}X\,\mathrm{d}\mu_k.
\]
Substituting this into \eqref{eq:before-harmonicity} proves the  identity \eqref{eq:component-stress-Pre}, as desired.
\medskip

\textit{Step 2.} We claim that for every $X\in\TestV(X_k)$ with compact support in $X_k$
\begin{equation}\label{eq:metric-divergence}
 g_k:(\nabla X)^{\flat}=\operatorname{div}X.
\end{equation}

\textit{Proof of Step 2.}
Write $
 X=\sum_r a_r\nabla b_r$, with $a_r,b_r\in\TestF(X_k)$.
The covariant-derivative Leibniz rule
\cite[Theorem~3.4.2(v)]{GigliNDG} gives
\[
 (\nabla X)^{\flat}
 =\sum_r\left(\mathrm{d}a_r\otimes \mathrm{d}b_r+a_r\Hess b_r\right).
\]
Taking the metric trace gives
\begin{equation}\label{eq:trace-nablaX}
 g_k:(\nabla X)^{\flat}
 =\sum_r\left(
 \langle\nabla a_r,\nabla b_r\rangle
 +a_r\tr\Hess b_r\right).
\end{equation}
On a weakly noncollapsed $\mathrm{RCD}(K,n)$ space,
$
 \tr\Hess b_r=\Delta b_r
$; see \cite[Theorem~1.12 and (1.10)]{DePhilippisGigli}.  The divergence
Leibniz rule \cite[Proposition~4.2(2)]{GigliViolo} also gives
\[
 \operatorname{div}X
 =\sum_r\left(
 \langle\nabla a_r,\nabla b_r\rangle+a_r\Delta b_r\right).
\]
Comparing the latter with \eqref{eq:trace-nablaX} proves
\eqref{eq:metric-divergence}, as desired.
\medskip

\textit{Step 3}. Notice that
$
 \tr U_k=\sum_{i=1}^n|\nabla v_k^i|^2$. 
Applying \eqref{eq:component-stress-Pre} to every component $v_k^i$ and
summing gives
\begin{equation}\label{eq:sum-components}
 \int U_k:(\nabla X)^{\flat}\,\mathrm{d}\mu_k
 =\frac12\int(\tr U_k)\operatorname{div}X\,\mathrm{d}\mu_k,
\end{equation}
for every $X\in \TestV(X_k)$ with compact support in the common region where all $v_k^i$ are harmonic.
Since $g_k:(\nabla X)^{\flat}=\operatorname{div}X$ by \eqref{eq:metric-divergence}, equation
\eqref{eq:sum-components} is equivalent to
\begin{equation}\label{eq:usual-stress}
 \int\left[\frac12(\tr U_k)g_k-U_k\right]:(\nabla X)^{\flat}\,\mathrm{d}\mu_k=0.
\end{equation}
The tensor in brackets is the usual harmonic-map stress-energy tensor. This is precisely the non-smooth version of \eqref{eq:stress-weak-smooth}. Since $\tr g_k=n$ almost everywhere, one can readily compute that 
\[
 \mathcal S_k
 =g_k-U_k-\frac12\bigl(n-\tr U_k\bigr)g_k=\left(1-\frac n2\right)g_k
   +\left[\frac12(\tr U_k)g_k-U_k\right].
\]
The second tensor in the right-hand-side above has zero weak divergence by
\eqref{eq:usual-stress}, and the first contributes, by using Step 2 above,
\[
 \left(1-\frac n2\right)
 \int g_k:(\nabla X)^{\flat}\,\mathrm{d}\mu_k
 =\left(1-\frac n2\right)
 \int\operatorname{div}X\,\mathrm{d}\mu_k=0.
\]
Therefore \eqref{eq:div-Sk} is true, as desired.
\end{proof}

Using this result we can prove Lemma~\ref{lem:StressIdentity}.

\begin{proof}[Proof of Lemma~\ref{lem:StressIdentity}]
The aim is to pass the result of Lemma~\ref{lem:RCDconserv} to the limit space. 
Fix
$\xi\in C_c^\infty(\mathbb R^n;\mathbb R^n)$, and choose $L$ such that
$\operatorname{supp}\xi\Subset B_L(0)$. Let $\zeta_k\in\TestF(X_k)$ be a cut-off satisfying
\[
 \zeta_k\equiv1\ \text{on }B_{L+2}(x_k),
 \qquad
 \operatorname{supp}\zeta_k\subset B_{L+3}(x_k).
\]
Such a cutoff is supplied by \cite[Lemma 3.1]{MondinoNaber}.
Choose a second such cutoff
$\chi_k$, equal to one on $B_{L+3}(x_k)$, supported in
$B_{L+4}(x_k)$, and set $\widetilde v_k^j:=\chi_kv_k^j$.  Since
$v_k^j\in\TestF_{\mathrm{loc}}(B_{2/r_k}^{\mathrm{d}'_k}(x_k))$, we have $\widetilde v_k^j\in\TestF(X_k)$: explicitly,
\[
 \Delta(\chi_kv_k^j)=v_k^j\Delta\chi_k
       +2\langle\nabla\chi_k,\nabla v_k^j\rangle,
\]
and the right-hand side is in $W^{1,2}$ by
\cite[Proposition 3.3.22]{GigliNDG}; boundedness and compact support give
the remaining test-function requirements.
Local uniform convergence $v_k\to\operatorname{id}_{\mathbb R^n}$ (see \eqref{eq:v-to-id}) implies
\[
 |v_k(x)|>L
 \quad\text{for }x\in
 B_{L+3}(x_k)\setminus B_{L+2}(x_k)
\]
once $k$ is sufficiently large. Hence $\xi(v_k)=0$ on the transition
region of $\zeta_k$, for $k$ large enough; the same is true for $D\xi(v_k)$.
For each $j$, write
\[
 F_k^j:=\zeta_k\left[\xi^j(\widetilde v_k)-\xi^j(0)\right]
          +\xi^j(0)\zeta_k.
\]
The condition that the bracketed smooth function vanishes at the origin,
the multivariate chain rule \cite[Lemma 3.1.4]{GigliNDG}, and the test
algebra property show that $F_k^j\in\TestF(X_k)$.  Therefore the compactly
supported vector field $
 Y_k:=\sum_{j=1}^nF_k^j\nabla\widetilde v_k^j$
belongs to $\TestV(X_k)$.  On $B_{L+3}(x_k)$ it is
$\zeta_k\sum_j\xi^j(v_k)\nabla v_k^j$, while all cutoff contributions coming from $\xi$
vanish where $\zeta_k$ varies, as noticed above.  Consequently, on $B_{L+3}(x_k)$, for all large $k$,
\begin{equation}
 (\nabla Y_k)^{\flat}
 =\sum_{j,\ell}(\partial_\ell\xi^j)(v_k)\,
       \mathrm{d}v_k^\ell\otimes \mathrm{d}v_k^j
   +\sum_j\xi^j(v_k)\operatorname{Hess}v_k^j.
 \label{eq:nabla-Xk}
\end{equation}

We compute the contraction of the first term in the right-hand-side of \eqref{eq:nabla-Xk}.  All the following identities are
identities of measurable functions, and hence hold $\mu_k$-almost
everywhere.  Recall that
$
 (G_k)_{ij}
 =\langle \mathrm{d}v_k^i,\mathrm{d}v_k^j\rangle
 =\langle\nabla v_k^i,\nabla v_k^j\rangle$. 
We have
\[
 U_k:(\mathrm{d}v_k^\ell\otimes \mathrm{d}v_k^j)
 =\sum_{i=1}^n
   \langle \mathrm{d}v_k^i,\mathrm{d}v_k^\ell\rangle
   \langle \mathrm{d}v_k^i,\mathrm{d}v_k^j\rangle
 =\sum_{i=1}^n(G_k)_{\ell i}(G_k)_{ij}
   =(G_k^2)_{\ell j}.
\]
Since
$\mathcal K_k=g_k-U_k$, it follows that
\[
 \mathcal K_k:(\mathrm{d}v_k^\ell\otimes \mathrm{d}v_k^j)
 =(G_k)_{\ell j}-(G_k^2)_{\ell j}.
\]
Moreover,
\[
 \tr\mathcal K_k
 =\tr g_k-\tr U_k
 =n-\sum_{i=1}^n|\mathrm{d}v_k^i|^2
 =n-\tr G_k.
\]
Therefore
\begin{equation}
 \mathcal S_k:(\mathrm{d}v_k^\ell\otimes \mathrm{d}v_k^j)
 =\left[G_k-G_k^2-\frac12
       (n-\tr G_k)G_k\right]_{\ell j}.
 \label{eq:stress-Gram-exact}
\end{equation}
Since $G_k=I+a_kH_k$, the right-hand side of
\eqref{eq:stress-Gram-exact}, divided by $a_k$,
is
\begin{equation}
 -H_k+\frac12(\operatorname{tr}H_k)I+R_k,
 \qquad
 R_k=a_k\left[-H_k^2+\frac12
       (\operatorname{tr}H_k)H_k\right].
 \label{eq:stress-expansion}
\end{equation}
In particular,
\begin{equation}\label{eq:stress-remainder}
 |R_k|\leq C(n)a_k|H_k|^2
 =C(n)\frac{|G_k-I|^2}{a_k},
\end{equation}
and thus, from \eqref{eq:quadratic-remainderPre} we have that, for every $R>0$,
\begin{equation}\label{eq:stress-remainderIntegral}
 \int_{B_R(x_k)}|R_k|\mathrm{d}\mu_k\to 0.
\end{equation}

It remains to control the contraction of the second term in the right-hand-side of \eqref{eq:nabla-Xk}.
We have
\[
 |\mathcal K_k|^2
 =|g_k-U_k|^2
 =|g_k|^2-2\,g_k:U_k+|U_k|^2
 =n-2\tr G_k+\sum_{i,j=1}^n(G_k)_{ij}^2
 =|I-G_k|^2.
\]
Thus
\[
 |\mathcal K_k|=|I-G_k|,
 \qquad
 |\tr\mathcal K_k|
 =|\tr(I-G_k)|
 \leq\sqrt n\,|I-G_k|,
\]
and consequently
\begin{equation}\label{eq:S-Gram-bound}
 |\mathcal S_k|\leq C(n)|G_k-I|.
\end{equation}

On the other hand, the measure inequality
\eqref{eq:Hess-Bochner-bound-RCD} gives, for every fixed $R>0$ and all
sufficiently large $k$,
\begin{equation}
 \int_{B_R(x_k)}\frac{|\operatorname{Hess}v_k|^2}{a_k}\mathrm{d}\mu_k
 \leq\frac12\operatorname{tr}\nu_k(B_R(x_k)),
 \label{eq:Hess-Bochner-bound}
\end{equation}
where $|\Hess v_k|^2:=\sum_{j=1}^n |\Hess v_k^j|^2$.
Combining \eqref{eq:quadratic-remainderPre}, \eqref{eq:S-Gram-bound},
\eqref{eq:Hess-Bochner-bound}, and \eqref{eq:nuk-growth} for every fixed $R>0$ we obtain
\begin{equation}
 \begin{split}
 \int_{B_R(x_k)}\frac{|\mathcal S_k|}{a_k}
       |\operatorname{Hess}v_k|\mathrm{d}\mu_k
 &\leq C(n)
 \left(\int_{B_R(x_k)}\frac{|G_k-I|^2}{a_k}\mathrm{d}\mu_k\right)^{\frac{1}{2}}\\
 &\phantom{++}\cdot \left(\int_{B_R(x_k)}\frac{|\operatorname{Hess}v_k|^2}{a_k}\mathrm{d}\mu_k\right)^{\frac{1}{2}}
 \to 0.
 \end{split}
 \label{eq:Hess-stress-vanish}
\end{equation}

Finally, since
$D\xi(v_k)\to D\xi(x)$ locally uniformly, and, since $H_k\to H$ strongly in
$L^q_{\mathrm{loc}}$ with $1<q<\frac{n}{n-1}$, we have, using \cite[Proposition 3.27]{HondaLp}, that
\begin{equation}\label{eqn:Final}
\int_{B_{L+3}(x_k)}
 \left[-H_k+\frac12(\operatorname{tr}H_k)I\right]:D\xi(v_k)
 \,\mathrm{d}\mu_k\to  \int_{\mathbb R^n}
 \left[-H+\frac12(\operatorname{tr}H)I\right]:D\xi
 \,\mathrm{d}\mu_\infty.
\end{equation}
Testing \eqref{eq:div-Sk} with $X=Y_k$, dividing by $a_k$, and using
\eqref{eq:nabla-Xk}--\eqref{eq:stress-remainderIntegral}, \eqref{eq:Hess-stress-vanish} and \eqref{eqn:Final} we may now pass to the limit.
We conclude that
\[
 \int_{\mathbb R^n}
 \left[-H+\frac12(\operatorname{tr}H)I\right]:D\xi
 \,\mathrm{d}\mu_\infty=0
 \qquad
 \text{for every }\xi\in C_c^\infty(\mathbb R^n;\mathbb R^n),
\]
as desired.
\end{proof}

We shall use the following monotonicity formula.

\begin{lemma}\label{lem:Monotonicity}
Let $n\geq2$, and let $\sigma$ be a locally finite, symmetric,
positive semidefinite matrix-valued Radon measure on $\mathbb R^n$. Put
$\mu:=\operatorname{tr}\sigma$ and $M(r):=\mu(B_r(0))$. If
\begin{equation}
 \operatorname{div}\left(\sigma-\frac12\mu I\right)=0,
 \label{eq:stationary-assumption}
\end{equation}
then
\begin{equation}
 \frac{M(s)}{s^{n-2}}\leq\frac{M(R)}{R^{n-2}}
 \qquad \forall 0<s<R.
 \label{eq:stationary-monotonicity}
\end{equation}
\end{lemma}

\begin{proof}
Let $\rho_\varepsilon$ be a standard nonnegative mollifier and set
$\sigma_\varepsilon:=\rho_\varepsilon*\sigma$ and
$\mu_\varepsilon:=\operatorname{tr}\sigma_\varepsilon$.
Convolution preserves positive semidefiniteness and
\eqref{eq:stationary-assumption}. Define
$
 T_\varepsilon:=\frac12\mu_\varepsilon I-
       \sigma_\varepsilon$.
Then $\operatorname{div}T_\varepsilon=0$ and
$\operatorname{tr}T_\varepsilon=(n-2)\mu_\varepsilon/2$. We apply the divergence
theorem to the vector field $T_\varepsilon x$ on $B_r$. Since $\operatorname{div} (T_\epsilon x) = \tr T_\epsilon$, this gives
\[
 \frac{n-2}{2}M_\varepsilon(r)
 =r\int_{\partial B_r}T_\varepsilon(e_r,e_r)
 =\frac r2M_\varepsilon'(r)
   -r\int_{\partial B_r}
       \sigma_\varepsilon(e_r,e_r),
\]
where $e_r$ is the unit radial vector field, and where $M_\varepsilon(r)=\mu_\epsilon(B_r(0))$. Thus
\begin{equation}
 rM_\varepsilon'(r)-(n-2)M_\varepsilon(r)
 =2r\int_{\partial B_r}
       \sigma_\varepsilon(e_r,e_r)\geq0.
 \label{eq:stationary-derivative}
\end{equation}
Equivalently, $
 \frac{\mathrm{d}}{\mathrm{d}r}\left(r^{2-n}M_\varepsilon(r)\right)\geq0$. 
It follows that \eqref{eq:stationary-monotonicity} holds with
$M_\varepsilon$ in place of $M$.
If $r$ is a continuity radius for $\mu$, weak convergence
$\mu_\varepsilon\rightharpoonup\mu$ gives
$M_\varepsilon(r)\to M(r)$. We may therefore pass to the limit first at
continuity radii and then at arbitrary radii by monotone approximation.
This proves \eqref{eq:stationary-monotonicity}. \end{proof}

\subsubsection{Conclusion of the proof of Theorem~\ref{thm:main}}
\label{sec:stationary-vanishing} 
 Lemma
\ref{lem:StressIdentity} and \eqref{eqn:LimitingLaplacian} give
\[
 0=\Delta\operatorname{div}\left(H-\frac12(\tr H)I\right)
  =\omega_n\operatorname{div}\left(\nu-\frac12(\tr\nu)I\right).
\]
Hence, we may apply Lemma
\ref{lem:Monotonicity} to $\sigma=\nu$, because we proved that $\nu$ is locally finite, symmetric, and positive semidefinite. With
$M(R):=\operatorname{tr}\nu(B_R)$, the estimate obtained in Lemma \ref{lem:EstimatesH} gives
\begin{equation}
 M(R)\leq C R^{n-2-\kappa},\qquad \forall R\geq1.
 \label{eq:nu-growth-final}
\end{equation}
For fixed $s>0$, \eqref{eq:stationary-monotonicity} and
\eqref{eq:nu-growth-final} give
\[
 \frac{M(s)}{s^{n-2}}\leq C R^{-\kappa}
 \qquad \forall R>\max\{s,1\}.
\]
Letting $R\to\infty$ shows that $M(s)=0$. Hence $\operatorname{tr}\nu=0$ as a measure.
Positive semidefiniteness of $\nu$ then implies $\nu=0$ as a measure.

Therefore, from Lemma \ref{lem:EstimatesH} we have $\Delta_{\mathbb R^n}H=\omega_n\nu=0$ distributionally. Thus, each entry of $H$ is
a smooth entire harmonic function. Moreover, by Lemma \ref{lem:EstimatesH} we also have
\begin{equation}
 \fint_{B_1(0)}|H|\mathrm{d}\mu_\infty=1,
 \qquad
 \fint_{B_R(0)}|H|\mathrm{d}\mu_\infty\leq 2R^{-\kappa}\qquad \forall R\geq1.
 \label{eq:H-growth-final}
\end{equation}
From the mean value inequality applied to each entry of $H$, and the inequality above (sending $R\to \infty$), we conclude $H=0$, reaching a contradiction with the equality above. This contradiction finally proves Theorem~\ref{thm:main}, as desired.

\section{The proof of \texorpdfstring{\cref{thm:extension}}{Theorem 1.5}}
Our main goal in this section is to prove Theorem~\ref{thm:extension}, and we will deduce Theorems~\ref{thm:4dEinstein} and \ref{thm:highdimEinstein} from it. The main ingredient is to show that in harmonic coordinates the metric tensor $g_{ij}$ is H\"older continuous, which we will prove in Lemma~\ref{lemmaholder}. This in turn relies on the decay estimate provided by Proposition~\ref{propThreeAnnulus}. The basic idea is to exploit that under the Ricci curvature bound, at least on the smooth locus, the metric components behave more and more like harmonic functions at smaller and smaller scales. At the same time Theorem~\ref{thm:main} can be used to control the behavior of the metric components near the singular locus.

\subsection{The three-annulus lemma and consequences}
Throughout this section, let $K\geq 0$ and let $(X,\mathrm{d},\mathcal{H}^n)$ be an $\mathrm{RCD}(-K,n)$ space.

\begin{remark}\label{remHausDim}
    Under the condition of Definition~\ref{def:Codimension}, we have the following:
    \begin{enumerate}
        \item By Bishop--Gromov monotonicity, for every $x\in X$, and every $s<\rho_0/4$
        \[
        \mathcal{H}^n(B_s(\Sigma)\cap B_{\rho_0/2}(x))\leq C(n,K,\mathcal{K}) s^{n-c}.
        \]
        \item From the definition of Hausdorff dimension,  we thus get
        \[
        \mathrm{dim}_{\mathrm{H}}(\Sigma) \leq c < n-2.
        \]
    \end{enumerate}
\end{remark}

The proof of \cref{thm:extension} will build on the following crucial three-annulus-type lemma.
\begin{prop}\label{propThreeAnnulus}
    For every $\beta>0$, every integer $n\geq 2$, and $\mathcal{K}>0$, $\rho_0>0$, $0\leq c<n-2$ there are $0<r_0<1$, $0<\delta<1$, and $C>0$ such that the following holds. 
    
    Let $(X,\mathrm{d},\mathcal{H}^n)$ be an $\mathrm{RCD}(-\eta,n)$ space, with $\eta\geq 0$ such that:
    \begin{enumerate}
        \item $\Sigma \subset X$ satisfies \cref{def:Codimension} with $k=2$ and the choices of $\mathcal{K},\rho_0,c$ above: in particular, $\Sigma$ has codimension greater than $2$. Moreover, $X\setminus\Sigma$ is (locally isometric to) a smooth Riemannian manifold;
        \item $|\mathrm{Ric}| \leq \eta$ on $X\setminus \Sigma$.
    \end{enumerate}
    Let $x\in X$, and assume $\{f_i\}_{i=1}^n$ are harmonic functions on $B_4(x)$. Denote
    \[ 
    D:=\sum_{i,j} \fint_{B_4(x)} |\nabla f_i \cdot \nabla f_j - \delta_{ij}|\mathrm{d}\mathcal{H}^n.
    \]
    Assume $D+\eta\leq \delta$. Then there is a symmetric matrix $A_{ij}$ with $\|A_{ij}\| < C(D+\eta^{1/2})$ such that $\tilde{f}_i := f_i + \sum A_{ij}f_j$ satisfy
    \[ \sum_{i,j} \fint_{B_{r_0}(x)} |\nabla \tilde{f}_i\cdot \nabla \tilde{f}_j - \delta_{ij}|\d\mathcal{H}^n \leq r_0^{1-\beta} (D+\eta^{1/2}). \]
\end{prop}
\begin{proof}
    In this proof the constants $C$ will only depend on $n$, $\mathcal{K}$, $\rho_0$, $c$ (see \cref{def:Codimension}), and on an arbitrary $0<\varepsilon<n-c-2$ which we now fix once and for all. When we write $A\lesssim B$ we mean $A\leq CB$. We argue by contradiction.
    
    If the assertion is not true, we can find a contradicting sequence $(X_k,\mathrm{d}_k,\mathcal{H}^n,x_k)$ of pointed $\mathrm{RCD}(-\eta_k,n)$ spaces, with $\eta_k\geq 0$, for which the following holds. We have harmonic functions $f^{(k)}_i$ defined on $B_4(x_k)$, with uniformly bounded gradients (see item (2) of \cref{lem:PropertiesHarmonic}) on $B_3(x_k)$, such that if $g^{(k)}_{ij} := \nabla f^{(k)}_i\cdot \nabla f^{(k)}_j$, and 
    \[ D_k := \sum_{i,j} \fint_{B_4(x_k)} |g_{ij}^{(k)} - \delta_{ij}|\d\mathcal{H}^n_k, \]
    we have $D_k+\eta_k\to 0$;
    but, for any $k$ sufficiently large, we have no $0<r_0<1$ and $C>0$ for which the conclusion of the statement holds.
    
    First, taking into account item (1) of \cref{lem:PropertiesHarmonic}, we can apply Lemma~\ref{lem:almost-splitting} (up to a harmless reparametrization of the metric)  with
$C(D_k+\eta_k)$ in place of its parameter.
    We thus get $B_4(x_k) \to B_4^{\mathbb R^n}(0)$ in the pointed GH sense.

    By Blaschke's selection theorem, up to subsequences, $\Sigma_k\cap B_4(x_k) \to \Sigma_\infty\subset B_4^{\mathbb R^n}(0)$ in the Hausdorff sense in a realization of the GH convergence, where $\Sigma_\infty$ is a closed set. Since \cref{def:Codimension} is stable under GH convergence, we have that $\Sigma_\infty$ has codimension greater than $2$ as well. From now on, all convergences below must be understood possibly after passing to a subsequence.
    
    We claim that up to choosing a subsequence, we have $(D_k+\eta_k^{1/2})^{-1}(g_{ij}^{(k)} - \delta_{ij}) \to h_{ij}$ locally $C^{1,\alpha'}$, for every $0<\alpha'<1$, on $B_2^{\mathbb R^n}(0)\setminus \Sigma_\infty$, where the $h_{ij}$ are harmonic functions on $B_2^{\mathbb R^n}(0)\setminus\Sigma_{\infty}$. Indeed, let us take open sets $U\Subset V\Subset B_2^{\mathbb R^n}(0)\setminus \Sigma_\infty$, and let $U_k\Subset B_2^{X_k}(x_k)\setminus \Sigma_k$ be such that $\overline{U_k}\to \overline{V}$. Note that the $g^{(k)}_{ij}$ are the (inverse) metric components in the harmonic coordinates defined by  the $f^{(k)}_i$. 
    Since the metrics $g^{(k)}$ have uniformly bounded Ricci curvature outside $\Sigma_k$, by Anderson's theorem~\cite{Anderson90} together with volume convergence and Bishop--Gromov monotonicity, we have uniform $C^{1,\alpha}$ estimates for the metric components $g^{(k)}_{ij}$ on the $U_k$ in these harmonic coordinates. These estimates are independent of $k$. 
    
    Using the expression for the Ricci curvature in  harmonic coordinates, we have that on $\overline{U_k}$
    \begin{equation}\label{eqn:ToGiveSchauder}
    \frac{1}{2}\Delta_{g^{(k)}} (g^{(k)}_{ij}-\delta_{ij})+\partial g^{(k)}\ast \partial g^{(k)}=\mathrm{Ric}^{(k)}_{ij}.
    \end{equation}
    With $\partial g^{(k)}\ast \partial g^{(k)}$ we schematically denote a term that is quadratic in $\partial g^{(k)}$, but might contain interactions with $g$ as well. We can view this as a first order term in the equation for the quantity $g^{(k)}_{ij} - \delta_{ij}$, with coefficients depending on $g^{(k)}, \partial g^{(k)}$. By item (2) of \cref{lem:PropertiesHarmonic}, we have $|g_{ij}^{(k)}|\lesssim 1$ on $\overline{U_k}$, and at the same time we use the bound $|\mathrm{Ric}^{(k)}| \leq \eta_k$. Hence, by Schauder estimates for the system \eqref{eqn:ToGiveSchauder} we have, for $0<\alpha<1$,
    \[
    \| g^{(k)}_{ij}-\delta_{ij}\|_{C^{1,\alpha}(U_k)}\leq\mathcal{C}\left( \|g^{(k)}_{ij}-\delta_{ij}\|_{L^1(B_2^{X_k}(x_k))} + \|\mathrm{Ric}^{(k)}\|_{L^\infty(U_k)}\right) \leq C(D_k+\eta_k),
    \]
    where $C$ above might depend on $\alpha$. Using a diagonal procedure, it follows that $(D_k+\eta_k^{1/2})^{-1}(g_{ij}^{(k)}-\delta_{ij})$ converges to $h_{ij}$ in $C^{1,\alpha}$ away from $\Sigma_\infty$, up to subsequences, for every $0<\alpha<1$. We divide through \eqref{eqn:ToGiveSchauder} by $D_k + \eta_k^{1/2}$ (we use $\eta_k^{1/2}$ rather than $\eta_k$ in order to eliminate the Ricci term in the limit), and note that by the estimate above, we have 
    \[ (D_k + \eta_k^{1/2})^{-1}(\partial g^{(k)}\ast \partial g^{(k)}) \to 0, \qquad (D_k + \eta_k^{1/2})^{-1} \mathrm{Ric}^{(k)}_{ij} \to 0. \]
    It follows (multiplying the equation by a smooth test function and arguing in a weak sense) that $\Delta h_{ij} = 0$ away from $\Sigma_\infty$, as claimed.

    \smallskip
    We now claim that for every $\varepsilon>0$ there exists a constant $C>0$ such that the following holds. For every $x\in B_2^{\mathbb R^n}(0)$ and every $r<1/8$ we have that 
    \begin{equation}\label{eqn:Claim}
    \begin{aligned}
    \fint_{B_r(x)}|h_{ij}|\d\mathcal{H}^n&\leq Cr^{-\varepsilon}. \\
    \end{aligned}
    \end{equation}
    Indeed, using \cref{thm:main}, we have that for every $k$ large enough the following holds. For every $\tilde x\in B_2(x_k)$ and $r<1/8$ we have
    \begin{equation}\label{eqn:Prelimit}
    \fint_{B_r(\tilde x)}|g_{ij}^{(k)}-\delta_{ij}|\d\mathcal{H}^n_k\leq Cr^{-\varepsilon}(D_k+\eta_k).
    \end{equation}
    Now we have that $(D_k+\eta_k^{1/2})^{-1}(g_{ij}^{(k)}-\delta_{ij})$ converges to $h_{ij}$ smoothly away from $\Sigma_\infty$, and thus \eqref{eqn:Claim} follows from \eqref{eqn:Prelimit} and the volume convergence theorem.
    \smallskip
    
    We now claim that $h_{ij}$ extends smoothly across $\Sigma_\infty$. Let $\delta>0$ be small enough, and let $\phi_\delta$ be a standard cutoff function supported outside of the $\delta/2$-neighborhood $(\Sigma_\infty)_{\delta/2}$ of $\Sigma_\infty$, and equal to $1$ outside of $(\Sigma_\infty)_{2\delta}$; and let $\chi$ be any compactly supported smooth function on $B_2^{\mathbb R^n}(0)$. We can choose $\phi_\delta$ so that $|\nabla\phi_\delta| \lesssim \delta^{-1}$ and $|\Delta\phi_\delta| \lesssim \delta^{-2}$. Indeed, let us cover $\Sigma_\infty$ with $\lesssim \delta^{-c}$ balls $B(y_k, \delta/2)$, such that the $B(y_k, \delta/10)$ are disjoint, and $y_k\in \Sigma_\infty$. For each $k$ we have a standard cutoff $\psi_k$ such that $\psi_k=1$ on $B(y_k, \delta)$ and $\psi_k=0$ outside of $B(y_k, 2\delta)$. These satisfy $|\nabla\psi_k|\lesssim \delta^{-1}$ and $|\Delta\psi_k|\lesssim \delta^{-2}$. Define $\Phi = \sum\psi_k$, and $\phi_\delta(x) = 1 - \rho(\Phi(x))$, where $\rho$ is a smooth function with $\rho(0)=0$, $\rho(t)=1$ for $t > 9/10$, and $|\rho'|, |\rho''| < 10$. Note that
    \[ |\nabla\phi_\delta| \leq 10 |\nabla \Phi|, \qquad |\nabla^2 \phi_\delta|\leq 10( |\nabla \Phi|^2 + |\nabla^2\Phi|). \]

    If $x\in (\Sigma_\infty)_{\delta/2}$, then $\Phi(x)\geq 1$, so $\phi_\delta(x)=0$. If $x \not\in (\Sigma_\infty)_{2\delta}$, then $\Phi(x)=0$, so $\phi_\delta(x)=1$. For all other $x$ we have at most a controlled number $N(n)$ of balls $B(x_k, 2\delta)$ containing $x$, so we get bounds
    \[ \begin{aligned} 
    |\nabla\phi_\delta|\lesssim \delta^{-1}, \qquad |\Delta\phi_\delta|\lesssim \delta^{-2}, \qquad \text{as desired}. \end{aligned} \]
    
    Calling $h=h_{ij}$, it follows that
\begin{equation}\label{eqn:preapproxharmoni}
\begin{aligned}  
\int h \phi_\delta \Delta\chi\, \d\mathcal{H}^n &= \int \left(h \Delta(\phi_\delta \chi) - h \chi \Delta\phi_\delta - 2h\nabla\phi_\delta\cdot \nabla \chi\right)\d\mathcal{H}^n. 
\end{aligned}
\end{equation}
The first term vanishes using $\Delta h_{ij}=0$ away from $\Sigma_\infty$.  Since we have chosen and fixed one $0<\varepsilon < n-c-2$ at the beginning of the proof, the second term is estimated (using \eqref{eqn:Claim}) by
\[ \left|\int h \chi \Delta\phi_\delta\, \d\mathcal{H}^n \right| \lesssim \delta^{-2} \int_{\Sigma_{2\delta}} |h| \lesssim \delta^{-2} \delta^{-c} C\delta^{-\epsilon + n} \to 0,\]
as $\delta\to 0$ (note that this would not work if $\Sigma$ had codimension $\leq 2$, which is to be expected). The last term in \eqref{eqn:preapproxharmoni} similarly converges to zero. Sending $\delta\to 0$ in \eqref{eqn:preapproxharmoni}, it follows that $h$ satisfies $\Delta h=0$ on $B_2^{\mathbb R^n}(0)$ in the distributional sense, so it extends smoothly across $\Sigma_\infty$ and satisfies $\Delta h = \Delta h_{ij}=0$ on $B_2^{\mathbb R^n}(0)$.
\smallskip

    Now, each $h_{ij}$ can be decomposed as follows 
    \[ 
    h_{ij} = h_{ij}(0) + h_{ij}^{>0}:=B_{ij}+ h_{ij}^{>0}. 
    \]
    By the mean-value inequality we have 
    \begin{equation}\label{eqn:LinftyboundBij}
    B_{ij}=h_{ij}(0)\lesssim \int_{B_2^{\mathbb R^n}(0)}|h_{ij}|\d\mathcal{H}^n\lesssim 1.
    \end{equation}
    We now define $\tilde{f}^{(k)}_i = f^{(k)}_i - (D_k+\eta_k^{1/2}) \frac{1}{2}B_{ij} f^{(k)}_j$, and let $\tilde{g}_{ij}^{(k)} = \nabla \tilde{f}^{(k)}_i\cdot \nabla \tilde{f}^{(k)}_j$. Then we have that $(D_k+\eta_k^{1/2})^{-1}(\tilde{g}_{ij}^{(k)} - \delta_{ij}) \to h_{ij}^{>0}$ in $L^1(B_2^{\mathbb R^n}(0)\setminus\Sigma_\infty)$. {Indeed, the latter follows from the fact that
    \begin{equation}\label{eqn:tildegijk}
    \tilde g_{ij}^{(k)} = g_{ij}^{(k)}-\frac{1}{2}(D_k+\eta_k^{1/2})\left(B_{is}g_{js}^{(k)}+B_{jl}g_{il}^{(k)}\right)+O((D_k+\eta_k^{1/2})^2),
    \end{equation}
    and the fact that $g_{ij}^{(k)}$ converges to the Euclidean metric in  $L^1$, and that $(D_k+\eta_k^{1/2})^{-1}(g_{ij}^{(k)}-\delta_{ij})\to h_{ij}$ smoothly locally away from $\Sigma_\infty$.
    }
    \smallskip
    
    Since $h_{ij}^{>0}$ is harmonic, it follows that
    \begin{equation}\label{eqn:DecayL2NormHarmonic}
    \sum_{i,j} \fint_{B_{r_0}^{\mathbb R^n}(0)} |h_{ij}^{>0}|\d\mathcal{H}^n \lesssim r_0,
    \end{equation}
    for all $r_0 \leq 1$. {In addition, we also have
    \begin{equation}\label{eqn:CDktildeg}
    \fint_{B_2(x_k)} |\tilde{g}_{ij}^{(k)} - \delta_{ij}|\d\mathcal{H}^n_k < C(D_k+\eta_k^{1/2}), 
    \end{equation}
    for every $k$, up to subsequences. Indeed, this directly follows from the assumptions, \eqref{eqn:LinftyboundBij}, and \eqref{eqn:tildegijk}.}
    \smallskip

    Now, for small $s > 0$ we can cover the $s$-neighborhood of $\Sigma_k \subset B_2(x_k)$ by $\lesssim s^{-c}$ balls of radius $s$. Outside of this neighborhood $(D_k+\eta_k^{1/2})^{-1}(\tilde g_{ij}^{(k)}-\delta_{ij})$ converges in $L^1$ to $h_{ij}^{>0}$ as $k\to\infty$, so for sufficiently large $k$ we have the following. For $r_0<1/8$, using \eqref{eqn:DecayL2NormHarmonic}, and the volume convergence,
    \[ \int_{B_{r_0}(x_k)\setminus B_s(\Sigma_k)} |\tilde{g}_{ij}^{(k)} - \delta_{ij}|\d\mathcal{H}^n_k \leq C(D_k+\eta_k^{1/2})r_0 r_0^n.
    \]
    Moreover, using \eqref{eqn:CDktildeg} and \cref{thm:main},  we have that for every $\varepsilon>0$, up to choosing $k$ large enough, the following holds. For every $x\in B_2(x_k)$ and $r<1/8$ we have
    \begin{equation}\label{eqn:Claim2}
    \fint_{B_r(x)}|\tilde g_{ij}^{(k)}-\delta_{ij}|\d\mathcal{H}^n_k\leq C(D_k+\eta_k^{1/2})r^{-\varepsilon}.
    \end{equation}
     Using the estimate \eqref{eqn:Claim2}, and \cref{remHausDim}, we finally have, for $s$ small enough,
    \[ \int_{B_{r_0}(x_k)\cap B_s(\Sigma_k)} |\tilde{g}_{ij}^{(k)} - \delta_{ij}|\d\mathcal{H}^n_k \leq C (D_k+\eta_k^{1/2}) s^{n-c-\varepsilon}.
    \]
    \smallskip
    
    Combining the previous estimates, we have, for $r_0<1/8$ and all $s>0$ small enough,
     \[ \fint_{B_{r_0}(x_k)} |\tilde{g}_{ij}^{(k)} - \delta_{ij}|\d\mathcal{H}^n_k \leq  C(D_k+\eta_k^{1/2}) r_0 + C(D_k+\eta_k^{1/2}) s^{n-c-\varepsilon}r_0^{-n}. \]
    Since we have fixed once and for all $0<\varepsilon < n-c-2 < n-c$ at the beginning of the proof, it is clear that we can first choose $s = r_0^K$ for some large $K$, and then $r_0$ sufficiently small, so that, absorbing the constants $C$, we have 
    \[ \sum_{i,j}\fint_{B_{r_0}(x_k)} |\tilde{g}_{ij}^{(k)} - \delta_{ij}|\d\mathcal{H}^n_k \leq (D_k+\eta_k^{1/2}) r_0^{1-\beta}, \]
    for every $k$ large enough, thus finding a contradiction.
\end{proof}

Using the three-annulus lemma, we have the following H\"older estimate for the metric components in harmonic coordinates. 
\begin{lemma}\label{lemmaholder}
For every $0<\gamma<1$, integer $n\geq 2$, and $\mathcal{K}>0$, $\rho_0>0$, $0\leq c<n-2$ there is $0<\delta<1$ such that the following holds. 

    Let $(X,\mathrm{d},\mathcal{H}^n)$ be an $\mathrm{RCD}(-\eta,n)$ space, with $\eta\geq 0$ such that:
    \begin{enumerate}
        \item $\Sigma \subset X$ satisfies \cref{def:Codimension} with $k=2$ and the choices of $\mathcal{K},\rho_0,c$ above: in particular, $\Sigma$ has codimension greater than $2$. Moreover, $X\setminus\Sigma$ is (locally isometric to) a smooth Riemannian manifold;
        \item $|\mathrm{Ric}| \leq \eta$ on $M\setminus \Sigma$.
    \end{enumerate} 
    Let $x\in X$, and $\{f_i\}_{i=1}^n$ be harmonic functions on $B_2(x)$ and set 
    \[ 
    \sum_{i,j} \fint_{B_1(x)} |\nabla f_i \cdot \nabla f_j - \delta_{ij}|\mathrm{d}\mathcal{H}^n =: D.
    \]
    Assume $D+\eta^{1/2}\leq \delta$. Then 
    \[
    g^{ij}:=\nabla f_i\cdot \nabla f_j \in C^\gamma(B_{1/2}(x)),
    \]
    and in addition
    \[ |g^{ij} - \delta_{ij}| \leq C\delta \text{ on } B_{1/2}(x). \]
\end{lemma}
\begin{proof}
    Let $\beta:=1-\gamma$. By iteratively\footnote{When we iterate \cref{propThreeAnnulus} we have that if the metric changes as $\tilde{\mathrm{d}}=r_0^{-1}\mathrm{d}$, then $(D,\eta)\rightarrow (r_0^{1-\beta}(D+\eta^{1/2}),r_0^2\eta)$. This allows to iterate the three-annulus-like lemma \cref{propThreeAnnulus} so as to obtain summable errors. Notice also that the hypothesis on the Assouad codimension being greater than 2, see \cref{def:Codimension}, is scale-invariant.} applying \cref{propThreeAnnulus} we obtain $r_0<1$ and $C$ such that the following holds. The constant $C$ might change from line to line. Call $s=r_0^{1-\beta}$. For every $y\in B_{1/2}(x)$, and with $k\geq 1$, there exist a sequence of functions $f^k:=(f_1^k,\ldots,f_n^k)$ (possibly depending on the point $y$) and matrices $A_k:=A_k(y)$ such that 
    \begin{itemize}
        \item $\|A_k\|\leq C(D+\eta^{1/2})s^{k-1}$;
        \item Setting $f^0:=(f_1,\ldots,f_n)$, we have $f^k=(I+A_k)f^{k-1}$ for every $k\geq 1$;
        \item We have 
        \[
        \sum_{i,j}\fint_{B_{r_0^k}(y)} |\nabla f^k_i\cdot \nabla f^k_j - \delta_{ij}|\d\mathcal{H}^n < s^{k}(D+\eta^{1/2}).
        \]
    \end{itemize}
    Notice that since $\sum_{k=1}^{\infty} \|A_k\|\leq \frac{C(D+\eta^{1/2})}{1-s}$, we have that, up to subsequence, $f^k\to \tilde f$ in $W^{1,2}(B_{r_0}(y))$. Moreover, since $|\nabla f|$ is bounded (see item (2) of \cref{lem:PropertiesHarmonic}), we have that for every $k$, $|\nabla f^k|\leq \prod_{m=1}^{\infty} (1+\|A_m\|)|\nabla f|$ is uniformly bounded above in $k$ as well. Furthermore, for every $k>\ell\geq 1$ we can estimate
    \[
    \begin{aligned}
        \fint_{B_{r_0^\ell}(y)} &|\nabla f^k_i\cdot\nabla f^k_j-\delta_{ij}|\d\mathcal{H}^n\leq \fint_{B_{r_0^\ell}(y)} |\nabla f^\ell_i\cdot\nabla f^\ell_j-\delta_{ij}|\d\mathcal{H}^n\\
        &+\sum_{m=\ell}^{k-1} |\nabla f^{m+1}_i\cdot\nabla f^{m+1}_j-\nabla f^m_i\cdot\nabla f^m_j| \\
        &\leq s^\ell (D+\eta^{1/2}) + C\sum_{m=\ell}^{k-1} \|(I+A_m)^T(I+A_m)-I\||\nabla f^m|^2 \\&\leq s^\ell(D+\eta^{1/2}) + C\sum_{m=\ell}^{k-1} \|A_m\| \leq s^\ell (D+\eta^{1/2}) + C(D+\eta^{1/2})\sum_{m=\ell}^{\infty} s^{m} \\
        &= (D+\eta^{1/2})\left(s^\ell+C\frac{s^\ell}{1-s}\right)\leq C(D+\eta^{1/2})s^\ell.
    \end{aligned}
    \]
    Passing to the limit as $k\to+\infty$ in the previous inequality we get that for every $\ell\geq 1$ we have 
    \[
    \fint_{B_{r_0^\ell}(y)} |\nabla \tilde f_i\cdot\nabla \tilde f_j-\delta_{ij}|\d\mathcal{H}^n\leq C(D+\eta^{1/2})s^\ell,
    \]
    and thus, interpolating and using Bishop--Gromov monotonicity we get that for every $r<r_0$
    \begin{equation}\label{eqnAlmsotDoneHolder}
    \fint_{B_{r}(y)} |\nabla \tilde f_i\cdot\nabla \tilde f_j-\delta_{ij}|\d\mathcal{H}^n\leq C(D+\eta^{1/2})r^{1-\beta}.
    \end{equation}
    Notice that $\tilde f = \prod_{k=1}(I+A_k) f=:(I+A)f$. If $\delta$ is small enough, $\|A\|$ is small and then $S:=(I+A)^T(I+A)$ is invertible with $\|S^{-1}\|\leq C$. Calling $\tilde A(y):=S^{-1}$ we thus have the following. Using \eqref{eqnAlmsotDoneHolder}, for every $y\in B_{1/2}(x)$ there is a matrix $\tilde A_{ij}(y)$ for which for every $r<r_0$ it holds
    \[
    \fint_{B_r(y)}|\nabla f_i\cdot \nabla f_j-\tilde A_{ij}(y)|\d\mathcal{H}^n \leq C(D+\eta^{1/2})r^{1-\beta}.
    \]
    Thus, calling $g^{ij}=\nabla f_i\cdot \nabla f_j$ we have, for all $r<r_0$,
    \[ \sum_{i,j} \fint_{B_{r}(y)} \left|g^{ij} - \fint_{B_{r}(y)} g^{ij}\d\mathcal{H}^n\right|\d\mathcal{H}^n\leq 2\sum_{i,j} \fint_{B_{r}(y)} |g^{ij} - \tilde A_{ij}(y)|\d\mathcal{H}^n \leq Cr^{1-\beta} (D+\eta^{1/2}).
    \]
    Thus by the Campanato embedding theorem, which holds on doubling metric measure spaces satisfying the Poincaré inequality, and thus a fortiori on $\mathrm{RCD}$ spaces, see \cite{CampanatoPI}, we conclude that $g^{ij}\in C^{1-\beta}(B_{1/2}(x))$, as desired. 
To see the bound for $g^{ij} - \delta_{ij}$, note that $\widetilde A_{ij}(y)=g^{ij}(y)$, and from the proof above we have $\Vert \widetilde{A}(y) - I\Vert \lesssim \delta$. 
\end{proof}

Using the bounded Ricci curvature assumption, the following result improves the H\"older estimate for the metric components to $C^{1,\alpha}$, which in turn will imply Theorem~\ref{thm:extension}. 
\begin{prop}\label{propCinfty}
    For every integer $n\geq 2$, $0<\alpha<1$,  and $\mathcal{K}>0$, $\rho_0>0$, $0\leq c<n-2$ there is $0<\delta<1$ such that the following holds. 

    Let $(X,\mathrm{d},\mathcal{H}^n)$ be an $\mathrm{RCD}(-\eta,n)$ space, with $\eta\geq 0$ such that:
    \begin{enumerate}
        \item $\Sigma \subset X$ satisfies \cref{def:Codimension} with $k=2$ and the choices of $\mathcal{K},\rho_0,c$ above: in particular, $\Sigma$ has codimension greater than $2$. Moreover, $X\setminus\Sigma$ is (locally isometric to) a smooth Riemannian manifold;
        \item $|\mathrm{Ric}| \leq \eta$ on $X\setminus \Sigma$.
    \end{enumerate}
    Let $x\in X$, and let $\{f_i\}_{i=1}^n$ be harmonic functions on $B_2(x)$. Set  
    \[ 
    D:=\sum_{i,j} \fint_{B_1(x)} |\nabla f_i \cdot \nabla f_j - \delta_{ij}|\mathrm{d}\mathcal{H}^n.
    \]
    Assume $D+\eta^{1/2}\leq \delta$. Then the following hold:
    \begin{enumerate}
        \item The vector-valued function $f:=(f_1,\ldots,f_n):B_{1/2}(x)\to \mathbb R^n$ is a biLipschitz homeomorphism onto its image, and $f(B_{1/2}(x))\subset\mathbb R^n$ is an open set.
        \item Let $g^{ij}:=\nabla f_i\cdot\nabla f_j$. Then
        \[
    g^{ij}\circ f^{-1}
    \]
    (a priori smooth only on $f(B_{1/2}(x)\setminus\Sigma)$), can be extended to a $C^{1,\alpha}$-function on the whole open set $f(B_{1/2}(x))$.
    \end{enumerate}
\end{prop}
\begin{proof}
    The item (1) is classical. We include a sketch of the proof for the reader's convenience. The constant $C$ below might change from line to line.
    
    For every $\varepsilon>0$, up to choosing $\delta$ small enough, we have the following. By \cref{lem:almost-splitting} (which we can apply after using item (1) of \cref{lem:PropertiesHarmonic}) and \cite[Theorem 4.3]{ChJiNa} we can assume $\mathrm{d}_{\mathrm{GH}}(B_r(y), B_r^{\mathbb R^n}(0))<\varepsilon r$ for every $y\in B_{3/2}(x)$ and $r<1/2$. Moreover, by \cref{rem:HessianNotNecessary}, $f=(f_1,\ldots,f_n)$ is a $C\delta$-splitting map. Then we are in a position to apply the geometric transformation theorem (see \cite[Proposition 3.13 \& Corollary 3.16]{BNSInv}, and the original reference for Ricci limit spaces \cite[Theorem 7.2]{ChJiNa}) to $f$. Hence, arguing verbatim as in the first part of the proof of \cite[Theorem 7.10]{ChJiNa}, using the geometric transformation theorem, we get that $f:B_{1/2}(x)\to\mathbb R^n$ is a biH\"older homeomorphism onto its image. Once we know that $f$ is a homeomorphism, it follows from the last claim in Lemma~\ref{lemmaholder} that $f$ is actually biLipschitz.
    \smallskip

    In order to prove item (2), we first claim the following: 
    
    \smallskip
    \noindent\emph{Claim.} For every $1\leq p<\infty$, if $\delta:=\delta(p,n)$ is chosen small enough, then $|\partial g|\in L^p(B_{1/2}(x))$. 

    \begin{proof}[Proof of Claim]
    Let us fix $\beta>0$, let $p\in X\setminus \Sigma$, and define $r = \frac{1}{2}d(p)$ where $d(p)$ is the distance to $\Sigma$.  Rescale the metric by the factor $r^{-2}$ so that the ball $B_r(p)\Subset X\setminus\Sigma$ becomes a unit ball. In the rescaled metric, we denote the corresponding quantities by capital letters. 
    
    For any fixed $\alpha<1$, if $\delta:=\delta(n,\alpha)$ is chosen small enough, then from the Hölder control obtained in \cref{lemmaholder}, we have
\[
|G(x)-G(y)| \lesssim d^\alpha |x-y|^\alpha,
\]
for every $x,y$ in the unit rescaled ball. Hence the Hölder seminorm satisfies $[G]_\alpha \lesssim d^\alpha$ in the unit rescaled ball. At the same time, Anderson's result~\cite{Anderson90} together with the uniform Ricci bound on $B_{2r}(p)$ implies that the rescaled metric $G$ has a uniform $C^{1,\alpha}$-bound on the unit rescaled ball. We interpolate between these H\"older bounds for $G$ and $\partial G$: 
\[
\|\partial G\|_{L^\infty}
\lesssim
[G]_\alpha^{\alpha}[\partial G]_\alpha^{1-\alpha} \lesssim d^{\alpha^2}.
\]
Choosing $\alpha=\alpha(\beta)$ close enough to $1$ we get
\[
\|\partial G\|_{L^\infty}
\lesssim
d^{1-\beta},
\]
so that scaling back to the original metric, this yields
\begin{equation}\label{eqn:ToUseAcross}
|\partial g| \lesssim d^{-\beta}.
\end{equation}

    Let us now go back to the claim. Fix $1\leq p<\infty$.
    Since the singular set has codimension $2+a$, with $a>0$, we have, using the latter estimate on $|\partial g|$, and choosing $\delta:=\delta(\beta,n)$ small enough,
    \begin{equation}\label{eqn:LpControlll}
    \int_{B_{1/2}(x) \setminus \Sigma} |\partial g|^p \lesssim \int_0^1 s^{-\beta p}\, s^{(2+a)-1}\, \mathrm{d}s \lesssim 1, \end{equation}
    as long as $\beta:=\beta(p)$ is small enough for the given $p$. This finishes the proof of our claim.
    \end{proof}
    
    We now write the Ricci curvature in the harmonic coordinates defined by the $f_i$: 
\[ 
\Delta_g g = 2\mathrm{Ric} + \partial g \ast \partial g.
\]
Since $f$ is biLipschitz, it follows that in the harmonic coordinates $\Sigma$ still has codimension greater than 2.
Next notice that the above equation holds across $\Sigma$ in the distributional sense (using \eqref{eqn:ToUseAcross} and arguing as at the end of \cref{thm:4dEinstein} below). If $\delta$ is small enough, the coefficients are $C^\alpha$ on $B_{1/2}(x)$ by \cref{lemmaholder}, and by \eqref{eqn:LpControlll} we can arrange that the right hand side is in $L^p(B_{1/2}(x))$ for some fixed $p > n$. Then it follows that $g^{ij} \in W^{2,p}$, i.e. $g^{ij} \in C^{1,\alpha}$ for some $\alpha < 1$, and we can take $\alpha$ arbitrarily close to $1$ by choosing $p$ large. This completes the proof of item (2) in the proposition. 
\end{proof}

We can now prove Theorem~\ref{thm:extension}.
\begin{proof}[Proof of \cref{thm:extension}]
    Let $\epsilon > 0$, and consider the set $X_\epsilon \subset X$ consisting of points $x$ for which $\lim_{r\to 0} r^{-n} \mathcal{H}^n(B_r(x)) > \omega_n - \epsilon$, where $\omega_n$ is the volume of the unit ball. Note that this set is open. If $x\in X_\epsilon$ for sufficiently small $\epsilon$, then after scaling up the metric (by a factor depending on $x$), the assumptions of Proposition \ref{propCinfty} hold on the rescaled ball $B_2(x)$; here we are using the standard fact that balls with volume close to the volume of Euclidean balls are GH-close to Euclidean balls, and then there are top-dimensional splitting harmonic functions with small $D$ defined as in \cref{propCinfty}. As a consequence, for a fixed given $0<\alpha<1$, up to choosing $\varepsilon$ small enough, and rescaling the metric, the harmonic coordinates on the rescaled $B_{1/2}(x)$ define the structure of a $C^{1,\alpha}$ Riemannian manifold. In particular it also follows that actually the tangent cone at $x$ is $\mathbb{R}^n$, so $x\in X^{\mathrm{reg}}$. This shows that if $\epsilon$ is chosen sufficiently small, then $X_\epsilon = X^{\mathrm{reg}}$, and so $X^{\mathrm{reg}}$ is open, and has the structure of a $C^{1,\alpha}$ Riemannian manifold for every $0<\alpha<1$, as required. 
\end{proof}

\begin{remark}\label{rem:Codimension-Sharp}
The assumption about the codimension of $S$ being greater than $2$ is sharp in \cref{thm:extension}. Indeed, one can consider the flat $\mathbb R^2$ where, for every $n\geq 1$, the metric is changed locally around $(2^{-n},0)$ so that it has a cone singularity of cone angle $2\pi(1-4^{-n})$ at $(2^{-n},0)$. Denote by $X$ the resulting space. If $\Sigma:=\{0\}\cup\{2^{-n}\}_{n\geq 1}$, $X$ is a $2$-dimensional Alexandrov space with curvature bounded from below, hence, endowed with $\mathcal{H}^2$, an $\mathrm{RCD}(-K,2)$ for some $K\geq 0$. Moreover $X\setminus \Sigma$ is smooth, and $|\mathrm{Ric}|\leq K$ on $X\setminus \Sigma$, possibly choosing a larger $K\geq 0$. Notice that $0\in X^{\mathrm{reg}}$, but $2^{-n}\notin X^{\mathrm{reg}}$, so that $X^{\mathrm{reg}}$ is not open. In fact, we cannot apply \cref{thm:extension} because $\Sigma$ has Assouad codimension $=2$.
\end{remark}

As a consequence we  prove \cref{thm:4dEinstein}.
\begin{proof}[Proof of \cref{thm:4dEinstein}]
    By \cref{thm:RCDcodim3} we have that $(M,\mathrm{d},\mathcal{H}^4)$ is an $\mathrm{RCD}(K,4)$ space. We claim that all points of $M$ are regular.  
    
    Let $p\in M$, and let $T_p$ be a tangent cone at $p$. 
    By the $L^\infty$ metric assumption and \cite{DePhilippisGigli} we have that $T_p=C(X)$ is biLipschitz equivalent to $\mathbb R^4$, where $C(X)$ is the cone over a metric space. Moreover $T_p\setminus\{o\}$ is biLipschitz equivalent to $\mathbb R^4\setminus\{0\}$, and hence $X$ is simply connected. Let $r_i\to 0$ be a sequence of radii realizing $T_p$. We have
    \[
    (M,r_i^{-1}\mathrm{d},p,S)\to (C(X),\mathrm{d}_{C(X)},o,S'),
    \]
    where $S'$ has codimension greater than $3-\frac{1}{3}$ in $C(X)$. Let $A:=\{x\in X:[1/2,2]\times \{x\}\subset S'\}$. Then $A\subset X$ is closed, and since $[1/2,2]\times A\subset S'$, $A$ has codimension greater than $3-\frac{1}{3}$ in $X$ as well. Moreover, since $M\setminus S$ is smooth and satisfies $\mathrm{Ric}=K$, we also have, using the codimension-four result in \cite{CNAnnMath}, that the convergence is smooth locally outside $S'\cup Z$, where $\mathrm{dim}_H(Z)\leq 0$. Hence $C(X)\setminus S'$ is smooth and Ricci flat outside $Z\subset C(X)$ with $\mathrm{dim}_H(Z)\leq 0$, and from this we can deduce that $X\setminus A$ is smooth and it satisfies $\mathrm{Ric}=2$. Now let us blow-up again at $x\in A$. Given a sequence $t_i\to 0$ we have, up to subsequences,
    \[
    (X,t_i^{-1}\mathrm{d}_X,x,A)\to (C(Y),\mathrm{d}_{C(Y)},o',S'').
    \]
    Notice that $S''$ has codimension greater than $3-\frac{1}{3}$ in $C(Y)$. Arguing as above we get that $C(Y)\setminus S''$ is smooth and Ricci flat, and, using that $Y$ has Hausdorff dimension 2 and the bound on the codimension of the singular set, we deduce that $Y$ must be a smooth closed surface of constant sectional curvature equal to $1$. Hence either $Y=\mathbb S^2$ or $Y=\mathbb{RP}^2$. 

    The tangent cone of $C(X)$ at $(1,x)$ is $\mathbb R\times C(Y)$ and is itself biLipschitz equivalent to $\mathbb R^4$. This forces $Y=\mathbb S^2$, and thus $C(Y)=\mathbb R^3$. Since the tangent cone $C(Y)$ was arbitrary, this implies that every point $x\in A$ is regular in $X$. Moreover, since $X\setminus A$ is smooth, we deduce that actually every point $x\in X$ is regular.
    
    Hence $X=X^{\mathrm{reg}}$ and by \cref{thm:extension} (recall that the set $A$ has codimension greater than $3-\frac{1}{3}>2$), and using Einstein equation to boot-strap the regularity of the metric (see the end of this proof for details), we have that $X$ has a smooth structure in which the metric is smooth and satisfies $\mathrm{Ric}=2$. Thus it has constant sectional curvature because we are in dimension $3$. But then, since $X$ is simply connected, $X$ must be isometric to $\mathbb S^3$. Thus every tangent cone at $p$ is isometric to $(\mathbb R^4,g_{\mathrm{eu}})$, as claimed. 
    
    As a consequence $M^{\mathrm{reg}}=M$, and thus \cref{thm:extension} applies, and we conclude that $M$ has the structure of a $C^{1,\alpha}$-Riemannian manifold. In particular, in harmonic coordinates the metric components are in $C^{1,\alpha}$. In a harmonic coordinate chart $B_1(x)$, writing the Einstein equation $\mathrm{Ric}=K$ we have
    \[ \Delta g + \partial g \ast \partial g = 2Kg, \]
    on $B_1(x) \setminus S$. 
    
    Hence, the equation satisfied by $g$ in coordinates is $\partial_a(\sqrt{\det g}g^{ab}\partial_b g_{ij})=H_{ij}$ where $g\in C^{1,\alpha}$ and $H\in C^{0,\alpha}$. Since the codimension of $S$ is $>2$ (actually, for this last argument, codimension $>1$ is enough), for every $U\Subset B_1(x)$ we can find a family of functions $\{\chi_\varepsilon\}_{0<\varepsilon<\varepsilon_0}$ such that $0\leq \chi_\varepsilon\leq 1$, $\chi_\varepsilon\equiv 1$ outside $B_{2\varepsilon}(S)$, $\chi_\varepsilon\equiv 0$ on $B_{\varepsilon}(S)\cap U$, and $\int_U |\nabla\chi_\varepsilon|\to 0$ as $\varepsilon\to 0$. Testing the equation with $\chi_\varepsilon\varphi$, for any $\varphi\in C^\infty_c(B_1(x))$, and sending $\varepsilon\to 0$, we get that $\partial_a(\sqrt{det}(g)g^{ab}\partial_b g_{ij})=H_{ij}$ holds in the distributional sense on $B_1(x)$. Thus, using classical bootstrap of regularity we get that $g\in C^\infty$, as desired.
    
\end{proof}

\begin{remark}\label{rem:ToDim5}
    The argument in the proof of \cref{thm:4dEinstein} extends to dimension $5$ (the assumption being ``Assouad codimension $>3-\frac{1}{4}$") as soon as we know that the round metric is the only Einstein metric with $\mathrm{Ric}=3$ on $\mathbb S^4$. Indeed, if the only Einstein metric with $\mathrm{Ric}=3$ on $\mathbb S^4$ is the standard one, in the reverse induction argument in the proof of \cref{thm:highdimEinstein} (see below) one can still prove $T_p=\{\mathbb R^5\}$ for every point $p\in M$. This is enough to apply \cref{thm:extension}, thus extending \cref{thm:4dEinstein} to dimension 5 as well.
\end{remark}

Before giving the proof of \cref{thm:highdimEinstein}, we need a Lemma. For the following statement, see also \cite[Theorem 4.3]{Carron}.
\begin{lemma}\label{lemEinstein}
    For every $n\geq 2$ there exists $\eta(n)$ for which the following holds. Let $g$ be a smooth metric on a closed connected manifold $X$ for which $\mathrm{Ric}_g=(n-1)g$. Assume 
    \[
    |\vol_g(X)-\vol(\mathbb S^n)|\leq \eta(n),
    \]
    where $\vol(\mathbb S^n)$ is the volume of the round $\mathbb S^n$. Then $(X,g)$ is isometric to the round $\mathbb S^n$.
\end{lemma}
\begin{proof}
    Let $g_{\mathrm{rd}}$ be the round metric on $\mathbb S^n$. Arguing by contradiction, and using Anderson's result \cite[Theorem 1.2]{Anderson90}, there are metrics $g_k$ on $\mathbb S^n$ with $\mathrm{Ric}_{g_k}=(n-1)g_k$ that are not isometric to the round metric, and such that $g_k\to g_{\mathrm{rd}}$ in the $C^{1,\alpha}$-topology. Now, by bootstrapping using the Einstein equation, we have that there are diffeomorphisms $\phi_k$ such that $\phi_k^*g_k\to g_{\mathrm{rd}}$  in the $C^{\infty}$-topology. By the results in \cite[Chapter 4 \& 12]{BesseBook}, we have that $g_{\mathrm{rd}}$ is isolated in the moduli space of Einstein metrics with $\mathrm{Ric}=n-1$ on the sphere $\mathbb S^n$. Then, for large $k$, $\phi_k^*g_k$ is isometric to $g_{\mathrm{rd}}$, and thus $(\mathbb S^n,g_k)$ is isometric to the round sphere, contradiction.
\end{proof}

\begin{proof}[Proof of \cref{thm:highdimEinstein}]
The case $n=2$ is trivial, so let us assume $n\geq 3$. 

First, by the slightly improved version of \cref{thm:RCDcodim3} discussed in \cref{rem:Better-on-Lambda}, we can choose $\varepsilon:=\varepsilon(n,a)>0$ small enough in (a) so that 
$(M^n,\mathrm{d},\mathcal{H}^n)$ is an $\mathrm{RCD}(K,n)$ space. For a continuous metric, the latter follows from \cite[Remark 2.8]{DWWW}. 

Let us fix $p\in M$. If $g$ is continuous, then by using \cite[Lemma 2.5]{AFNP} we have that the tangent cone at $p\in M$ is unique and isometric to $\mathbb R^n$. Thus the proof is concluded using \cref{thm:extension}, and the fact that the Einstein equation boot-straps the regularity from $C^{1,\alpha}$ to $C^\infty$. Let us now continue the proof under the assumptions of item (a).

If $g$ satisfies $(1-\varepsilon)h\leq g\leq (1+\varepsilon)h$, then every tangent cone $C(X)$ at $p$ satisfies 
\begin{equation}\label{eqn:VolumeControl}
|\mathcal{H}^{n-1}(X)-\vol(\mathbb S^{n-1})|\leq \Psi(\varepsilon),
\end{equation} 
where $\Psi(\varepsilon)\to 0$ as $\varepsilon \to 0$. Let $T_p$ be the family of tangent cones at $p$. We want to prove that $T_p=\{\mathbb R^n\}$. We define inductively
\[
T_p^0:=T_p,\quad \quad 
T_p^i:=\{T_x\Sigma:\Sigma=C(X)\in T_p^{i-1}, x\in X\hookrightarrow \Sigma\},\,\forall i\geq 1.
\]
Notice that if $\Sigma\in T_p^i$, then by using Gigli's splitting theorem \cite{GigliSplitting}, $\Sigma=\mathbb R^i\times C(X)$ where $X$ is $\mathrm{RCD}(n-i-2,n-i-1)$. 
We aim at proving, with a reversing induction procedure, that $T_p^i=\{\mathbb R^n\}$ for every $0\leq i\leq n-3$. We do this in two steps.
\begin{itemize}
    \item \textbf{For every $n\geq 3$, we show $T_p^{n-3}=\{\mathbb R^n\}$}. Indeed, if $\Sigma\in T_p^{n-3}$, then $\Sigma=\mathbb R^{n-3}\times C(X')$. The singular set $S$, after iterated blow-ups, converges to a singular set $S'\subset \Sigma$ with codimension greater than $2+a$. Using Cheeger--Naber \cite{CNAnnMath} we have that $\Sigma\setminus S'$ is a smooth Ricci flat manifold. Indeed, every point $x\in \Sigma\setminus S'$ must be a regular point by the codimension-four result in \cite{CNAnnMath}, and thus the convergence to the blow-up is smooth locally away from $S'$. 
    
    Since $X'$ is $2$-dimensional, the assumption on the codimension of $S'$ being greater than $2+a$, together with the fact that $\Sigma\setminus S'$ is smooth and Ricci flat, implies that $X'$ is in fact smooth and has constant sectional curvature equal to $1$. Thus $X'$ is isometric to either $\mathbb R\mathbb P^2$ or $\mathbb S^2$. The former is excluded for $\varepsilon$ small enough due to \eqref{eqn:VolumeControl}\footnote{For this step, it is enough to use that $g$ is $L^\infty$, compare the beginning of the proof of \cref{thm:4dEinstein}.}. Hence $X'\cong\mathbb S^2$ and thus $\Sigma=\mathbb R^n$, as desired.
    \item \textbf{For every $n\geq 4$, let $0\leq i \leq n-4$. We show that if $T_p^{i+1}=\{\mathbb R^n\}$, then $T_p^{i}=\{\mathbb R^n\}$}. Let $\Sigma\in T_p^i$. Then $\Sigma=\mathbb R^{i}\times C(X')$ where  $(X',\mathrm{d}_{X'},\mathcal{H}^{n-i-1}$) is an $\mathrm{RCD}(n-i-2,n-i-1)$ space. Take $x\in X'$. We aim at showing that $T_xX'\cong\mathbb R^{n-i-1}$. Indeed, $T_{(0,x)}\Sigma\cong\mathbb R^{i+1}\times T_xX'\in T_p^{i+1}$, so that, since $T_p^{i+1}=\{\mathbb R^n\}$, we deduce that $T_xX'\cong \mathbb R^{n-i-1}$. 

    In this iterated blow-up $\Sigma$ the singular set $S$ converges to a singular set $S'\subset \Sigma$ with codimension greater than $2+a$. Since we proved that every point $x\in \Sigma\setminus S'$ is regular, applying \cite{CNAnnMath} and using the assumption on the codimension as in the previous step, we get that there exists $S''\subset X'$ such that $X'\setminus S''$ is smooth with $\mathrm{Ric}=n-i-2$, and $S''$ has codimension greater than $2+a$ in $X'$. Hence, we can apply \cref{thm:extension} to $X'$, and, using the Einstein equation to boot-strap the regularity of the metric, we get that $X'$ has a smooth structure in which the metric is smooth with $\mathrm{Ric}=n-i-2$. Moreover, by \eqref{eqn:VolumeControl}, we have that, if $\varepsilon(n,a)$ is chosen small enough, 
    \[
    |\vol(X')-\vol(\mathbb S^{n-i-1})|\leq \eta(n-i-1).
    \]
    Thus by \cref{lemEinstein} we have that $X'$ is isometric to the round sphere $\mathbb S^{n-i-1}$ and thus $\Sigma\cong \mathbb R^n$. Hence $T_p^i=\{\mathbb R^n\}$, as desired.
\end{itemize}

Inductively, using the two items above, we have proved that $T_p=\{\mathbb R^n\}$. Now, to conclude the proof, we apply \cref{thm:extension} to $M$, and then use the Einstein equation to bootstrap the regularity, completing the proof.
\end{proof}

\appendix

\section{Splitting maps}\label{sec:Splitting}

We recall the following classical definition; see, e.g., \cite{CNAnnMath, ChJiNa, BPSJems, BruePasqualettoSemola, BNSInv}.
\begin{definition}\label{eqn:SplittingMap}
    Let $n\geq 2$ be an integer. Let $(X,\mathrm{d},\mathcal{H}^n)$ be an $\mathrm{RCD}(-1,n)$ metric measure space, and let $1\leq k\leq n$. Let $\delta>0$, $r>0$. We say that $u=(u_1,\ldots,u_k):B_r(x)\to \mathbb R^k$ is a {\rm $\delta$-splitting map} if there is a constant $C=C(n)$ such that the following hold:
    \begin{enumerate}
        \item\label{item1} $u_i$ is harmonic on $B_r(x)$ for every $i=1,\ldots,k$.
        \item\label{item2} $|\nabla u_i|\leq C$ for every $i=1,\ldots,k$.
        \item\label{item3} We have, for every $i,j=1,\ldots,k$,
        \[
        \fint_{B_r(x)} |\nabla u_i\cdot \nabla u_j-\delta_{ij}|\d\mathcal{H}^n \leq \delta.
        \]
        \item\label{item4} We have, for every $i=1,\ldots,k$ 
        \[
        r^2\fint_{B_r(x)} |\Hess u_i|^2\d\mathcal{H}^n \leq \delta
        \]
    \end{enumerate}
\end{definition}
The following lemma is well known; see, e.g.,
\cite[proof of Lemma 3.34, (3.42)--(3.46)]{CNAnnMath}
for the smooth setting and
\cite[proof of Lemma 4.3, (4.14)--(4.22)]{HondaPeng}
for its $\operatorname{RCD}$ implementation; we spell out its proof here for
the reader's convenience.
\begin{lemma}\label{lem:PropertiesHarmonic}
    Let $n\geq 2$ be an integer. Let $(X,\mathrm{d},\mathcal{H}^n)$ be an $\mathrm{RCD}(-K,n)$ space, with $0\leq K\leq 1$. 
    Let $u:B_3(x)\to \mathbb R$ be a harmonic function, and let 
    \[
    D:=\fint_{B_2(x)} ||\nabla u|^2-1|\mathrm{d}\mathcal{H}^n.
    \] 
    Then there exists $C:=C(n)$ such that:
    \begin{enumerate}
        \item For every $0<r<1$ we have that
        \begin{equation}\label{eqn:HessianDecay}
        r^2\fint_{B_r(x)}|\Hess u|^2\mathrm{d}\mathcal{H}^n \leq C(n)\left(Kr^2 + \fint_{B_{2r}(x)}||\nabla u|^2-1|\mathrm{d}\mathcal{H}^n\right);
        \end{equation}
        \item We have that for almost every $y\in B_1(x)$
        \begin{equation}\label{eqn:GradientEstimate}
        |\nabla u|^2(y)\leq 1+C(n)(D+K).
        \end{equation}
    \end{enumerate}
\end{lemma}
\begin{proof}
 We assume that the reader is familiar with elements of calculus in $\mathrm{RCD}$ spaces \cite{GigliNDG}. For a detailed self-contained presentation of the results we will use, the reader might read \cite[Section 2]{GigliViolo}.

 Before starting the proof, notice that $|\nabla u|^2\in W^{1,2}_{\mathrm{loc}}(B_{3}(x))$ and $\mathbf{\Delta}|\nabla u|^2$ is a locally finite measure on $B_3(x)$: see \cite[Section 2.2.1]{GigliViolo}, and \cite[Corollary 3.3.9]{GigliNDG}. All the constants $C$ in the proof might change from line to line and will only depend on $n$.
 
 Let us prove item (1). Let $\varphi_r$ be a cut-off function such that $\varphi_r\equiv 1$ on $B_r(x)$, $\varphi_r\equiv 0$ on $X\setminus B_{2r}(x)$, and $r|\nabla\varphi_r|+r^2|\Delta\varphi_r|\leq C(n)$. The latter exists due to \cite[Lemma 3.1]{MondinoNaber}, and the bound is $\leq C(n)$ because we have $0\leq K\leq 1$. Using the localized improved Bochner inequality \cite[Proposition~2.14]{GigliViolo} (based on \cite[Theorem 3.3]{Bangxian}), and Bishop--Gromov monotonicity formula we have that for every $r<3/2$
    \[
    \begin{aligned}
    2r^2\fint_{B_r(x)}&|\Hess u|^2\d\mathcal{H}^n\leq \frac{r^2}{\mathcal{H}^n(B_r(x))}\int_{X} \varphi_r\left(\d\mathbf{\Delta}(|\nabla u|^2-1) +2K|\nabla u|^2\d\mathcal{H}^n\right)  \\
    &\leq C(n)\left(r^2\fint_{B_{2r}(x)}||\nabla u|^2-1||\Delta\varphi_r|\d\mathcal{H}^n + Kr^2\fint_{B_{2r}(x)}||\nabla u|^2-1|\d\mathcal{H}^n + Kr^2\right) \\
    &\leq C(n)\left(Kr^2 + \fint_{B_{2r}(x)}||\nabla u|^2-1|\d\mathcal{H}^n\right),
    \end{aligned}
    \]
    as desired. 
    
Let us now prove item (2) directly using heat kernel estimates. A different proof using Moser iteration is possible, but we do not pursue it here. Choose $\varphi\in\mathrm{Test}_{c}(X)$ such that
\[
 0\leq\varphi\leq1,\qquad
 \varphi\equiv1\ \text{on }B_{1.1}(x),\qquad
 \supp\varphi\Subset B_{1.2}(x),\qquad
 |\nabla\varphi|+|\Delta\varphi|\leq C(n),
\]
whose existence follows from \cite[Lemma 3.1]{MondinoNaber}, and put $h:=\varphi^2$ and $\mathcal A:=\supp(|\nabla h|+|\Delta h|)$. If $y\in B_{1}(x)$, then $\dist(y,\mathcal A)\geq 0.1$. Consequently, the Gaussian and spatial-gradient heat-kernel bounds \cite[(2.12)]{GigliViolo} (based on
\cite[Theorem 1.2 and Corollary 1.2]{JiangLiZhang}), together with Bishop--Gromov comparison, and $0\leq K\leq 1$ give
\[
 p_t(y,z)+|\nabla_zp_t(y,\cdot)|(z)
 \leq\frac{C(n)\omega(t)}{\Hn(B_1(y))}
 \quad\text{for $\Hn$-a.e. }z\in\mathcal A,\quad 0<t\leq1,
\]
where $\omega(t):=t^{-(n+1)/2}e^{-c(n)/t}$, and thus $\int_0^1\omega(t)\,dt\leq C(n)$. Set $w:=|\nabla u|^2-1$. By the localized improved Bochner inequality used above we have, in the sense of measures,
\begin{equation}\label{eqn:ImprovedBochnerw}
\mathbf{\Delta}w +2K(w+1)\mathcal{H}^n\geq 0.
\end{equation}
For every Lebesgue point $y\in B_{1}(x)$ of $h w$, define
\[
 F_y(t):=\int_X h(z)w(z)p_t(y,z)\,\mathrm{d}\Hn(z).
\]
Writing $p_t=p_t(y,\cdot)$, for a.e. $t\in(0,1)$ we have using \eqref{eqn:ImprovedBochnerw} and Leibniz rule for the Laplacian,
\[
 \begin{aligned}
 F_y'(t)
 &=\int_X hp_t\,\mathrm{d}\mathbf{\Delta}w
   -\int_Xwp_t\Delta h\,\mathrm{d}\Hn
   -2\int_Xw\langle\nabla h,\nabla_zp_t\rangle\,\mathrm{d}\Hn\\
 &\geq-2K F_y(t)-2K
   -C(n)\int_{\mathcal A}|w|
          \left(p_t+|\nabla_zp_t|\right)\,\mathrm{d}\Hn.
 \end{aligned}
\]
Since by Bishop--Gromov monotonicity and $0\leq K\leq 1$ we have
\[
\int_{\mathcal{A}}|w|\d\mathcal{H}^n\leq \int_{B_{2}(x)}|w|\,\mathrm{d}\Hn
 \leq D\,\Hn(B_{2}(x)),
 \qquad
 \frac{\Hn(B_{2}(x))}{\Hn(B_1(y))}\leq C(n),
\]
we obtain
\[
 F_y'(t)+2K F_y(t)
 \geq-C(n)\left(K+D\omega(t)\right).
\]
The heat-kernel upper bound at time one also gives $|F_y(1)|\leq C(n) D$. Therefore, integrating the preceding differential inequality from $1$ to $t>0$,
\[
 F_y(t)\leq e^{2 K(1-t)}F_y(1)
 +C(n)\int_t^1e^{2K(s-t)}
                 \left(K+D\omega(s)\right)\mathrm{d}s.
\]
Finally, $F_y(t)=P_t(hw)(y)\to h(y)w(y)=w(y)$ as $t\to 0$ at $\Hn$-a.e.\! $y$. Since $0\leq K\leq1$ and $\omega\in L^1(0,1)$, it follows that
\[
 \mathop{\rm ess\,sup}_{B_{1}(x)}w
 \leq C(n)(K+D),
\]
concluding the proof of item (2).
\end{proof}

The previous \cref{lem:PropertiesHarmonic} implies that \cref{item2} and \cref{item4} in the definition of \cref{eqn:SplittingMap} are often redundant.

\begin{remark}\label{rem:HessianNotNecessary}
    Let $n\geq 2$ be an integer, and let $1\leq k\leq n$ be an integer. Let $0\leq \delta\leq 1$ and $0<r\leq 1$. Let us assume:
    \begin{enumerate}
        \item $(X,\mathrm{d},\mathcal{H}^n)$ is an $\mathrm{RCD}(-\delta,n)$ metric measure space;
        \item For every $i=1,\ldots, k$, $u_i$ is defined and harmonic on $B_{3r}(x)\subset X$,
        \item We have, for every $i,j=1,\ldots,k$,
        \[
        \fint_{B_{2r}(x)} |\nabla u_i\cdot \nabla u_j-\delta_{ij}|\d\mathcal{H}^n \leq \delta.
        \]
    \end{enumerate}
    Then, up to possibly replacing $\delta$ with $C(n)\delta$, we have item (4) in \cref{eqn:SplittingMap}. Indeed, this follows from item (1) in \cref{lem:PropertiesHarmonic}. Hence, in the setting of the present remark, item (4) in \cref{eqn:SplittingMap} (up to changing $\delta$ with $C(n)\delta$) is a consequence of item (1), and item (3) in \cref{eqn:SplittingMap}. 
    
    Finally, notice the gradient estimate in item (2) in \cref{eqn:SplittingMap} comes from item (2) in \cref{lem:PropertiesHarmonic}. Actually, in the setting of the present remark, we have the following self-improved estimate on $B_r(x)$ for every $i=1,\ldots,k$:
    \[
    |\nabla u_i|^2\leq 1+C(n)\delta.
    \]
    Hence, in the setting of the present remark, item (2) in \cref{eqn:SplittingMap} is a consequence of item (1), and item (3) in \cref{eqn:SplittingMap}. 
\end{remark}

\section{From measure bounds on the Laplacian to \texorpdfstring{$W^{1,q}$}{W1q}-bounds}

In \cref{sec:gram-defect} we used the following result, which we prove here in full detail for the reader's convenience. The following proof is an adaptation to the RCD setting of classical results on $\mathbb R^n$; see, e.g, \cite[Theorem 9.1]{Stampacchia} and \cite[Proposition 4.1 and Section 4.2]{Ponce}.
\begin{lemma}\label{lem:measure-data-rigorous}
Let $(Y,\mathrm{d},\mathcal{H}^n,y)$ be a pointed
$\operatorname{RCD}(K,n)$ space, with an integer $n\geq 2$, and $K\in\mathbb R$. Let $R>0$, and suppose that $\overline{B}_{\frac{3}{2}R}(y)\neq B_{2R}(y)$, and
\[
 f\in W^{1,2}_{\mathrm{loc}}(B_{2R}(y)),\qquad
 \mathbf{\Delta} f=\sigma,
\]
where $\sigma$ is a finite signed Radon measure on $B_{2R}(y)$.  Then,
for every $1\leq q<n/(n-1)$,
\begin{equation}\label{eq:measure-data-estimate}
 \|f\|_{W^{1,q}(B_R(y))}
 \leq C\left(\|f\|_{L^1(B_{2R}(y))}+|\sigma|(B_{2R}(y))\right).
\end{equation}
The constant $C$ depends on $q,K,n,R$ and geometric data on $B_{2R}(y)$\footnote{From the proof, one can track down precisely the constants on which this $C$ depends on: those are the constants appearing in \eqref{eq:P-Lq}, \eqref{eq:gradient-Lq}, \eqref{eqn:LinftyL1}, \eqref{eqn:Caccioppoli}, and \eqref{eqn:Ebbasta}.} (upper/lower bounds of volumes of balls of radius $\sim R$, upper/lower bounds of Ahlfors constant on $B_{2R}(y)$, Poincaré/Dirichlet inequality constants...).
\end{lemma}

\begin{proof}
Let $h$ be the harmonic replacement of $f$ in $B_{\frac{4}{3}R}(y)$:
\[
 h-f\in W^{1,2}_0(B_{\frac{4}{3}R}(y)),\qquad
 \int_{B_{\frac{4}{3}R}(y)}\langle \nabla h,\nabla\varphi\rangle\,\mathrm{d}\mathcal{H}^n=0
 \quad\forall \varphi\in W^{1,2}_0(B_{\frac{4}{3}R}(y)).
\]
Set $P=f-h\in W^{1,2}_0(B_{\frac{4}{3}R}(y))$.  Then
$\mathbf{\Delta} P=\sigma$ in $B_{\frac{4}{3}R}(y)$. 

Let $T_s(t)=\max\{-s,\min\{t,s\}\}$ and
$M:=|\sigma|(B_{2R}(y))$. If $M=0$, $f$ is harmonic and one can jump to \eqref{eqn:LinftyL1} and conclude the proof from there. Hence, we assume $M>0$.  Testing $\mathbf{\Delta} P=\sigma$ with $T_s(P)$ gives
\begin{equation}\label{eq:truncation-energy}
 \int_{\{|P|<s\}}|\nabla P|^2\,\mathrm{d}\mathcal{H}^n
 \leq sM.
\end{equation}
We record explicitly an elementary calculation.
\smallskip

\textbf{Claim.}
Let $F$ be an $\mathbb R$-valued measurable map defined on a measurable space $(\mathcal{X},\mathcal{B},\mu)$,  where $0<\mu(\mathcal{X})\leq V<\infty$.  Suppose that,
for some $A>0$, $m>0$, and $\theta>0$,
\[
 \mu(\{|F|>t\})\leq A\left(\frac mt\right)^\theta
 \qquad\text{for every }t>0.
\]
Then, for every $1\leq r<\theta$,
\begin{equation}\label{eq:weak-to-strong}
 \|F\|_{L^r}
 \leq
 \left(\frac{\theta}{\theta-r}\right)^{1/r}
 A^{1/\theta}V^{1/r-1/\theta}m.
\end{equation}

\begin{proof}[Proof of the claim]
Since $\mu(\{|F|>t\})\leq \mu(\mathcal{X})<V$, we have
\[
 \mu(\{|F|>t\})
 \leq\min\left\{V,A\left(\frac mt\right)^\theta\right\}.
\]
Let $
 t_0:=m\left(\frac AV\right)^{1/\theta}$,
so that $V=A(m/t_0)^\theta$.  The layer-cake formula gives
\begin{align*}
 \|F\|_{L^r}^r
 &=r\int_0^\infty
       t^{r-1}\mu(\{|F|>t\})\,\mathrm{d}t\leq rV\int_0^{t_0}t^{r-1}\,\mathrm{d}t
   +rAm^\theta\int_{t_0}^\infty t^{r-\theta-1}\,\mathrm{d}t\\
 &=Vt_0^r+\frac{r}{\theta-r}Am^\theta t_0^{r-\theta}=\frac{\theta}{\theta-r}
   A^{r/\theta}V^{1-r/\theta}m^r.
\end{align*}
Taking the $r$-th root proves \eqref{eq:weak-to-strong}, as desired.
\end{proof}
Let us now assume $n>2$.
Set $
 2^*:=\frac{2n}{n-2}$, $\alpha:=\frac{n}{n-2}$, $\beta:=\frac{n}{n-1}$.
Since $\overline{B}_{\frac{3}{2}R}(y)\neq B_{2R}(y)$, let $C_S>0$ be such that
\[
 \|v\|_{L^{2^*}(B_{\frac{4}{3}R}(y))}^2
 \leq C_S\int_{B_{\frac{4}{3}R}(y)}|\nabla v|^2\,\mathrm{d}\mathcal{H}^n
 \qquad\text{for every }v\in W^{1,2}_0(B_{\frac{4}{3}R}(y)).
\]
Apply the latter to $T_s(P)$.  Therefore, using \eqref{eq:truncation-energy}
\[
\begin{aligned}
 s^2\mathcal{H}^n(\{|P|\geq s\})^{2/2^*}
 &\leq \|T_s(P)\|_{L^{2^*}(B_{\frac{4}{3}R}(y))}^2\leq C_S\int_{B_{\frac{4}{3}R}(y)}|\nabla T_s(P)|^2\,\mathrm{d}\mathcal{H}^n\leq C_SsM.
\end{aligned}
\]
It follows that
\begin{equation}\label{eq:P-tail}
 \mathcal{H}^n(\{|P|\geq s\})
 \leq C_S^{\alpha}\left(\frac{M}{s}\right)^\alpha.
\end{equation}
Since $
 1\leq q<\beta<\alpha$,
the Claim above with $r=q$, yields
\begin{equation}\label{eq:P-Lq}
 \|P\|_{L^q(B_{\frac{4}{3}R}(y))}
 \leq C(\alpha,q)C_S
       V^{1/q-1/\alpha}M,
\end{equation}
where $V:=\mathcal{H}^n(B_{\frac{4}{3}R}(y))$.

Let us now deal with the gradient estimate. Let $s,t>0$.  By splitting according to whether
$|P|\geq s$ and then using
\eqref{eq:truncation-energy},
\begin{align}
 \mathcal{H}^n(\{|\nabla P|>t\})
 &\leq \mathcal{H}^n(\{|P|\geq s\})
   +t^{-2}\int_{\{|P|<s\}}|\nabla P|^2\,\mathrm{d}\mathcal{H}^n \leq C_S^{\alpha} M^\alpha s^{-\alpha}
       +Ms t^{-2}.
\label{eq:gradient-preoptimization}
\end{align}
Choose $
 s:=\left(C_S^{\alpha} M^{\alpha-1}t^2\right)^{1/(\alpha+1)}$.
The two terms on the right of
\eqref{eq:gradient-preoptimization} are then equal, and hence, doing the computations,
\begin{equation}\label{eq:gradient-tail}
 \mathcal{H}^n(\{|\nabla P|>t\})
 \leq B\left(\frac{M}{t}\right)^\gamma,
 \qquad \text{where} \qquad 
 B:=2C_S^{\alpha/(\alpha+1)},\qquad
 \gamma:=\frac{2\alpha}{\alpha+1}.
\end{equation}
Since $\alpha=n/(n-2)$, we have $
 \gamma=\frac{2n/(n-2)}{n/(n-2)+1}
       =\frac{n}{n-1}
       =\beta$.
Applying
the Claim to \eqref{eq:gradient-tail} gives,
for every $q<\beta$,
\begin{equation}\label{eq:gradient-Lq}
 \|\nabla P\|_{L^q(B_{\frac{4}{3}R}(y))}
 \leq C(\beta,q)B^{1/\beta}
       V^{1/q-1/\beta}M,
\end{equation}
where $V:=\mathcal{H}^n(B_{\frac{4}{3}R}(y))$.
Equations \eqref{eq:P-Lq} and \eqref{eq:gradient-Lq} finally give
\begin{equation}\label{eq:P-W1q-n-large}
 \|P\|_{W^{1,q}(B_{\frac{4}{3}R}(y))}\leq C M
 \qquad \forall {1\leq q<\frac{n}{n-1}}.
\end{equation}

When $n=2$, use the  embedding
$W^{1,2}_0(B_{\frac{4}{3}R}(y))\hookrightarrow L^p(B_{\frac{4}{3}R}(y))$ for every finite $p<\infty$.
Fix $1\leq q<2$.  Choose $p_0$ so large that $
 p_0>\frac{2q}{2-q}$.
The two-dimensional Dirichlet Sobolev inequality gives
\[
 \|v\|_{L^{p_0}(B_{\frac{4}{3}R}(y))}^2
 \leq C_{p_0}\int_{B_{\frac{4}{3}R}(y)}|\nabla v|^2\,\mathrm{d}\mathcal{H}^n
 \qquad\text{for every }v\in W^{1,2}_0(B_{\frac{4}{3}R}(y)).
\]
Repeating the same argument above gives $\|P\|_{W^{1,q}(B_{\frac{4}{3}R}(y))}\leq CM$ also when $n=2$. We skip the details.

Finally, $h=f-P$ is harmonic on $B_{\frac{4}{3}R}(y)$.  The $L^\infty$--$L^1$ estimate (see, e.g., \cite[Theorem A.4]{BenattiViolo}) gives
\begin{equation}\label{eqn:LinftyL1}
\|h\|_{L^\infty(B_{\frac{5}{4}R}(y))}\leq C\|h\|_{L^1(B_{\frac{4}{3}R}(y))}\leq C(\|f\|_{L^1(B_{2R}(y))}+M).
\end{equation}
Moreover the Caccioppoli inequality, and \eqref{eqn:LinftyL1} give (absorbing $R^{-2}$ in the constant $C$)
\begin{equation}\label{eqn:Caccioppoli}
\|\nabla h\|_{L^2(B_{R}(y))}^2\leq CR^{-2}\|h\|_{L^2(B_{\frac{5}{4}R}(y))}^2\leq C\|h\|_{L^1(B_{\frac{4}{3}R}(y))}^2.
\end{equation}
Thus, since $q<\frac{n}{n-1}\leq 2$ we have, using \eqref{eqn:LinftyL1}, and \eqref{eqn:Caccioppoli}:
\begin{equation}\label{eqn:Ebbasta}
\|\nabla h\|_{L^q(B_{R}(y))}\leq C\|\nabla h\|_{L^2(B_{R}(y))}\leq C\|h\|_{L^1(B_{\frac{4}{3}R}(y))}\leq   C(\|f\|_{L^1(B_{2R}(y))}+M).
\end{equation}
Combining \eqref{eqn:LinftyL1} and \eqref{eqn:Ebbasta} we get
\begin{equation}\label{eqn:Ebbasta2}
    \|h\|_{W^{1,q}(B_R(y))}\leq C(\|f\|_{L^1(B_{2R}(y))}+M).
\end{equation}
Finally, combining \eqref{eq:P-W1q-n-large}, \eqref{eqn:Ebbasta2}, and using the fact that $f=h+P$, we get the sought conclusion.
\end{proof}

\printbibliography

@article {CampanatoPI,
    AUTHOR = {G\'orka, P.},
     TITLE = {Campanato theorem on metric measure spaces},
   JOURNAL = {Ann. Acad. Sci. Fenn. Math.},
  FJOURNAL = {Annales Academi\ae\ Scientiarum Fennic\ae. Mathematica},
    VOLUME = {34},
      YEAR = {2009},
    NUMBER = {2},
     PAGES = {523--528},
      ISSN = {1239-629X,1798-2383},
   MRCLASS = {46E35 (26B35 28C99)},
  MRNUMBER = {2553810},
MRREVIEWER = {Jeremy\ T.\ Tyson},
       NOTE = {\url{https://afm.journal.fi/article/view/135282}},
}

@article {ColdingNaberReifenberg,
    AUTHOR = {Colding, T. H. and Naber, A.},
     TITLE = {Lower {R}icci curvature, branching and the bilipschitz
              structure of uniform {R}eifenberg spaces},
   JOURNAL = {Adv. Math.},
  FJOURNAL = {Advances in Mathematics},
    VOLUME = {249},
      YEAR = {2013},
     PAGES = {348--358},
      ISSN = {0001-8708,1090-2082},
  MRNUMBER = {3116575},
       DOI = {10.1016/j.aim.2013.09.005},
       URL = {https://doi.org/10.1016/j.aim.2013.09.005},
}

@article {JiangShengZhang,
    AUTHOR = {Jiang, W. and Sheng, W. and Zhang, H.},
     TITLE = {Removable singularity of positive mass theorem with continuous
              metrics},
   JOURNAL = {Math. Z.},
  FJOURNAL = {Mathematische Zeitschrift},
    VOLUME = {302},
      YEAR = {2022},
    NUMBER = {2},
     PAGES = {839--874},
      ISSN = {0025-5874,1432-1823},
   MRCLASS = {53C20},
  MRNUMBER = {4480213},
MRREVIEWER = {David\ J.\ Wraith},
       DOI = {10.1007/s00209-022-03081-w},
       URL = {https://doi.org/10.1007/s00209-022-03081-w},
}

@book {HKST15,
    AUTHOR = {Heinonen, J. and Koskela, P. and Shanmugalingam, N. and
              Tyson, J. T.},
     TITLE = {Sobolev spaces on metric measure spaces},
    SERIES = {New Mathematical Monographs},
    VOLUME = {27},
      NOTE = {An approach based on upper gradients},
 PUBLISHER = {Cambridge University Press, Cambridge},
      YEAR = {2015},
     PAGES = {xii+434},
      ISBN = {978-1-107-09234-1},
  MRNUMBER = {3363168},
       DOI = {10.1017/CBO9781316135914},
       URL = {https://doi.org/10.1017/CBO9781316135914},
}

@article {HK00,
    AUTHOR = {Haj\l asz, P. and Koskela, P.},
     TITLE = {Sobolev met {P}oincar\'e},
   JOURNAL = {Mem. Amer. Math. Soc.},
  FJOURNAL = {Memoirs of the American Mathematical Society},
    VOLUME = {145},
      YEAR = {2000},
    NUMBER = {688},
     PAGES = {x+101},
      ISSN = {0065-9266,1947-6221},
   MRCLASS = {46E35 (30C65 31C25 53C17 58J60)},
  MRNUMBER = {1683160},
MRREVIEWER = {Alexander\ D.\ Ukhlov},
       DOI = {10.1090/memo/0688},
       URL = {https://doi.org/10.1090/memo/0688},
}

@article {UhlenbeckYangMills,
    AUTHOR = {Uhlenbeck, K. K.},
     TITLE = {Removable singularities in {Y}ang-{M}ills fields},
   JOURNAL = {Comm. Math. Phys.},
  FJOURNAL = {Communications in Mathematical Physics},
    VOLUME = {83},
      YEAR = {1982},
    NUMBER = {1},
     PAGES = {11--29},
      ISSN = {0010-3616,1432-0916},
   MRCLASS = {53C05 (58E20 81E10)},
  MRNUMBER = {648355},
MRREVIEWER = {Wolfgang\ L\"ucke},
       DOI = {10.1007/BF01947068},
       URL = {https://doi.org/10.1007/BF01947068},
}

@article {ShiTamPacific,
    AUTHOR = {Shi, Y. and Tam, L.-F.},
     TITLE = {Scalar curvature and singular metrics},
   JOURNAL = {Pacific J. Math.},
  FJOURNAL = {Pacific Journal of Mathematics},
    VOLUME = {293},
      YEAR = {2018},
    NUMBER = {2},
     PAGES = {427--470},
      ISSN = {0030-8730,1945-5844},
   MRCLASS = {53C20 (53A30 53C25)},
  MRNUMBER = {3730743},
MRREVIEWER = {Wei\ Yuan},
       DOI = {10.2140/pjm.2018.293.427},
       URL = {https://doi.org/10.2140/pjm.2018.293.427},
}

@article {JiangLiZhang,
    AUTHOR = {Jiang, R. and Li, H. and Zhang, H.},
     TITLE = {Heat kernel bounds on metric measure spaces and some
              applications},
   JOURNAL = {Potential Anal.},
  FJOURNAL = {Potential Analysis. An International Journal Devoted to the
              Interactions between Potential Theory, Probability Theory,
              Geometry and Functional Analysis},
    VOLUME = {44},
      YEAR = {2016},
    NUMBER = {3},
     PAGES = {601--627},
      ISSN = {0926-2601,1572-929X},
   MRCLASS = {53C23 (35K05 35K08 42B20 47B06 58J35)},
  MRNUMBER = {3489857},
MRREVIEWER = {Martin\ Kell},
       DOI = {10.1007/s11118-015-9521-2},
       URL = {https://doi.org/10.1007/s11118-015-9521-2},
}

@article {CNAnnMath,
    AUTHOR = {Cheeger, J. and Naber, A.},
     TITLE = {Regularity of {E}instein manifolds and the codimension 4
              conjecture},
   JOURNAL = {Ann. of Math. (2)},
  FJOURNAL = {Annals of Mathematics. Second Series},
    VOLUME = {182},
      YEAR = {2015},
    NUMBER = {3},
     PAGES = {1093--1165},
      ISSN = {0003-486X,1939-8980},
   MRCLASS = {53C25 (53C23)},
  MRNUMBER = {3418535},
MRREVIEWER = {Luis\ Guijarro},
       DOI = {10.4007/annals.2015.182.3.5},
       URL = {https://doi.org/10.4007/annals.2015.182.3.5},
}

@article {BPSJems,
    AUTHOR = {Bru\`e, E. and Pasqualetto, E. and Semola, D.},
     TITLE = {Rectifiability of the reduced boundary for sets of finite
              perimeter over {${\rm RCD}(K,N)$} spaces},
   JOURNAL = {J. Eur. Math. Soc. (JEMS)},
  FJOURNAL = {Journal of the European Mathematical Society (JEMS)},
    VOLUME = {25},
      YEAR = {2023},
    NUMBER = {2},
     PAGES = {413--465},
      ISSN = {1435-9855,1435-9863},
   MRCLASS = {49Q15 (26B30 49Q20 53C23)},
  MRNUMBER = {4556787},
       DOI = {10.4171/jems/1217},
       URL = {https://doi.org/10.4171/jems/1217},
}

@article{GigliSplitting,
  author  = {Gigli, N.},
  title   = {The splitting theorem in non-smooth context},
  journal = {Mem. Amer. Math. Soc.},
  volume  = {317},
  year    = {2026},
  number  = {1609},
  doi     = {10.1090/memo/1609},
}

@article {Bangxian,
    AUTHOR = {Han, B.-X.},
     TITLE = {Ricci tensor on {${\rm RCD}^*(K,N)$} spaces},
   JOURNAL = {J. Geom. Anal.},
  FJOURNAL = {Journal of Geometric Analysis},
    VOLUME = {28},
      YEAR = {2018},
    NUMBER = {2},
     PAGES = {1295--1314},
      ISSN = {1050-6926,1559-002X},
   MRCLASS = {53C23 (30L05 31E05 49J52 51F99 53C21)},
  MRNUMBER = {3790501},
MRREVIEWER = {Fernando\ Galaz-Garc\'ia},
       DOI = {10.1007/s12220-017-9863-7},
       URL = {https://doi.org/10.1007/s12220-017-9863-7},
}

@article {BNSInv,
    AUTHOR = {Bru\`e, E. and Naber, A. and Semola, D.},
     TITLE = {Boundary regularity and stability for spaces with {R}icci
              bounded below},
   JOURNAL = {Invent. Math.},
  FJOURNAL = {Inventiones Mathematicae},
    VOLUME = {228},
      YEAR = {2022},
    NUMBER = {2},
     PAGES = {777--891},
      ISSN = {0020-9910,1432-1297},
   MRCLASS = {53C21 (53C23)},
  MRNUMBER = {4411732},
MRREVIEWER = {Emil\ Saucan},
       DOI = {10.1007/s00222-021-01092-8},
       URL = {https://doi.org/10.1007/s00222-021-01092-8},
}

@article{Lin1999,
  author        = {Lin, F.-H.},
  title         = {Gradient estimates and blow-up analysis for
                   stationary harmonic maps},
  journal       = {Annals of Mathematics. Second Series},
  volume        = {149},
  number        = {3},
  year          = {1999},
  pages         = {785--829},
  doi           = {10.2307/121073},
  url           = {https://doi.org/10.2307/121073},
}

@incollection{BairdEells,
  author    = {Baird, P. and Eells, J.},
  title     = {A conservation law for harmonic maps},
  editor    = {Looijenga, Eduard and Siersma, Dirk and Takens, Floris},
  booktitle = {Geometry Symposium, Utrecht 1980},
  series    = {Lecture Notes in Mathematics},
  volume    = {894},
  publisher = {Springer-Verlag},
  address   = {Berlin--New York},
  year      = {1981},
  pages     = {1--25},
  doi       = {10.1007/BFb0096222},
  url       = {https://doi.org/10.1007/BFb0096222}
}

@article{Rajala,
  author  = {Rajala, T.},
  title   = {Local {Poincar{\'e}} inequalities from stable curvature
             conditions on metric spaces},
  journal = {Calculus of Variations and Partial Differential Equations},
  volume  = {44},
  number  = {3--4},
  year    = {2012},
  pages   = {477--494},
  doi     = {10.1007/s00526-011-0442-7},
  url     = {https://doi.org/10.1007/s00526-011-0442-7}
}

@article{GigliNDG,
  author  = {Gigli, N.},
  title   = {Nonsmooth differential geometry---an approach tailored
             for spaces with {Ricci} curvature bounded from below},
  journal = {Memoirs of the American Mathematical Society},
  volume  = {251},
  number  = {1196},
  year    = {2018},
  note    = {v+161 pp.},
  doi     = {10.1090/memo/1196},
  url     = {https://doi.org/10.1090/memo/1196}
}

@article{GigliViolo,
  author  = {Gigli, N. and Violo, I. Y.},
  title   = {Monotonicity formulas for harmonic functions in\\
             {${\rm RCD}(0,N)$} spaces},
  journal = {The Journal of Geometric Analysis},
  volume  = {33},
  number  = {3},
  year    = {2023},
  note    = {Article No.~100, 89 pp.},
  doi     = {10.1007/s12220-022-01131-7},
  url     = {https://doi.org/10.1007/s12220-022-01131-7}
}

@article{DePhilippisGigli,
  author  = {{De Philippis}, G. and Gigli, N.},
  title   = {Non-collapsed spaces with {Ricci} curvature bounded
             from below},
  journal = {Journal de l'{\'E}cole polytechnique---Math{\'e}matiques},
  volume  = {5},
  year    = {2018},
  pages   = {613--650},
  doi     = {10.5802/jep.80},
  url     = {https://doi.org/10.5802/jep.80}
}

@article {Anderson90,
    AUTHOR = {Anderson, M. T.},
     TITLE = {Convergence and rigidity of manifolds under {R}icci curvature
              bounds},
   JOURNAL = {Invent. Math.},
  FJOURNAL = {Inventiones Mathematicae},
    VOLUME = {102},
      YEAR = {1990},
    NUMBER = {2},
     PAGES = {429--445},
      ISSN = {0020-9910,1432-1297},
   MRCLASS = {53C23 (53C21 58D27)},
  MRNUMBER = {1074481},
MRREVIEWER = {Gudlaugur\ Thorbergsson},
       DOI = {10.1007/BF01233434},
       URL = {https://doi.org/10.1007/BF01233434},
}

@book {BesseBook,
    AUTHOR = {Besse, A. L.},
     TITLE = {Einstein manifolds},
    SERIES = {Ergebnisse der Mathematik und ihrer Grenzgebiete (3) [Results
              in Mathematics and Related Areas (3)]},
    VOLUME = {10},
 PUBLISHER = {Springer-Verlag, Berlin},
      YEAR = {1987},
     PAGES = {xii+510},
      ISBN = {3-540-15279-2},
   MRCLASS = {53C25 (53-02 53C21 53C30 53C55 58D17 58E11)},
  MRNUMBER = {867684},
MRREVIEWER = {S.\ M.\ Salamon},
       DOI = {10.1007/978-3-540-74311-8},
       URL = {https://doi.org/10.1007/978-3-540-74311-8},
}

@article{BenattiViolo,
  author  = {Benatti, L. and Violo, I. Y.},
  title   = {Second-order estimates for the {$p$}-Laplacian in
             {${\rm RCD}$} spaces},
  journal = {Journal of Differential Equations},
  volume  = {439},
  year    = {2025},
  note    = {Article No.~113398},
  doi     = {10.1016/j.jde.2025.113398},
  url     = {https://doi.org/10.1016/j.jde.2025.113398}
}

@article{AmbrosioHondaSpectral,
  author  = {Ambrosio, L. and Honda, S.},
  title   = {Local spectral convergence in
             {${\rm RCD}^*(K,N)$} spaces},
  journal = {Nonlinear Analysis},
  volume  = {177},
  year    = {2018},
  pages   = {1--23},
  note    = {Part A},
  doi     = {10.1016/j.na.2017.04.003},
  url     = {https://doi.org/10.1016/j.na.2017.04.003}
}

@article{LeeTam,
	author = {Lee, M.-C. and Tam, L.-F.},
	doi = {10.1090/tran/9379},
	issn = {1088-6850},
	journal = {Transactions of the American Mathematical Society},
	month = jan,
	number = {3},
	pages = {1531--1550},
	publisher = {American Mathematical Society (AMS)},
	title = {Continuous metrics and a conjecture of Schoen},
	url = {http://dx.doi.org/10.1090/tran/9379},
	volume = {378},
	year = {2025}}

@article{LiMantoulidis,
	author = {Li, C. and Mantoulidis, C.},
	doi = {10.1007/s00208-018-1753-1},
	issn = {1432-1807},
	journal = {Mathematische Annalen},
	month = sep,
	number = {1--2},
	pages = {99--131},
	publisher = {Springer Science and Business Media LLC},
	title = {Positive scalar curvature with skeleton singularities},
	url = {http://dx.doi.org/10.1007/s00208-018-1753-1},
	volume = {374},
	year = {2018}}

@article {LittmanStampacchiaWeinberger,
    AUTHOR = {Littman, W. and Stampacchia, G. and Weinberger, H. F.},
     TITLE = {Regular points for elliptic equations with discontinuous
              coefficients},
   JOURNAL = {Ann. Scuola Norm. Sup. Pisa Cl. Sci. (3)},
  FJOURNAL = {Annali della Scuola Normale Superiore di Pisa. Classe di
              Scienze. Serie III},
    VOLUME = {17},
      YEAR = {1963},
     PAGES = {43--77},
      ISSN = {0391-173X},
   MRCLASS = {35.42},
  MRNUMBER = {161019},
MRREVIEWER = {C.\ B.\ Morrey, Jr.},
}

@article{Bohm,
	author = {B{\"o}hm, C.},
	doi = {10.1007/s002220050261},
	issn = {1432-1297},
	journal = {Inventiones Mathematicae},
	month = sep,
	number = {1},
	pages = {145--176},
	publisher = {Springer Science and Business Media LLC},
	title = {Inhomogeneous Einstein metrics on low-dimensional spheres and other low-dimensional spaces},
	url = {http://dx.doi.org/10.1007/s002220050261},
	volume = {134},
	year = {1998}}

@article{BandoKasueNakajima,
	author = {Bando, S. and Kasue, A. and Nakajima, H.},
	doi = {10.1007/bf01389045},
	issn = {1432-1297},
	journal = {Inventiones Mathematicae},
	month = jun,
	number = {2},
	pages = {313--349},
	publisher = {Springer Science and Business Media LLC},
	title = {On a construction of coordinates at infinity on manifolds with fast curvature decay and maximal volume growth},
	url = {http://dx.doi.org/10.1007/BF01389045},
	volume = {97},
	year = {1989}}

@unpublished {CecchiniFrenckZeidler,
    AUTHOR = {Cecchini, S. and Frenck, G. and Zeidler, R.},
     TITLE = {Positive scalar curvature with point singularities},
      YEAR = {2024},
      NOTE = {To appear in Duke Math. J.; preprint arXiv:2407.20163},
       URL = {https://arxiv.org/abs/2407.20163},
}

@article{SmithYang,
	author = {Smith, P. D. and Yang, D.},
	doi = {10.1090/s0002-9947-1992-1052910-2},
	issn = {1088-6850},
	journal = {Transactions of the American Mathematical Society},
	number = {1},
	pages = {203--219},
	publisher = {American Mathematical Society (AMS)},
	title = {Removing point singularities of Riemannian manifolds},
	url = {http://dx.doi.org/10.1090/S0002-9947-1992-1052910-2},
	volume = {333},
	year = {1992}}

@article {AFNP,
    AUTHOR = {Antonelli, G. and Fogagnolo, M. and Nardulli,
              S. and Pozzetta, M.},
     TITLE = {Positive mass and isoperimetry for continuous metrics with
              nonnegative scalar curvature},
   JOURNAL = {Ann. Inst. H. Poincar\'e C Anal. Non Lin\'eaire},
  FJOURNAL = {Annales de l'Institut Henri Poincar\'e C. Analyse Non
              Lin\'eaire},
      YEAR = {2026},
      ISSN = {0294-1449,1873-1430},
       DOI = {10.4171/AIHPC/178},
       URL = {https://doi.org/10.4171/AIHPC/178},
}

@article{Kazaras,
	author = {Kazaras, D.},
	doi = {10.1007/s00208-024-02829-5},
	issn = {1432-1807},
	journal = {Mathematische Annalen},
	month = apr,
	number = {4},
	pages = {4951--4972},
	publisher = {Springer Science and Business Media LLC},
	title = {Desingularizing positive scalar curvature 4-manifolds},
	url = {http://dx.doi.org/10.1007/s00208-024-02829-5},
	volume = {390},
	year = {2024}}

@article {ChJiNa,
    AUTHOR = {Cheeger, J. and Jiang, W. and Naber, A.},
     TITLE = {Rectifiability of singular sets of noncollapsed limit spaces
              with {R}icci curvature bounded below},
   JOURNAL = {Ann. of Math. (2)},
  FJOURNAL = {Annals of Mathematics. Second Series},
    VOLUME = {193},
      YEAR = {2021},
    NUMBER = {2},
     PAGES = {407--538},
      ISSN = {0003-486X,1939-8980},
   MRCLASS = {53B20 (35A21 53C23)},
  MRNUMBER = {4226910},
MRREVIEWER = {Daniele\ Semola},
       DOI = {10.4007/annals.2021.193.2.2},
       URL = {https://doi.org/10.4007/annals.2021.193.2.2},
}

@article {ChCo0,
    AUTHOR = {Cheeger, J. and Colding, T. H.},
     TITLE = {Lower bounds on {R}icci curvature and the almost rigidity of
              warped products},
   JOURNAL = {Ann. of Math. (2)},
  FJOURNAL = {Annals of Mathematics. Second Series},
    VOLUME = {144},
      YEAR = {1996},
    NUMBER = {1},
     PAGES = {189--237},
      ISSN = {0003-486X,1939-8980},
   MRCLASS = {53C21 (53C20 53C23)},
  MRNUMBER = {1405949},
MRREVIEWER = {Joseph\ E.\ Borzellino},
       DOI = {10.2307/2118589},
       URL = {https://doi.org/10.2307/2118589},
}

@article{BruePasqualettoSemola,
  author  = {Bru{\`e}, E. and Pasqualetto, E. and
             Semola, D.},
  title   = {Rectifiability of {${\rm RCD}(K,N)$} spaces via
             {$\delta$}-splitting maps},
  journal = {Annales Fennici Mathematici},
  volume  = {46},
  number  = {1},
  year    = {2021},
  pages   = {465--482},
  doi     = {10.5186/aasfm.2021.4627},
  url     = {https://doi.org/10.5186/aasfm.2021.4627}
}

@misc{HondaZhang,
author = {Honda, S. and Zhang, R.},
archivePrefix = {arXiv},
doi = {10.48550/arXiv.2608.16021},
eprint = {2608.16021},
month = aug,
note = {Preprint, arXiv:2608.16021v1},
primaryClass = {math.DG},
title = {Regularity and structure of spaces with synthetic {Ricci} bounds and positive injectivity radius},
url = {https://arxiv.org/abs/2608.16021},
year = {2026}
}

@article{BertrandKettererMondelloRichard,
author = {Bertrand, J. and Ketterer, C. and Mondello, I. and Richard, T.},
doi = {10.5802/aif.3393},
issn = {1777-5310},
journal = {Annales de l'Institut Fourier},
month = jun,
number = {1},
pages = {123--173},
publisher = {Association des Annales de l'Institut Fourier},
title = {Stratified spaces and synthetic {Ricci} curvature bounds},
url = {https://doi.org/10.5802/aif.3393},
volume = {71},
year = {2021}
}

@article{Carron,
author = {Carron, G.},
doi = {10.2422/2036-2145.201201_002},
issn = {2036-2145},
journal = {Annali della Scuola Normale Superiore di Pisa, Classe di Scienze. Serie V},
month = dec,
number = {4},
pages = {1091--1113},
publisher = {Edizioni della Normale},
title = {Some old and new results about rigidity of critical metric},
url = {https://doi.org/10.2422/2036-2145.201201_002},
volume = {13},
year = {2014}
}

@article{ByunWang,
author = {Byun, S.-S. and Wang, L.},
doi = {10.1002/cpa.20037},
issn = {1097-0312},
journal = {Communications on Pure and Applied Mathematics},
month = oct,
number = {10},
pages = {1283--1310},
publisher = {Wiley},
title = {Elliptic equations with {BMO} coefficients in {Reifenberg} domains},
url = {https://doi.org/10.1002/cpa.20037},
volume = {57},
year = {2004}
}

@article{BurkhardtGuim,
author = {Burkhardt-Guim, P.},
doi = {10.1007/s00208-026-03489-3},
issn = {1432-1807},
journal = {Mathematische Annalen},
month = may,
number = {3},
pages = {57},
publisher = {Springer},
title = {Smoothing {$L^\infty$} {Riemannian} metrics with nonnegative scalar curvature outside of a singular set},
url = {https://doi.org/10.1007/s00208-026-03489-3},
volume = {395},
year = {2026}
}

@article{ChengYau,
	author = {Cheng, S. Y. and Yau, S. T.},
	doi = {10.1002/cpa.3160280303},
	issn = {1097-0312},
	journal = {Communications on Pure and Applied Mathematics},
	month = may,
	number = {3},
	pages = {333--354},
	publisher = {Wiley},
	title = {Differential equations on riemannian manifolds and their geometric applications},
	url = {http://dx.doi.org/10.1002/cpa.3160280303},
	volume = {28},
	year = {1975}}

@unpublished {CCHSTT,
    AUTHOR = {Chen, Y. and Chiu, S.-K. and Hallgren, M. and
              Sz{\'e}kelyhidi, G. and T{\^o}, T. D. and Tong,
              F.},
     TITLE = {On {K}\"ahler-{E}instein currents},
      YEAR = {2025},
      NOTE = {Preprint arXiv:2502.09825},
       URL = {https://arxiv.org/abs/2502.09825},
}

@article {HondaSun,
    AUTHOR = {Honda, S. and Sun, S.},
     TITLE = {From almost smooth spaces to {RCD} spaces},
   JOURNAL = {Calc. Var. Partial Differential Equations},
  FJOURNAL = {Calculus of Variations and Partial Differential Equations},
    VOLUME = {65},
      YEAR = {2026},
    NUMBER = {4},
     PAGES = {Paper No. 131, 45},
      ISSN = {0944-2669,1432-0835},
   MRCLASS = {53C21 (53C23)},
  MRNUMBER = {5045151},
       DOI = {10.1007/s00526-026-03297-2},
       URL = {https://doi.org/10.1007/s00526-026-03297-2},
}

@article{HondaPeng,
  author  = {Honda, S. and Peng, Y.},
  title   = {A note on the topological stability theorem from
             {${\rm RCD}$} spaces to {Riemannian} manifolds},
  journal = {Manuscripta Mathematica},
  volume  = {172},
  number  = {3--4},
  year    = {2023},
  pages   = {971--1007},
  doi     = {10.1007/s00229-022-01418-7},
  url     = {https://doi.org/10.1007/s00229-022-01418-7}
}

@article{AHPT,
  author  = {Ambrosio, L. and Honda, S. and
             Portegies, J. W. and Tewodrose, D.},
  title   = {Embedding of\\ {${\rm RCD}^*(K,N)$} spaces in {$L^2$}
             via eigenfunctions},
  journal = {Journal of Functional Analysis},
  volume  = {280},
  number  = {10},
  year    = {2021},
  note    = {Article No.~108968, 72 pp.},
  doi     = {10.1016/j.jfa.2021.108968},
  url     = {https://doi.org/10.1016/j.jfa.2021.108968}
}

@incollection{AmbrosioHonda,
  author    = {Ambrosio, L. and Honda, S.},
  title     = {New stability results for sequences of metric measure
               spaces with uniform {Ricci} bounds from below},
  editor    = {Gigli, Nicola},
  booktitle = {Measure Theory in Non-Smooth Spaces},
  series    = {Partial Differential Equations and Measure Theory},
  publisher = {De Gruyter Open},
  address   = {Warsaw},
  year      = {2017},
  pages     = {1--51},
  isbn      = {978-3-11-055083-2},
  doi       = {10.1515/9783110550832-001},
  url       = {https://doi.org/10.1515/9783110550832-001}
}

@article{MondinoNaber,
  author  = {Mondino, A. and Naber, A.},
  title   = {Structure theory of metric-measure spaces with lower
             {Ricci} curvature bounds},
  journal = {Journal of the European Mathematical Society},
  volume  = {21},
  number  = {6},
  year    = {2019},
  pages   = {1809--1854},
  doi     = {10.4171/JEMS/874},
  url     = {https://doi.org/10.4171/JEMS/874}
}

@article{HondaLp,
  author  = {Honda, S.},
  title   = {Ricci curvature and {$L^p$}-convergence},
  journal = {Journal f{\"u}r die reine und angewandte Mathematik},
  volume  = {705},
  year    = {2015},
  pages   = {85--154},
  doi     = {10.1515/crelle-2013-0061},
  url     = {https://doi.org/10.1515/crelle-2013-0061}
}

@article {Honda,
    AUTHOR = {Honda, S.},
     TITLE = {Bakry-\'Emery conditions on almost smooth metric measure
              spaces},
   JOURNAL = {Anal. Geom. Metr. Spaces},
  FJOURNAL = {Analysis and Geometry in Metric Spaces},
    VOLUME = {6},
      YEAR = {2018},
    NUMBER = {1},
     PAGES = {129--145},
      ISSN = {2299-3274},
   MRCLASS = {53C21 (53C23)},
  MRNUMBER = {3877323},
MRREVIEWER = {Luca\ Rizzi},
       DOI = {10.1515/agms-2018-0007},
       URL = {https://doi.org/10.1515/agms-2018-0007},
}

@unpublished {WWX26,
    AUTHOR = {Wang, J. and Wang, J. and Xie, Z.},
     TITLE = {{$L^\infty$}-metrics on tori and {S}choen's conjecture},
      YEAR = {2026},
      NOTE = {Preprint arXiv:2606.21325},
       URL = {https://arxiv.org/abs/2606.21325},
}

@article {Gabor,
    AUTHOR = {Sz{\'e}kelyhidi, G.},
     TITLE = {Singular {K}\"ahler--{E}instein metrics and {RCD} spaces},
   JOURNAL = {Forum Math. Pi},
  FJOURNAL = {Forum of Mathematics, Pi},
    VOLUME = {13},
      YEAR = {2025},
     PAGES = {Paper No. e24, 33},
      ISSN = {2050-5086},
   MRCLASS = {53C25 (53C21)},
       DOI = {10.1017/fmp.2025.10015},
       URL = {https://doi.org/10.1017/fmp.2025.10015},
}

@unpublished {KWW26,
    AUTHOR = {Khuri, M. and Wang, J. and Wang, J.},
     TITLE = {Riemannian positive mass theorem in all dimensions in the
              presence of low-codimension singularities},
      YEAR = {2026},
      NOTE = {Preprint arXiv:2606.23529},
       URL = {https://arxiv.org/abs/2606.23529},
}

@unpublished {BZ26,
    AUTHOR = {Bi, Y. and Zhu, J.},
     TITLE = {Positive scalar curvature obstructions via singular dimension
              descent},
      YEAR = {2026},
      NOTE = {Preprint arXiv:2606.20528},
       URL = {https://arxiv.org/abs/2606.20528},
}

@article{Stampacchia,
  author   = {Stampacchia, G.},
  title    = {Le probl{\`e}me de {Dirichlet} pour les {\'e}quations
              elliptiques du second ordre {\`a} coefficients discontinus},
  journal  = {Annales de l'Institut Fourier},
  volume   = {15},
  number   = {1},
  pages    = {189--257},
  year     = {1965},
  doi      = {10.5802/aif.204},
  url      = {https://doi.org/10.5802/aif.204},
  langid   = {french}
}

@misc{Ponce,
  author        = {Ponce, A. C.},
  title         = {Selected Problems on Elliptic Equations Involving Measures},
  year          = {2012},
  eprint        = {1204.0668},
  archivePrefix = {arXiv},
  primaryClass  = {math.AP},
  note          = {Revised version, arXiv:1204.0668v3, 16 May 2017},
}

@article {DWWW,
    AUTHOR = {Dai, X. and Wang, C. and Wang, L. and Wei, G.},
     TITLE = {Singular metrics with non-negative scalar curvature and {RCD}},
   JOURNAL = {Commun. Contemp. Math.},
  FJOURNAL = {Communications in Contemporary Mathematics},
      YEAR = {2026},
     PAGES = {Paper No. 2650034},
       DOI = {10.1142/S0219199726500343},
       URL = {https://doi.org/10.1142/S0219199726500343},
}

\end{document}